\documentclass[12pt]{article}
\usepackage{amssymb,latexsym, amsmath, amsthm, amsfonts, amssymb, enumerate, paralist,cite,extpfeil}
\usepackage{cases}
\usepackage[title,titletoc]{appendix}
\usepackage{yfonts,eufrak}
\usepackage[usenames,dvipsnames,svgnames,x11names,table]{xcolor}
\usepackage{color}
\usepackage{oplotsymbl}
\usepackage{graphicx,overpic}
\usepackage{hyperref}
\hypersetup{
	CJKbookmarks=true
	colorlinks=true,
	linkcolor=blue,
	filecolor=blue,
	urlcolor=blue,
	citecolor=blue,
	bookmarksnumbered=true
}

\usepackage{geometry}
\def\nzz{\zeta_0}

\def\rr{{\mathbb R}}

\def\nn{{\mathbb N}}
\def\zz{{\mathbb Z}}
\def\cc{{\mathbb C}}

\def\cw{{\mathcal W}}
\def\az{\alpha}

\def\dz{\delta}

\def\kz{\kappa}
\def\lz{\lambda}
\def\Lz{\Lambda}

\def\Oz{\Omega}
\def\tz{\theta}
\def\vtz{\vartheta}
\def\atz{|\tz|}

\def\gz{\gamma}

\def\C{\mathbb C}

\def\G{\mathbb G}
\def\H{\mathbb H}

\def\He{\mathrm{Hess}}
\def\I{{\mathbb{I}}}
\def\N{\mathbb N}

\def\OC{\Omega_{-, 4}}

\def\pp{\mathcal{Q}}
\def\R{\mathbb R}
\def\RS{{\mathbb R}^2_{<, +}}

\def\T{{\mathrm{T}}}
\def\u{\mathbf{U}}

\def\w{\mathrm{w}}
\def\X{{\mathrm{X}}}

\def\Z{\mathbb Z}
\def\ep{\epsilon}
\def\BX{{\mathbb{X}}}
\def\ggg{\widetilde{\mathrm{g}}}

\def\ss{\mathbf{\overline s}}
\def\ww{\mathbf{\overline w}}

\def\cc{\mathbf{C}}

\def\pl{\mathrm{p}}

\def\rF{\mathbf{F}}
\def\Ft{\mathbf{F}_2}
\def\rP{\mathbf{P}}
\def\cP{\mathcal{P}}

\def\rD{\mathbf{D}}
\def\cD{\mathcal{D}}
\def\cDp{\mathcal{D}_\pl}

\def\K{\mathrm{K}}
\def\cZ{\mathcal{Z}}
\def\m{\mathfrak{m}}
\def\mz{\mathbf{H}(1)}
\def\du{d(g_u)}
\def\L{\widetilde{\mathbf{U}}}
\def\bA{\mathbf{A}}
\def\bT{\widetilde{\mathbf{A}}}
\def\bH{\mathbf{H}}
\def\p{\mathfrak{p}}
\def\q{\mathfrak{q}}
\def\chri{\mathfrak{L}_1}
\def\chrii{\mathfrak{L}_2}
\def\lsim{\lesssim}
\def\gsim{\gtrsim}

\def\J{\mathbf{J}}
\def\var{\vartheta_1}
\def\JA{\mathrm{J}_\Lambda}

\def\cV{\mathbf{V}}
\def\Ga{\widetilde{\Gamma}}
\def\vv{\mathcal{V}}

\def\bY{\mathbf{Y}}
\def\bS{\mathbf{S}}
\def\bW{\mathbf{W}}
\def\rQ{\mathbf{Q}}
\def\cI{\mathcal{I}}
\def\cK{\mathbf{K}}
\def\wu{\widetilde{u}}

\def\cnA{\mathbb{A}}

\newtheorem{theo}{Theorem}[section]

\newtheorem{lem}[theo]{Lemma}
\newtheorem{cor}[theo]{Corollary}
\newtheorem{prop}[theo]{Proposition}

\newtheorem{remark}[theo]{Remark}

\newcommand\addctmk[1]{\phantomsection\addcontentsline{toc}{section}{#1}{}}
\numberwithin{equation}{section}

\title{Heat kernel asymptotics on the free step-two Carnot group with $3$ generators}
\author{Hong-Quan Li, Sheng-Chen Mao, Ye Zhang}
\date{}

\begin{document}

	\renewcommand{\theequation}{\thesection.\arabic{equation}}
	\setcounter{equation}{0} \maketitle
	
	\vspace{-1.0cm}
	
	\bigskip
	
	{\bf Abstract.} In this work, we establish the uniform heat kernel asymptotics as well as sharp bounds for its derivatives on the free step-two Carnot group with $3$ generators.  As a by-product, on this highly non-trivial toy model, we completely solve the Gaveau--Brockett problem, in other words,  we obtain the expression of the squared Carnot--Carath\'eodory distance,  as explicitly as one can possibly hope for. Furthermore, the precise estimates of the heat kernel, and its small-time asymptotic behaviors are deduced.
	
	\medskip
	
	{\bf Mathematics Subject Classification (2020):} {\bf 58J37; 35B40; 35H10; 35B45; 35K08; 43A80; 58J35; 43A85}
	
	\medskip
	
	{\bf Key words and phrases: Asymptotic behavior; Carnot--Carath\'eodory distance; Free step-two Carnot group; Heat kernel; Precise estimates}
	
	\medskip
	
	\addctmk{Contents}\tableofcontents
	
	\medskip
	
	\section{Introduction}
	\setcounter{equation}{0}
	
		The heat semi-group and its kernel play a crucial role in various areas of Mathematics. See for example \cite{LY86, BGV92, SC92, VSC92, AJ99, CD99, SC02, ACDH04, G09, BGL14} and huge references therein.
		
		We restrict our attention to the heat kernel in the setting of Carnot groups. Let $\ell = 2, 3, \ldots$. According to \cite[p.~44]{VSC92} or \cite[Definition~2.2.3]{BLU07}, a connected and simply connected Lie group $\G$ is said to be step-$\ell$ Carnot group (or stratified group) if its left-invariant Lie algebra $\mathfrak{g}$ admits a direct sum decomposition
		\begin{align*}
			\mathfrak{g} = \mathfrak{g}_1 \oplus \ldots \oplus \mathfrak{g}_{\ell}, \quad
			[\mathfrak{g}_1, \mathfrak{g}_{j - 1}] = \mathfrak{g}_j, \ 2 \le j \le \ell, \quad
			[\mathfrak{g}_1, \mathfrak{g}_{\ell}] = \{0\},
		\end{align*}
		where $[\cdot,\cdot]$ denotes the Lie bracket on $\mathfrak{g}$. We say that $\mathfrak{g}_1$ is the first slice in the stratification above. We identify $\G$  with
		$\mathfrak{g}$ via the exponential map, and fix a (bi-invariant) Haar measure $dg$ on $\G$, namely the lift of Lebesgue measure on $\frak g$ via $\exp$. Fix a basis $\BX = \{\X_1, \cdots, \X_n\}$ for $\mathfrak{g}_1$ and consider the sub-Laplacian $\Delta = \sum_{j=1 }^n \X_j^2 $. Let $p_h$ ($h > 0$) denote the associated heat kernel, i.e. the fundamental solution of $\frac{\partial}{\partial h} - \Delta$. It is well-known that $0 < p_h \in C^{\infty}(\R^+ \times \G)$, and (see for example \cite{VSC92})
		\[
		e^{h \, \Delta} f(g) = f * p_h(g) = \int_{\G} f(g_*) \, p_h(g_*^{-1} \, g) \, dg_*.
		\]
		We use $d :=d_{\BX}$ to denote the Carnot--Carath\'eodory distance defined by $\BX$, which inherits the left-invariant property  (see e.g. \cite[III. 4]{VSC92}), i.e.
		\[
		d(g \, g_1, g \, g_2) = d(g_1, g_2), \qquad \forall \, g, g_1, g_2 \in \G.
		\]
		We shall therefore put $d(g) := d(o, g)$, where $o$ is the identity of $\G$.
		Let $Q$ denote the homogeneous dimension of $\G$, namely, $Q = \sum_{j = 1}^{\ell} j \, {\rm dim} \, \mathfrak{g}_j$.
		
		There is a natural family of dilations on $\mathfrak g$ defined for $r > 0$ as follows:
		\begin{align*}
			\delta_r \left( \sum_{i = 1}^\ell v_i \right) := \sum_{i = 1}^\ell r^i v_i, \quad \mbox{with $v_i \in \mathfrak{g}_i$}.
		\end{align*}
		This induces the definition of dilation on $\G$, which we still denote by $\delta_r$. The following scaling properties are well-known
		\begin{align} \label{scap}
			d(\delta_h(g)) = h \, d(g), \quad p_h(g) = h^{-\frac{Q}{2}} p_1(\delta_{\frac{1}{\sqrt{h}}}(g)), \quad \, h > 0, \ g \in \G.
		\end{align}
		
		Let $\nabla = (\X_1, \ldots, \X_n)$ denote the horizontal gradient. The classical Li--Yau estimates for the heat kernel and their improvements (as well as their wide application)
		can be found in some much more general situations than Carnot groups. See for instance \cite{DP89,SC90, V90, VSC92,C93,S96,S04, CS08, BCS15} and the references therein. In particular, it holds for any $h > 0$ and all $g \in \G$ that:
		\begin{gather}
			p_h(g) \le C \, h^{-\frac{Q}{2}} \, \left( 1 + \frac{d(g)}{\sqrt{h}} \right)^{Q - 1} \, e^{-\frac{d(g)^2}{4 h}}, \label{GSB} \\
			p_h(g) \ge C(\varpi) \, h^{-\frac{Q}{2}} \,  e^{-\frac{d(g)^2}{4 \, (1 - \varpi) h}}, \qquad 0 < \varpi < 1, \label{GLB} \\
			|\nabla p_h(g)| \le C \, h^{- \frac{Q + 1}{2}} \,  \left( 1 + \frac{d(g)}{\sqrt{h}} \right)^{3 Q + 1} \, e^{-\frac{d(g)^2}{4 h}}, \label{GSE}
		\end{gather}
		where the absolute constant $C > 0$ is independent of $(h, g)$, and $C(\varpi)$ dependent only on $\varpi$. Moreover, we have for all $h > 0$ and $g \in \G$ that
		\begin{align} \label{GDB}
			\left| \frac{\partial^j}{\partial h^j} \nabla^k p_h(g) \right| \le C(k, j, \varpi) \, h^{-j - \frac{k}{2} - \frac{Q}{2}} \, e^{-\frac{d(g)^2}{4 \, (1 + \varpi) h}}, \qquad k, j = 1, 2, \ldots, \ 0 < \varpi \le 1.
		\end{align}
		By the fact that $\frac{\partial}{\partial h} p_h = \Delta p_h$, the last claim is equivalent to
		\begin{align} \label{GDB'}
			\left| \nabla^l p_h(g) \right| \le C(l, \varpi) \, h^{- \frac{l}{2} - \frac{Q}{2}} \, e^{-\frac{d(g)^2}{4 \, (1 + \varpi) h}}, \qquad l = 1, 2, \ldots, \ 0 < \varpi \le 1.
		\end{align}
		
		A direct consequence of \eqref{GSB} and \eqref{GLB} is the following well-known Varadhan's formula, which is actually valid in a very general circumstance than Carnot groups (cf. eg. \cite{L87},  \cite{AH05} and the references therein):
		\begin{align} \label{VF}
			\lim_{h \to 0^+} 4 \, h \ln{p_h(g)} = - d(g)^2, \qquad \forall \, g \in \G.
		\end{align}

		In this work, we investigate  much more complicated problems:  uniform asymptotic behaviours at infinity for $p_h$ in the sense of
		\begin{align} \label{AF}
			p_h(g) = h^{- \frac{Q}{2}} \, \Theta(g, h) \, e^{-\frac{d(g)^2}{4 h}} \, \left( 1 + o(1) \right), \qquad \mbox{as} \quad \frac{d(g)}{\sqrt{h}} \to +\infty,
		\end{align}
		as well as sharp estimates for its derivatives
		\begin{gather} \label{SGE}
			|\nabla \ln{ p_h(g) }| \le C \, \frac{d(g)}{h}, \qquad h > 0, \ g \in \G, \\
			\left| \nabla^l p_h(g) \right| \le C(l) \, h^{- \frac{l}{2}} \, \left( 1 + \frac{d(g)}{\sqrt{h}}  \right)^l \, p_h(g), \qquad l = 2, 3, \ldots, \ h > 0, \ g \in \G. \label{SDE}
		\end{gather}
		where $\Theta(g, h)$ has at most polynomial growth w.r.t. $d(g)/\sqrt{h}$. Here and in the sequel, we use the notation $f = o(w)$ if $\lim \frac{f}{w} = 0$. Notice that by the classical Li--Yau estimates, a direct consequence of \eqref{AF} is the precise bounds for the heat kernel. Also recall that these stronger estimates play a crucial role in \cite{HM89, Li06,BBBC08, HZ10, BCH20, LS22} for instance.
		
		Another by-product of \eqref{AF} is the complete description of short-time behaviour,
		\begin{align} \label{STA}
			p_h(g) = \Theta_h(g) \, e^{-\frac{d(g)^2}{4 h}} \left( 1 + o_g(1) \right), \qquad \mbox{as $h \to 0^+$,}  \  g \neq o,
		\end{align}
		where $\Theta_h(g) > 0$ is frequently of type $C(g) \, h^{-\sigma(g)}$. Sometimes, it allows us to determine the cut locus for the identity $\mathrm{Cut}_o$ on the underlying group, namely the set of points where the squared distance is not smooth.

We point out that the satisfactory explicit formula of the heat kernel for Carnot groups with step $\ell \ge 3$ is still missing. There are some attempts to study the short-time behaviour of $p_h$  on Carnot groups,  by means of the generalized Fourier transform and the Trotter product formula. However, as far as the authors know, it is unclear whether we can recover correctly  the well-known Varadhan's formulas \eqref{VF} along the way, even in the simplest Carnot group (namely the Heisenberg group of dimension 3) without using the explicit formula of the heat kernel.  For example, C. S\'eguin and A. Mansouri have conducted such investigation in \cite{SM12}, but the ``remainder terms''  in their main results (cf. e.g. (2), (3) as well as the estimate following (41) in \cite{SM12}) are not real remainders since they are of polynomial decay, which is much larger than the exponential decay of the heat kernel by Varadhan's formula \eqref{VF}. Also notice that the deduction from (38) to (39) in \cite{SM12} is incorrect due to the same reason.
			
		 From now on,	 we focus on step-two Carnot groups, in other words, $\ell = 2$. In such case, $p_h$ can be written as the partial Fourier transform w.r.t. the center of the underlying group. See for example \cite{C79} or \cite{BGG96}.  In particular, there exist two key functions, $s\cot s$ in the phase and $\frac{s}{\sin s}$ in the amplitude.
		
Recall that determining the cut locus and the distance between two points on the underlying space are two fundamental problems in (Riemannian) geometry and optimal control. It seems that the second one is much more difficult than the first one. As pointed out in \cite[\S~6.5.4]{B03}, the distance and the cut locus of a Riemannian manifold cannot in general be explicitly computed. To our best knowledge, few has been known even in the setting of step-two Carnot groups endowed with a left invariant Riemannian metric. As for the corresponding sub-Riemannian case, there are masses of works, but only limited to nonisotropic Heisenberg groups and H-type groups,  cf. e.g.  \cite{G77,BGG00, TY04, R05}. Their method is quite standard, namely finding all the possible candidates (which is usually to solve some differential equations with some boundary conditions) and picking out the shortest one.

Recently, the first author introduces an original, new and very powerful method in \cite{Li20} to attack these problems, say the operator convexity (of the key function $-s \cot{s}$). More precisely, by combining it with the method of stationary phase and Varadhan's formulas, the distance and the cut locus are characterized in an enormous kind of step-two Carnot groups, say GM-groups  (cf. \cite[Corollary 2.3 and Theorem 2.7]{Li20}, also \cite[Theorem 4, Corollary 9]{LZ21} for various equivalent characterizations of GM-groups via basic geometric properties).  Furthermore, for a given step-two Carnot group $\G$, by using the known sub-Riemannian exponential mapping in addition, up to a set of measure zero, all normal geodesics from the origin to any given point $g\in \G$ are characterized, by \cite[Theorem 2.4]{Li20}, which implies more qualitative results of $d(g)^2$.
		
	It follows from \cite[\S~11]{Li20} and \cite[Corollary 13]{LZ21} that the simplest step-two non-GM group is the free step-two Carnot group with $3$ generators, $N_{3,2}$ (see below for a definition). In such case,  the long-standing open problem of Gaveau--Brockett is completely solved in \cite[\S~11]{Li20} and \cite[\S~7]{LZ21} (cf. also Theorem \ref{RLT1} and Corollary \ref{RLT2} below) by means of two quite different methods, which rely heavily on the Hamilton--Jacobi theory.
			Also notice that the results of \cite{Li20, LZ21} are very dependent on the properties of $s \cot{s}$, but the other key function $\frac{s}{\sin{s}}$ is not used at all.
			
			Unlike in the setting of GM-groups, the expression of $d^2$ on $N_{3, 2}$ cannot be well explained directly by the integral expression for the heat kernel. Motivated by this new phenomenon, the first and third authors naturally turned to searching for a more essential explanation by studying the small-time behavior for the heat kernel in such case. In particular, we discovered the expression \eqref{ehk4} below where the properties of Bessel functions have been used, since $(\frac{s}{\sin{s}})^{-1} = \sqrt{\frac{\pi}{2}} \, s^{-\frac{1}{2}} J_{\frac{1}{2}}(s)$. See Section \ref{s4} below for more details. Furthermore, by using the recursion formulas for the Bessel functions and the Fourier transform of the Gaussian functions, the initial expression in the form of oscillatory integral for the heat kernel has been reformulated into a much more suitable one in the form of Laplace integral with positive integrand (cf. \cite[(15), (16) and Proposition 4]{LZ212}). Following  main ideas in \cite{Li20} (with some new methods introduced), the new expression of the heat kernel allows us to understand some basic geometric problems well in the setting of step-two Carnot groups. For instance, at least in principle, the extremely difficult problem of determining the distance from any given point to the origin is reduced to an elementary calculus, see \cite[Algorithm 1]{LZ212} for more details.  We point out that the methods in \cite{LZ212} can be adapted to study the corresponding Riemannian geometric problems on any step-two Carnot group.
				
				One goal of this article is to re-obtain the expression of $d^2$ on $N_{3, 2}$ via the heat kernel, without using the Hamilton--Jacobi theory. Indeed it can be considered as an illuminative model as well as application of \cite{LZ212}.  Notice that in this article, we do not care about any other geometric problems except the expression for $d^2$. Hence some reductions such as \eqref{orthod} will be used. Moreover, for this article to be self-contained, Theorem \ref{mapL} and  Corollary \ref{RLT2}, as well as their proofs, namely main parts in Sections \ref{sec3} and \ref{ssd} (up to a slight modification), are extracted from \cite{Li20, LZ21}.
				
				On the other hand, by adapting the method introduced in \cite{Li20}, refining and extending the argument in the proof of \cite[Theorem 1.3]{Li12} (also of \cite[Theorem 3]{LZ19}, the first and third authors study \eqref{AF}-\eqref{SDE} in \cite{LZ22} on a large class of step-two Carnot groups which are simultaneously GM and M\'etivier groups.
				
				Notice that $N_{3, 2}$ is neither a GM-group nor a M\'etivier group. The main goal of the work is to consider $N_{3, 2}$ as a toy model in order to study \eqref{AF}-\eqref{SDE}. We hope it provides a routine for (at least concrete) step-two Carnot groups. Although some new original and powerful methods are provided in \cite{Li20, LZ212, LZ22}, we will find that they are not enough and it is still very challenging to complete our mission on $N_{3, 2}$. More explanation about our major obstacles can be found in Subsection \ref{ideaS}.
	
    As an application of results obtained in this work, we will establish the gradient estimate for the heat semi-group on $N_{3,2}$ in a forthcoming paper.  Moreover, the method presented here will be adapted to
	study the heat kernel associated to the full Laplacian (which corresponds to the canonical left-invariant Riemannian structure) on $N_{3,2}$ as well. Precise estimates of the heat kernel for the Riemannian case will be provided in a future work.
	
	\subsection{Notation}
	
	We denote the sets of integers and nonnegative integers by  $\zz$ and $\nn$ respectively, and put $\zz^*:=\zz\setminus\{0\}$, $\nn^*:=\nn\setminus\{0\}$. Given a multi-index $\alpha=(\az_1,\az_2,\az_3) \in \N^3$ and $x = (x_1, x_2, x_3)$, $\lambda = (\lambda_1, \lambda_2, \lambda_3) \in \R^3$, we write $\partial^\alpha_x := \partial^{\alpha_1}_{x_1} \partial^{\alpha_2}_{x_2} \partial^{\alpha_3}_{x_3}$, $\lambda^{\alpha} := \lambda_1^{\alpha_1} \lambda_2^{\alpha_2} \lambda_3^{\alpha_3}$, and $|\alpha| := \alpha_1 + \alpha_2 + \alpha_3$ in the usual way; we also use the notation that $\lz = (\lz_1, \lz') \in \rr \times \rr^2$. The real part and imaginary part of a complex number $z$ are denoted by $\Re(z)$ and $\Im(z)$ respectively. Supposing that $q\in\nn^*$, we let $\mathbb{I}_q$ represent the $q \times q$ identity matrix, and ${\rm O}_q$ the orthogonal group of order $q$.
	
	The symbols $C$ and $c$ are used throughout to denote implicit positive constants which may vary from one line to the next. And when necessary we will specify with a subscript which parameters the values of $C$ and $c$ depend on.
	
	The usual asymptotic notation is employed. Let $w$ be a non-negative real-valued function. By $f = O(w)$ (resp. $f \lesssim w$ if $f$ is also real-valued) we mean $|f| \leq C w$ (resp. $f \leq C w$).  Correspondingly $f \gtrsim w$  if $f \ge C\, w$. Moreover, we will use the counterparts of such notation for Hermitian (in particular, real symmetric) operators or matrices. If the implicit constant depends on parameter $\az_0$, we will write $\az_0$ in the subscript of $O,o,\lesssim, \sim$, etc.
	
 The statement $f\ll w$, $f=o(w)$ or $w \gg f$ is shorthand for $f(x)/w(x) \to 0$ as $x$ tends to some $x_0$ (could be $\infty$); it will be clear that which variable is taken to the limit. Especially, if $f\equiv1$ we also write $w\to+\infty$. The statement $f=w(1+o(1))$ means $f-w = o(w)$.	

		For example, we use ``if $A \lesssim 1$ and $B \to +\infty$ then $E \sim 1$" to represent that ``for every $\zeta_0 > 0$, there exists a constant $C(\zeta_0)$ large enough and a constant $C'(\zeta_0) \ge 1$ such that when $A \le \zeta_0$ and $B \ge C(\zeta_0)$, we have $C'(\zeta_0)^{-1} \le E \le C'(\zeta_0)$".
	
	Finally, all the vectors appearing in this work are regarded as column vectors unless otherwise stated. However, we may write a column vector $t$ in $\R^q$ with scalar coordinates $t_1, \ldots, t_q$, simply as $(t_1, \ldots, t_q)$. Moreover, for a function $F := F(\nu, \nu^\prime)$ depending on $(\nu, \nu^\prime) \in \R^{q_1} \times \R^{q_2}$, we will use the following notation to denote the $\nu$-gradient (resp. $\nu$-Hessian matrix) at the point $(\nu_0, \nu^\prime_0)$:
	\[
	\nabla_{\nu_0}  F(\nu_0, \nu^\prime_0),  \quad (\mbox{resp.} \ \He_{\nu_0}  F(\nu_0, \nu^\prime_0)).
	\]

\section{Description of the setting and statements of results} \label{sDe}
\setcounter{equation}{0}

Recall the free step-two Carnot group with 3 generators $N_{3,2}$ is given by $\R^3 \times \R^3$ with the group structure
\[
(x,t) \cdot (x',t') = \left(x + x' , t + t' - \frac{1}{2} \, x \times x' \right),
\]
where ``$\times$'' denotes the cross product on $\R^3$, i.e.,
\[
x \times x' = (x_2 x_3' - x_3 x_2', x_3 x_1' - x_1 x_3', x_1 x_2' - x_2 x_1').
\]
Here we choose the same definition of $N_{3,2}$ as the one in \cite[\S~11]{Li20} rather than the one in \cite{LZ21}. In fact, they differ in a negative sign before the term $x \times x'$  and it will not affect the expression of the heat kernel. The corresponding left-invariant vector fields and  sub-Laplacian are given by
\begin{equation}\label{LVF}
	\begin{gathered}
		\X_1 := \frac{\partial}{\partial x_1} - \frac12 x_3 \frac{\partial}{\partial t_2} +\frac12 x_2 \frac{\partial}{\partial t_3}, \quad
		\X_2 := \frac{\partial}{\partial x_2} + \frac12 x_3 \frac{\partial}{\partial t_1} -\frac12 x_1 \frac{\partial}{\partial t_3}, \\
		\X_3 := \frac{\partial}{\partial x_3} -\frac12 x_2 \frac{\partial}{\partial t_1} +\frac12 x_1 \frac{\partial}{\partial t_2},
		\qquad \Delta := \sum_{j = 1}^3 \X_j^2.
	\end{gathered}
\end{equation}

It is well-known that the associated heat kernel $p_h$ ($h > 0$), i.e. the fundamental solution of $\frac{\partial}{\partial h} - \Delta$, has the following integral formulas (cf. e.g. \cite{G77, C79} or \cite{LP03})
\begin{align}\label{ehk}
	p_h(x,t) = \frac{\mathbf{C}}{h^\frac{9}{2}} \, p\left(\frac{x}{\sqrt{h}},\frac{t}{h}\right), \quad \forall \, h > 0,\,(x,t) \in N_{3,2},
\end{align}
with some positive constant $\mathbf{C}$ and
\begin{align}\label{ehk2}
	p(x,t) =  \, \int_{\R^3}\cV(\lambda) \,  e^{-\frac{1}{4} \widetilde{\phi}((x,t); \lambda)}  \, d\lambda,
\end{align}
where
\begin{align}\label{defV}
	\cV(\lambda) := \frac{|\lambda|}{\sinh{|\lambda|}} = \prod_{j = 1}^{+\infty} \left( 1 + \frac{ \lambda \cdot \lambda}{j^2 \pi^2}\right)^{- 1}
\end{align}
by \cite[1.431.2]{GR15} and
\begin{align}\label{MFF}
	\widetilde{\phi}((x,t);\lambda)  :=
	|x|^2 +  \frac{|\lambda| \coth{|\lambda|} - 1}{|\lambda|^2} \, (|\lambda|^2 |x|^2 - (\lambda \cdot x)^2)  - 4it \cdot \lambda .
\end{align}

The  corresponding reference function  in our setting is defined by (cf. \cite{Li20} and \cite{LZ22})
\begin{align} \label{defref}
	\phi((x,t);\tau) := \widetilde{\phi}((x,t); i \tau) = |x|^2 -  \frac{1 - |\tau| \cot{|\tau|}}{|\tau|^2} \, (|\tau|^2 |x|^2 - (\tau \cdot x)^2)  + 4t \cdot \tau.
\end{align}

Finally, the following equivalence between the Carnot--Carath\'eodory distance and a homogeneous norm (see for example \cite[III.4]{VSC92}) will be used again and again:
	\begin{align} \label{ehd}
		d(x,t)^2 \sim |x|^2 + |t|, \qquad \forall \, (x, t) \in N_{3,2}.
	\end{align}

\subsection{Reductions}

Our main target of this work is to establish the uniform heat kernel asymptotics at infinity on $N_{3,2}$ and the following simple observation will simplify our asymptotic problem to a large extent. In fact, using the orthogonal invariant property and the symmetry w.r.t. the origin of the heat kernel, we reduce the original problem (which is $6$-dimensional) to a $3$-dimensional one. Moreover, when deriving the explicit expression of the Carnot--Carath\'eodory distance,  we can even reduce to a $2$-dimensional problem; see the beginning of Subsection \ref{nss23n} for more details.

To be more precise, through the change of variables $\lambda \mapsto O \lambda$ and $\lambda \mapsto - \lambda$, \eqref{ehk2}-\eqref{MFF} imply that
\begin{equation}\label{prpty_p}
	p(O \, x, O \, t) = p(x,t), \quad \forall \, O \in \mathrm{O}_3, \qquad p(x,t) = p(x, -t).
\end{equation}
From this, together with the smoothness of $p$ and the benefit that our heat kernel asymtotics will be uniform enough, without loss of generality, we may assume that
\begin{align}\label{ass}
	x = |x| \, e_1 = |x|(1,0,0) \ne 0 \quad \mbox{ and } \quad t = (t_1,t_2,0), \mbox{ with } t_1 , t_2 > 0.
\end{align}
Here and in the sequel, $e_1$ will simultaneously denote the vectors in Euclidean spaces (possibly with different dimensions) with 1 in the first coordinate and zeros elsewhere. Further reduction will be provided in Assumption (A) below.

\subsection{Two key analytic diffeomorphisms}

Set in the sequel
\begin{gather}
	\R^2_{>, +} := \left\{(u_1, u_2) \in \R^2; \, u_2 > \frac{2}{\sqrt{\pi}} \sqrt{u_1} > 0\right\}, \\[1mm]
	\R_{<,+}^2 := \left\{ (u_1, u_2); \, u_1 > 0, 0 < u_2 < \frac{2}{\sqrt{\pi}} \sqrt{u_1}\right\},
 \end{gather} \\[-12mm]
 \begin{gather}
	\Omega_{+,1} := \{v := (v_1,v_2) \in \R^2; \, v_1, v_2 > 0, v_1^2 + v_2^2 < \pi^2 \}, \label{O+}  \\[2mm]
	\psi(\rho) := \frac{1 - \rho \, \cot{\rho}}{\rho^2}, \qquad \rho \not\in \left\{ k \, \pi; \, k \in \zz^* \right\}, \label{EFso}
 \end{gather} \\[-12mm]
 \begin{gather}
	\mathrm{K}_3(v_1, v_2) := 2 \psi(r) + \frac{\psi'(r)}{r} v_2^2 \ \mbox{ with } \ r := |v| = \sqrt{v_1^2 + v_2^2}, \label{nK3N} \\[1mm]
	\sqrt{2} \pi < 4.4933 < \var < 4.4935 < \frac{3}{2} \pi \ \mbox{ such that } \tan{\var} = \var, \label{cN32} \\[1mm]
	\Omega_{-,4} := \left\{(v_1,v_2); \, v_2 < 0, \pi < v_1 < r  < \vartheta_1, \mathrm{K}_3(v_1, v_2) < 0 \right\},   \label{O-4}
\end{gather}
and
\begin{align} \label{32N96}
	\Lambda(v_1, v_2) := v_2 \left[ \frac{\psi'(r)}{r} \, v_2 \, v + 2 \, \psi(r) \, e_2 \right], \qquad 0 < r = |v| \neq \pi < \var.
\end{align}
Notice that if $(u_1, u_2) = \Lambda(v_1, v_2)$, then $\Lambda(\pm v_1, \pm v_2) = (\pm u_1, \pm u_2)$. This analytic function actually comes from the equation of critical point: $\nabla_\theta \phi((e_1, 4^{-1} (u_1, u_2, 0)); \theta) = 0$. The interested readers can find the background as well as the geometrical meaning of $\theta$ in \cite[\S~2 and \S~11]{Li20}.

The starting point of this article is the following diffeomorphisms (see Figure \ref{fig1} for the sketch map):

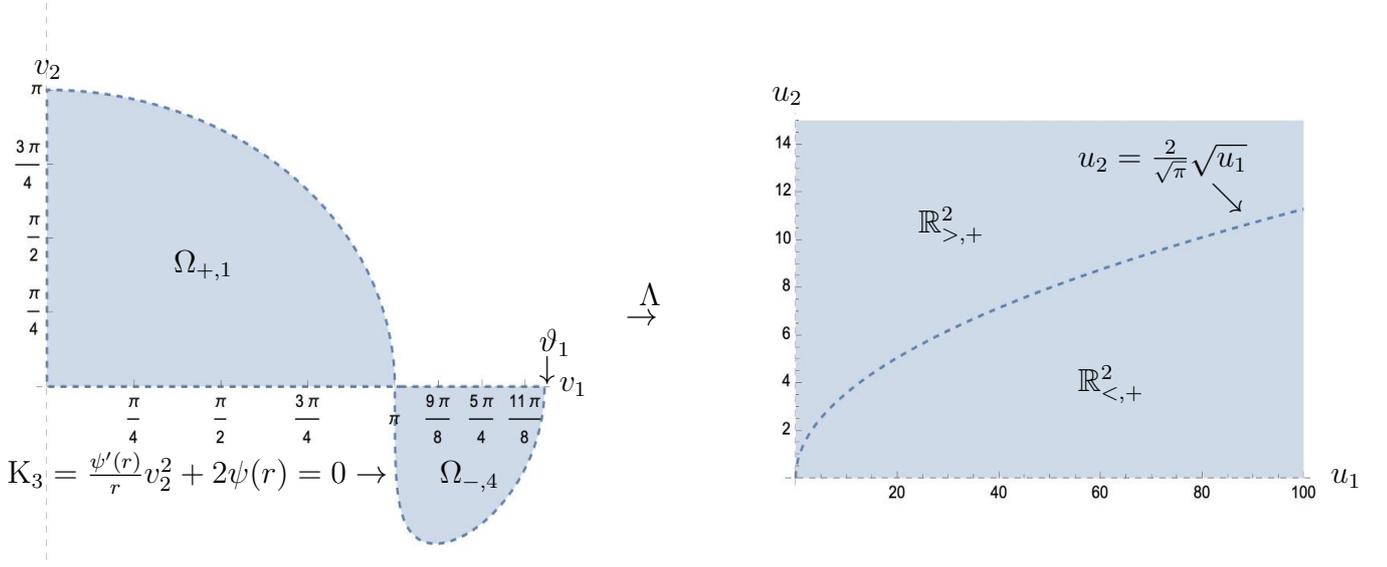
\begin{figure}[htp]
	\centering
	\begin{overpic}[width = 17.5cm, height=7.5cm]{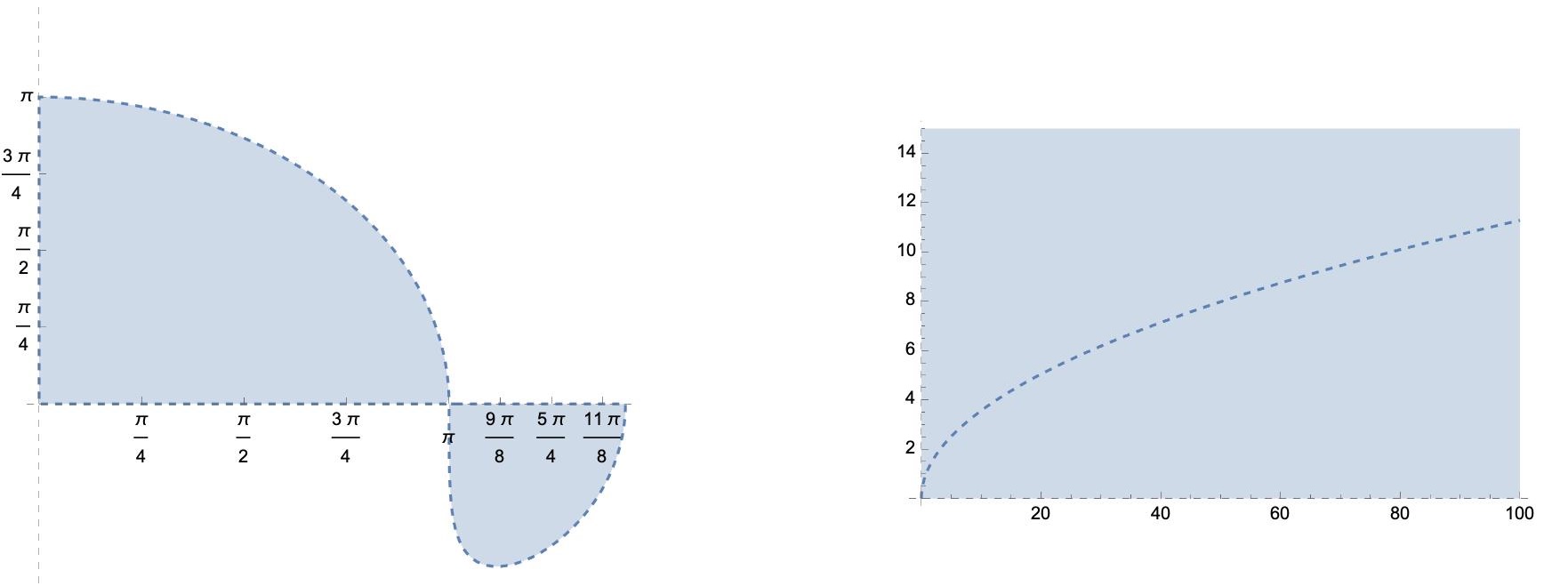}
		\put(32,6){$\Omega_{- , 4}$}
		\put(12,22){$\Omega_{+ , 1}$}
		\put(68,25){$\R^2_{>,+}$}
		\put(80,13){$\R^2_{<,+}$}
		\put(47,19.5){$\Lambda$}
		\put(46,18){$\to$}
		\put(41,13){$v_1$}
		\put(1.5,37){$v_2$}
		\put(99,6){$u_1$}
		\put(57,35){$u_2$}
		\put(90,27){$\searrow$}
		\put(80,30){$u_2 = \frac{2}{\sqrt{\pi}} \sqrt{u_1}$}
		\put(-0.5,6){$\K_3 = \frac{\psi'(r)}{r} v_2^2 + 2 \psi(r) = 0  \to$}
		\put(39.5,14){$\downarrow$}
		\put(39.5,16){$\vartheta_1$}
	\end{overpic}
	\caption[image]{Sketch map of the analytic-diffeomorphism $\Lambda$}\label{fig1}
\end{figure}

\begin{theo}\label{mapL}
	The map $\Lambda$ is an analytic-diffeomorphism from: \\
	(1) $\Omega_{+ , 1}$ onto $\R^2_{>,+}$. Moreover, let \[
	u = (\wu, 0) \mbox{\ with \ } \wu \in \R^2_{>,+}, \quad g_u = (e_1, 4^{-1} u) \in N_{3, 2}, \ \mbox{and } \theta = (\Lambda^{-1}(\wu),0).
	\]
	Then the Hessian matrix of $-\phi(g_u;  \cdot)$ at $\theta$ is positive definite. \\
	(2) $\Omega_{-, 4}$ onto $\R^2_{<, +}$. Furthermore, using the above notation with $\wu$ belonging to $\R^2_{<, +}$ instead of $\R^2_{>,+}$, the Hessian matrix of $\phi(g_u;  \cdot)$ at $\theta$ is nonsingular, and has
	exactly two positive eigenvalues.
\end{theo}

Its proof is based on the operator convexity and Hadamard's Theorem. See Section \ref{sec3} below for the details. We point out that the two diffeomorphisms first appeared in \cite{Li20}. In fact, the first claim is a direct consequence of  \cite[Propositions 11.1 and 11.2]{Li20},  while the second one is exactly \cite[Theorem 11.2]{Li20}.

In light of Theorem \ref{mapL}, we will make the following assumption repeatedly, unless otherwise specified:
\begin{align} \label{n78n}
	{\text{\bf Assumption (A):  }}
	\begin{cases}
		\theta := (\theta_1, \theta_2, 0) \  \mbox{with} \  (\theta_1, \theta_2) \in \Omega_{+ , 1} \cup \Omega_{-,4}, \,\ep:=\vtz_1-|\tz|;  \\[1mm]
		u := (u_1, u_2, 0) :=  (\Lambda(\theta_1, \theta_2), 0) \\  \qquad \mbox{(so $u_1, u_2 > 0$ and $\pi \, u_2^2 \neq 4 \, u_1$)}; \\[1mm]
		u_* := 4^{-1} u, \ g_u := (e_1, u_*) \in N_{3, 2}; \\[1mm]
		g = (x, t) := \delta_{|x|}(g_u) = (|x| e_1, \ |x|^2 u_*) =(|x|e_1, t_1, t_2, 0) \\  \qquad  \mbox{with $|x| > 0$}  ;\\[1mm]
		\ss = (\ss_1, \ss_2) :=(-1,0)\in\rr^2, \, \ww:=\tz_2 \, \psi(|\tz|)>0.
	\end{cases}
\end{align}

With this assumption, \eqref{ehk2} and  the reference function \eqref{defref} become
\begin{align} \label{ehk2'}
	p(g) = \int_{\R^3} \frac{|\lambda|}{\sinh{|\lambda|}} \, \exp\left\{ -\frac{|x|^2}{4} \left[ 1 + \frac{|\lambda| \coth{|\lambda|} - 1}{|\lambda|^2} (\lambda_2^2 + \lambda_3^2) - i \, u \cdot \lambda \right]    \right\} \, d\lambda,
\end{align}
and
\begin{align} \label{defref2}
	\phi(g; \tau) = |x|^2 \, \left[ 1 - \psi(|\tau|) \, (\tau_2^2 + \tau_3^2) + u \cdot \tau \right], \quad \tau = (\tau_1, \tau_2, \tau_3) \in \R^3,
\end{align}
respectively.

At this moment we will not formulate our main result, namely the uniform heat kernel asymptotics at infinity in the sense of \eqref{AF}, since there are too many notations to be introduced which may lead to confusion. Instead, we shall postpone the precise statement until Section \ref{Mtheo} and present some byproducts in this section. The first one is the following:

\subsection{Explicit expression for the squared Carnot--Carath\'eodory distance} \label{nss23n}

Using the scaling property of the heat kernel (cf.  \eqref{ehk}) and \eqref{VF}, \eqref{prpty_p} implies its counterpart for $d$:
\begin{align}
	\label{orthod}
	d(Ox,Ot)^2 = d(x,t)^2, \quad \forall \,  O \in \mathrm{O}_3, &\quad d(x,t)^2 = d(x,-t)^2.
\end{align}

 Combining with the scaling property of $d$, it suffices to determine $d(g)$ for special $g$, such as $(0, e_1)$ and $(e_1, (t_1, t_2, 0))$ with $t_1, t_2 \ge 0$.

In the sequel, we introduce some positive functions on $(0, \ +\infty)$,
\begin{gather}
	h(r) :=  r^2 + r \sin{r} \cos{r} - 2 \sin^2{r} \ (= \psi'(r) \, r^3 \sin^2r, \,  r \notin \{k \pi; \,  k \in \N^* \}), \label{nABn3}  \\
	\varphi_1(r) := \frac{r^2 - \sin^2{r}}{r - \sin{r} \cos{r}}, \quad
	\varphi_2(r) := \frac{r \, (r^2 - \sin^2{r})}{h(r)}, \quad \varphi_3 := \sqrt{\varphi_1 \varphi_2}. \label{defs}
\end{gather}
Indeed, the positivity of $h$ follows from Corollary \ref{nCn1} below, and the others are clear.

\begin{theo} \label{RLT1}
	Let $\theta, u, g_u, g$ be given as in the Assumption (A) (cf. \eqref{n78n}). Then
	\begin{eqnarray}\label{dEn}
		\begin{aligned}
			d(g_u)^2 &=  \frac{\theta_1^2}{|\theta|^2}  +  \frac{\theta_2^2}{\sin^2{|\theta|}}    =  - \theta_2^2 \, \psi(|\theta|)  +  u \cdot \theta + 1 = \, \varphi_1(|\theta|) \left( u_1 \, \frac{\theta_1}{|\theta|} + u_2 \, \frac{\theta_2}{|\theta|} \right)+ 1 \\
			&= \varphi_2(|\theta|) \, u_1 \, \frac{|\theta|}{\theta_1} + 1 = \varphi_3(|\theta|) \sqrt{u_1 \, \left( u_1 + u_2 \, \frac{\theta_2}{\theta_1} \right)} + 1.
		\end{aligned}
	\end{eqnarray}
	In particular, we have
	\begin{align} \label{nDFn}
		d(g_u)^2 = \phi(g_u; \theta), \qquad  d(g)^2 = \phi(g; \theta).
	\end{align}
\end{theo}

For the other cases, an argument of limit implies that:

\begin{cor} \label{RLT2}
	It holds that:
	{\em\begin{compactenum}[(i)]
			\item For $g_* = (0,e_1)$, we have $d(g_*)^2 = 4 \pi$.
			
			\medskip
			\item If $\alpha > 0$ and $g_{\alpha} := (e_1, 4^{-1} (\frac{\alpha^2}{\pi}, \frac{2}{\pi} \alpha, 0))$ such that $(\frac{\alpha^2}{\pi}, \frac{2}{\pi} \alpha) \in \partial \R^2_{>, +}$.  Then $d(g_\alpha)^2 = 1 + \alpha^2$.
			
			\medskip
			\item Let $\beta > 0$ and $g(\beta) := (e_1, 4^{-1} (\beta, 0, 0))$. Then
			\begin{align*}
				d(g(\beta))^2 =  \varphi_3(r) \, \beta + 1,
			\end{align*}
			where $r$ is the unique solution of the following equation
			\begin{align}\label{DCUTP}
				- 2 \, \psi(r) \, \sqrt{r^2 + 2 \, r \, \frac{\psi(r)}{\psi'(r)}} = \beta, \quad \pi < r < \var.
			\end{align}
			
			\item Set $\mu(\rho) := \frac{2 \rho - \sin(2 \rho)}{2 \sin^2 \rho}$, $-\pi < \rho < \pi$. For $g(\gamma)^* := (e_1, 4^{-1} (0, \gamma, 0))$ with $\gamma > 0$, we have
			\[
			d(g(\gamma)^*)^2 = \left( \frac{r}{\sin{r}}\right)^2, \quad \mbox{where $r$ is the unique solution of $\mu(r) = \gamma$.}
			\]
	\end{compactenum}}
	
\end{cor}

We give two remarks on these results as follows:

\begin{remark}\label{Rem21}
	 (1) For (i), the distance between $o$ and $g = (0, t)$ has been computed in \cite{B84}
		in the setting of free step-two Carnot groups with $k \geq 3$ generators. (iii) have been obtained in \cite{MM17}. (ii) (resp. (iv)) can be found in (resp. deduced from) \cite[Corollary 11.1]{Li20} (resp. \cite[Theorem 11.1]{Li20}). As for our main results, i.e. \eqref{nDFn}, the case where $g_u \in \R^2_{>, +}$ is given by \cite[Theorem 11.1]{Li20}. For the opposite case $g_u \in \R^2_{<, +}$, there exist two different proofs (as mentioned in Introduction), one is based on \cite[Corollary 2.4]{Li20} (cf. \cite[\S 11.2]{Li20}), the other is based on \cite[Theorem 2.4]{Li20} (cf. \cite[\S 7.4]{LZ212}).
	
	(2) The LHS of \eqref{DCUTP} is exactly the function $\frac{4}{P(r)}$ with  $P$ defined as in \cite[(3.3)]{MM17}. Then from
	\cite[Lemma 3.5]{MM17} it follows that $P$ is a strictly increasing diffeomorphism between $(\pi , \ \vartheta_1)$ and $(0, \ +\infty)$, which justifies the uniqueness of the solution $r$ in $(\pi, \  \vartheta_1)$.
	
	(3) For (iv), the function $\mu$ is actually a differomorphism from $(-\pi,\pi)$ to $\R$. See \cite[Lemme 3, p. 112]{G77} for more details. In this case, the expression of the square distance is the same with the point $(e_1,4^{-1} \gamma)$ on the Heisenberg group  of real dimension $3$.
\end{remark}

\begin{remark}
In our setting $N_{3, 2}$, ${\rm Cut}_o = \{(x, t); \, \mbox{$x$ and $t$ are linearly dependent} \}$. See for example  \cite[\S~7.5]{LZ21} for an elementary explanation. We will not use this fact in the proof, but it helps us  understand better some of our difficulties encountered in the uniform heat kernel asymptotics.
\end{remark}

\subsection{Precise estimates for the heat kernel} \label{sec24}

Another main byproduct of our uniform heat kernel asymptotics at infinity is the following sharp upper and lower bounds:

\begin{theo}\label{pbnd}
	Under the Assumption (A) (cf. \eqref{n78n}), we have
	\begin{equation}\label{pbound}
		\begin{aligned}
			p(x,t) \sim (1+d(g))^{-2}\frac{1+\ep  \, d(g)}{1+\ep \,  d(g) + \ep  \, t_2^{\frac12} \, |x|^{\frac12} \,  (d(g)^2-|x|^2)^{\frac14}} \, e^{-\frac{d(g)^2}{4}}.
		\end{aligned}
	\end{equation}
\end{theo}

Here we use the formula \eqref{pbound} in order to match the precise estimates for the heat kernel in our setting with the known results, cf. eg. \cite{Li12, LZ19, LZ22}. However, a much more explicit expression for $d(g) \gg 1$ can be found in Corollary \ref{simHUD} below.

	\begin{remark}
		 (1) From the above explicit expression of $d^2$ on $N_{3, 2}$, it is easy to get that $d(g)^2 \ge |x|^2$ for any $g = (x, t) \in N_{3, 2}$, with the equality holding if and only if $g = (x, 0)$. Indeed, the result is still valid on any step-two Carnot group (cf. \cite[\S 2]{Li20}). 	
			
			(2) Let $\G$ be a M\'etivier group with ${\rm dim} \, \mathfrak{g}_2 = m$, the standard Laplace's method implies that (cf. also \cite[\S 4.2]{Li20}) its heat kernel with time $1$ at $(x, 0)$ is $\sim |x|^{-m} e^{-|x|^2/4}$ as $|x| \to +\infty$. Remark that our result \eqref{pbound} with $g = (|x| \, e_1, 0)$ and $|x| \to +\infty$ is totally different from the classical one. Indeed in our case, the Laplace's method is no longer applicable, since it follows from \eqref{MFF} that the set of minimal points of $\widetilde{\phi}((|x| \, e_1, 0); \cdot)$ equals $\{(\lambda_1, 0, 0); \, \lambda_1 \in \R \}$. From a geometric point of view, it says that there exists on $N_{3, 2}$ a non-trivial abnormal set (or geodesics), which is exactly $\{(x, 0); x \in \R^3\} \setminus \{ o \}$.
			
			(3) Compare this result with the one in the setting of Heisenberg groups or even generalized H-type groups (cf. \cite{Li07, Li10, LZ19, LZ22}), \eqref{pbound} is much more complicated since the group law is more complex. Naturally, we believe that there will be some more complicated terms for precise bounds of the heat kernel on concrete step-two groups.

	\end{remark}

\subsection{Idea of the proof} \label{ideaS}

Let's start by explaining how to prove $d(g_u)^2 = \phi(g_u; \theta)$, namely the first equality in \eqref{nDFn}.  From \eqref{ehk2'} and Varandhan's formula, it suffices to investigate the asymptotic behavior of
	\begin{align} \label{nIFn}
		\int_{\R^3} \frac{|\lambda|}{\sinh{|\lambda|}} \, \exp\left\{ -\frac{1}{4 h} \left[ 1 + \frac{|\lambda| \coth{|\lambda|} - 1}{|\lambda|^2} (\lambda_2^2 + \lambda_3^2) - i \, u \cdot \lambda \right]    \right\} \, d\lambda, \  h \rightarrow 0^+.
	\end{align}
	It is a typical oscillatory integral. The usual processing method is to use  the method of stationary phase, and it allows us to guess the correct answer (cf. Corollary \ref{smhk})
	\begin{align}\label{rteigen}
		\frac{|\theta|}{\sin{|\theta|}} e^{-\frac{\phi(g_u; \theta)}{4 h}} \, (8 \pi h)^{\frac{3}{2}} \, \prod_{j = 1}^3 \mathrm{r}_j^{-\frac{1}{2}} \, (1 + o(1)),
	\end{align}
	where $\mathrm{r}_1, \mathrm{r}_2, \mathrm{r}_3$ are the eigenvalues of $-{\He}_{\theta} \, \phi(g_u ; \theta)$, and we adopt the convention $(-r)^{-1/2} = -i \,  r^{-1/2}$ for $r > 0$.
	Indeed, it is a special case of \cite[\S 4]{Li20} provided $g_u \in \R^2_{>, +}$. More precisely, it suffices to deform the contour from $\R^3$ to $\R^3 + i \theta$, then apply the method of stationary phase at the nondegenerate critical point $i \theta$, since all assumptions of the method can be verified to be satisfied. However, for the opposite case $g_u \in \R^2_{<, +}$, the situation becomes very tricky, and the two main obstacles we encountered were as follows. The first one is how to choose a suitable integration path passing through the point $i \theta$ so that all assumptions hold for the method of stationary phase. The second one is how to treat the ``residue problem'', since $|\theta| > \pi$ in such case and the integrand in \eqref{nIFn} has singularities as $\{\lambda + i \tau; \, |\lambda|^2 + 2i \lambda \cdot \tau - |\tau|^2 = -k^2 \pi^2, k \in \N^* \}$.
    It is actually the main difference between GM-groups and non-GM groups.

This motivates us to find a more appropriate expression. For that,  we drop $h$, and set in the sequel
	\begin{align*}
		\rP(X,T) := \int_{\R^3}  \, \vv(\lambda) \, e^{-\frac{1}{4} \Ga((X,T);\lambda)} \, d\lambda, \qquad X \in \R^2, \, T \in \R^3,
	\end{align*}
	with
	\begin{align*}
		\vv(\lambda) := \frac{|\lambda|^3}{|\lambda| \cosh{|\lambda|} - \sinh{|\lambda|}}, \quad
		\Ga((X,T);\lambda) := \frac{|\lambda|^2}{|\lambda| \coth{|\lambda|} - 1} |X|^2 - 4 i \, T \cdot \lambda.
	\end{align*}
 Remark that $\Ga$ admits the following scaling property: $\Ga((h X, h^2 T); \lambda) = h^2 \Ga((X,T);\lambda)$ for all $h > 0$ and $(X, T)$. By means of properties of Bessel functions, we can establish the following much more useful expression (compared with \cite[(15)]{LZ212} which is valid for any step-two Carnot groups):
	
	\begin{prop} \label{nPn1}
		Let $\theta, u, g$ and $\ww$ be as in Assumption (A) (cf. \eqref{n78n}). Then we have
		\begin{equation}\label{ehk4}
			p(g) =\frac{1}{4\pi} \, |x|^2 \, \ww^2\, e^{-\frac{|x|^2}{4}} \int_{\R^2} \cP(s) \, ds,
		\end{equation}
		where
		\begin{align} \label{def_cP}
			\cP(s) = \cP(x,u;s) := \rP\left( s\,|x|\,\ww,\frac14|x|^2(u + 2\ww \, s_1 \, e_2 + 2\ww \, s_2 \,  e_3)\right).
		\end{align}
	\end{prop}
	
	Its proof is provided in Subsection \ref{sec51}.   We emphasize that we have used here the new coordinates $|x|\, \ww \, s$ in view of the scaling property of $\Ga$. The reason for our choice of $\ww$ herein can be found in Theorem \ref{tmm} below. The key point is that $\rP$ can be considered formally as the heat kernel at time $1$ on an imagined but non-existent {\it Heisenberg-type group $\H(2, 3)$}, with dimension $3$ in the center and dimension $2$ on the first slice in the stratification.  To see this more explicitly, as in \cite[Theorem~2.2 and Corollary~2.3]{Li20} (cf. also \cite[Theorem~4.2]{LZ22}), we introduce the squared ``intrinsic distance'' associated to $\rP$:
	\begin{align*}
		\rD(X,T)^2 := \sup_{\tau \in \R^3, \ |\tau| < \vartheta_1} \Ga((X,T); i \tau), \qquad (X,T) \in \R^2 \times \R^3.
	\end{align*}
	See Proposition \ref{lD} below for more properties of $\rD(X, T)^2$, especially the scaling property.
	
	\medskip
	
	Let $I_0$ denote the modified Bessel function of order $0$ (cf. \cite[\S~8.431.3]{GR15}), i.e.,
	\begin{align} \label{defI0}
		I_0(r) := \frac1{\pi} \int_0^\pi e^{r\cos\gz} \, d\gz = \frac1{2\pi} \int_{-\pi}^\pi e^{r\cos\gz} \, d\gz,
	\end{align}
	and set in the sequel the even functions
	\begin{align} \label{qdef}
		\Upsilon(r) := \frac{r^2}{1 - r \, \cot{r}}, \qquad \q(r):=\frac{r^2 \, \Upsilon(r)}{-\sin r \, \Upsilon'(r) \, \sqrt{-\Upsilon''(r)}}, \qquad r\in(-\vtz_1,\vtz_1).
	\end{align}
	
	The following proposition says that $\rP$ satisfies the uniform bounds \cite[(1.5)]{Li10} as well as the uniforms asymptotics behaviors \cite[Th\'eor\`emes 1.4-1.5]{Li10} with $n = 1$ and $m = 3$ (of course, with $(x, t; |x|, d(x, t))$ therein replaced by our $(X, T; |X|, \rD(X, T))$, and some slight, natural modifications):
	
	\begin{prop}\label{lP}
		$\rP$ admits the following properties:	
		{\em\begin{compactenum}[(i)]
				\item The positivity, namely $\rP(X,T) > 0$.
				\item $P(1; 3; \cdot, \cdot)$-type uniform bounds in the sense of \cite[Th\'eor\`eme~1.1]{Li10}, namely,
				\begin{equation} \label{ulbdd}
					\rP(X,T) \sim \frac{1}{(1 + \rD(X,T))^2 \,  (1 + |X| \, \rD(X,T))^{\frac{1}{2}}} \, e^{- \frac{\rD(X,T)^2}{4}}, \quad \forall (X, T) \in \R^2 \times \R^3.
				\end{equation}
				
				\item It holds uniformly for any $X \ne 0$ and all $\rD(X,T)$ large enough that
				\begin{equation*}
					\rP(X,T)=\sqrt{2\vtz_1}16\pi^2 \, \q(|\tau^*|) \, e^{-\frac{\rD(X,T)^2}{4}}e^{- \frac{\vartheta_1 |X|^2}{2(\vartheta_1 - |\tau^*|)}}
					\,I_0 \left( \frac{\vartheta_1 |X|^2}{2(\vartheta_1 - |\tau^*|)} \right) \,
					\frac{ (\vartheta_1 - |\tau^*|)^{-\frac12}}{ |X|^2  }\,(1 + o(1)),
				\end{equation*}
				where $\tau^*=\tau^*(X,T)$ is the unique critical point of $\widetilde{\Gamma}((X,T),i\,(\cdot))$ in $B_{\rr^3}(0,\var)$.
		\end{compactenum}}
	\end{prop}
	To show Proposition \ref{lP}, it is enough to adapt the method in \cite{LZ212} for (i), and the method in \cite{LZ22} for (ii) and (iii). Hence its proof is postponed to Appendix \ref{secA}.

	 Return to $p(\delta_{h^{-1/2}}(g_u))$ ($h \to 0^+$), namely \eqref{nIFn}. Using the scaling property of $\rD$ (cf. Proposition \ref{lD} (iii) below), Propositions \ref{nPn1} and \ref{lP} say that it is now a typical Laplace-type integral with positive integrand. Then inspired by the main idea of the standard Laplace's method, it is natural to study the minimum of $\rD\!\left( \ww \, s,\frac14(u + 2\ww \, s_1 e_2 + 2\ww \, s_2  e_3)\right)^2$. The following theorem will play an important role.
	
	\begin{theo}\label{tmm}
		Let $\theta, u$ and $\ww$  be as in  Assumption (A) (cf. \eqref{n78n}). Set
		\begin{align}
			\cD(u;s) :=  \rD\!\left( \ww s,\frac14(u + 2\ww s_1 e_2 + 2\ww\,  s_2  e_3)\right)^2, \quad s \in \R^2. \label{decD}
		\end{align}
		Then $\ss =(-1,0)$ is the unique minimum point of $\cD(u;\cdot)$.
		Moreover,  with $\varphi_1, \varphi_2$ and $\varphi_3$ defined by \eqref{defs}, we have:
		\begin{equation}\label{dEn2}
			\begin{aligned}
				\cD(u;\ss)
				&=   -\theta_2^2 \, \psi(|\theta|) +   u \cdot \theta
				=  \frac{\theta_2^2 }{ |\theta|^2} \left[ \left( \frac{|\theta|}{\sin{|\theta|}}\right)^2 -1 \right] = \frac{\varphi_1(|\theta|)}{|\tz|} u\cdot\tz \\[1mm]
				&= \varphi_2(|\theta|) \, u_1 \, \frac{|\theta|}{\theta_1} = \varphi_3(|\theta|) \sqrt{u_1 \! \left( u_1 + u_2 \, \frac{\theta_2}{\theta_1} \right)}.
			\end{aligned}
		\end{equation}
	\end{theo}
	
	Its proof will be postponed to Section \ref{s4}.  Combining this with Proposition \ref{lP} (ii), it is easy to show that $d(g_u)^2 = 1 + \cD(u; \ss)$, namely Theorem \ref{RLT1}.
		
		However, the Laplace's method is far from sufficient for the main goal of this article. Roughly speaking, the phase function in our Laplace-type integral (cf. Propositions \ref{nPn1} and \ref{lP}), namely $|x|^2 \cD(u; s)/4$, is not $C^1$ when $s = 0$ (see Remark \ref{smoD} below), which makes things difficult when $\ww$ small. In fact, as far as the authors know, there is no suitable method to deal with this situation. Furthermore, let $\m$ denote its minimum (i.e. $\m = |x|^2 \cD(u; \ss)/4$), $\chri, \chrii > 0$  two eigenvalues of  $\mathrm{Hess}_{\ss} \,\frac{|x|^2}{4}\cD(u;\ss)$ with $\chri \gtrsim \chrii$ (cf. Subsection \ref{pfN} below).  It is worthwhile to point out that for $d(g) \to +\infty$, the Laplace's method is sufficient only for the case where both $\m, \chri$ and $\chrii$ are large enough.
		
		For instance, in the case $\m \to +\infty$, we first use Proposition \ref{lP} (iii) to simplify the integrand in \eqref{ehk4}. The most delicate situation is that $|\theta| \ge 1$ with $\chri \gg 1$, and notice that $\chrii$  can be bounded. To obtain the uniform asymptotic behavior of $p(g)$ in such case, we will choose suitable coordinates (in fact the modified polar coordinates up to a scaling), by which the phase function can be divided into the angular part and the radical part in the main region. We will use Laplace's method to deal with  the radical part and the method in \cite{Li07} to cope with the angular part. In the opposite case where $\m$ is bounded, we'll make use of a completely different approach.

		Furthermore, to establish the sharp estimates for the derivatives of the heat kernel, unlike in the setting of GM-M\'etivier groups, new techniques need to be introduced.
	
 The remainder of this article is organized as follows. Theorem \ref{mapL} is proved in Section \ref{sec3}. The proofs of Proposition \ref{nPn1} and Theorem \ref{tmm} are given in Section \ref{s4}. We obtain the explicit expression of the squared distance as a consequence of Laplace's method in Section \ref{ssd}.
 The uniform asymptotic behaviour at infinity for the heat kernel at time 1 are divided into cases. We establish the first asymptotic in Section \ref{s44} for the case where $|\theta| \le 3$ and $\theta_2 |x| \to +\infty$.  After some preparations in Section \ref{s5}, we will attack the most difficult and complicated situation in Section \ref{bigchr}. Section \ref{schr} is devoted to the case where $\m \to +\infty$ with $\chri \lsim 1$, while Section \ref{s32} is for the remaining case. The summaries of our main results as well as the proof of Theorem \ref{pbnd} are collected in Section \ref{sec11}. Finally in Section \ref{sdbpd} we establish the sharp bounds for derivatives of the heat kernel. In Appendix \ref{secA} we give the proof for our Proposition \ref{lP}.

\section{Proof of Theorem \ref{mapL}}\label{sec3}
\setcounter{equation}{0}

\medskip

In this section, we establish the two key analytic-diffeomorphisms provided in Theorem \ref{mapL}. Recall that this part (up to some modification) is extracted from \cite{Li20}.

As we have seen in Theorems \ref{mapL}-\ref{RLT1} and Corollary \ref{RLT2}, the
function $\psi$ plays an important role. Let us begin with

\subsection{Fine properties of the function $\psi$ and its derivatives}

Notice that:
\begin{align}\label{EFs}
	\psi(r) := \frac{1 - r \cot{r}}{r^2} =  \sum_{j = 1}^{+\infty} (j \pi)^{-2} \left(  \frac{1}{1 -  \frac{r}{j \, \pi}} +  \frac{1}{1 +  \frac{r}{j \, \pi}}  \right) = 2 \sum_{j = 1}^{+\infty} \frac{1}{(j \, \pi)^2 - r^2}.
\end{align}
In the second and last ``$=$'' we have used the series expansion of the function $r \cot r$; see for example \cite[\S 1.421.3, p.\,44]{GR15}  with slight modification. That is,
\begin{align} \label{IS}
	-r \cot{r} = -1 + r^2  \sum_{j = 1}^{+\infty} (j \pi)^{-2} \left(  \frac{1}{1 -  \frac{r}{j \, \pi}} +  \frac{1}{1 +  \frac{r}{j \, \pi}}  \right) = -1 + \int_{-\frac{1}{\pi}}^{\frac{1}{\pi}} \frac{r^2}{1 - \rho r} \, d\nu(\rho),
\end{align}
where the positive finite measure $\nu$ on $[-\frac{1}{\pi}, \, \frac{1}{\pi}]$
is given by
\begin{align*}
	\nu = \sum_{j = 1}^{+\infty} \left( \frac{1}{j \, \pi} \right)^2 \left( \delta_{\frac{1}{j \, \pi}} + \delta_{-\frac{1}{j \, \pi}} \right) \quad \mbox{with $\delta_a$ the usual Dirac measure at point $a$.}
\end{align*}

Now we collect some properties for $\psi$ on $[0,\pi)$.

\begin{lem} \label{n32l}
	We have
	\begin{gather}
		\psi''(r) > \frac{\psi'(r)}{r} \geq \lim_{r \to 0} \frac{\psi'(r)}{r} > 0, \quad 0 < r < \pi,     \label{Ii2} \\
		\psi(r) \, \psi''(r) > 2 \psi'(r)^2, \quad 0 \leq r < \pi.    \label{Ii3}
	\end{gather}
\end{lem}

\begin{proof}
	A simple and direct proof  for \eqref{Ii2} is to use the fact that (see the second equality in \eqref{EFs}, or \cite[\S 1.411.7, p.\,42]{GR15} with slight modification of notation)
	\begin{align*}
		\psi(r) = \sum_{j = 1}^{+\infty} b_j r^{2 j - 2}, \quad  \mbox{with } \, b_j > 0, \, \forall \, j \in \N^*.
	\end{align*}
	
	Aiming at \eqref{Ii3}, it follows from \eqref{EFs} that:
	\begin{align*}
		\psi(r) = \int_{-\frac{1}{\pi}}^{\frac{1}{\pi}} \frac{d\nu(\rho)}{1 - r \rho}, \quad \psi'(r) = \int_{-\frac{1}{\pi}}^{\frac{1}{\pi}} \frac{\rho}{(1 - r \rho)^2} \, d\nu(\rho), \quad \psi''(r) = 2 \int_{-\frac{1}{\pi}}^{\frac{1}{\pi}} \frac{\rho^2}{(1 - r \rho)^3} \, d\nu(\rho),
	\end{align*}
	then H\"older's inequality shows that
	\begin{align*}
		2 \psi'(r)^2 < 2 \left( \int_{-\frac{1}{\pi}}^{\frac{1}{\pi}} \frac{|\rho|}{(1 - r \rho)^2} \, d\nu(\rho) \right)^2 < \psi(r) \, \psi''(r).
	\end{align*}
	This concludes the proof of \eqref{Ii3} and hence the lemma.
\end{proof}

The following two lemmas will play an important part in the proof for the first analytic-diffeomorphism of Theorem \ref{mapL}. Lemma \ref{Hetau1} is also vital to the heat kernel asymptotics in Section \ref{s44}.

\begin{lem}
	It holds that
	\begin{align} \label{iN38}
		\psi(r) > \sqrt{\frac{\psi'(r)}{r}}, \qquad \forall \, 0 \, \leq r < \pi.
	\end{align}
\end{lem}

\begin{proof}
	By the second equality in \eqref{EFs}, we get
	\begin{align} \label{N32ei1}
		\psi(r) = 2 \sum_{j = 1}^{+\infty} \frac{1}{(j \, \pi)^2 - r^2}, \qquad \frac{\psi'(r)}{r} = 4 \sum_{j = 1}^{+\infty} \left[(j \, \pi)^2 - r^2 \right]^{-2},
	\end{align}
	which implies immediately the required estimate.
\end{proof}

\begin{lem}\label{Hetau1}
	Let $0 < \zeta_0 \le 1$.  Then
	\begin{align*}
		\He_\tau  \left( \tau_2^2 \, \psi(|\tau|) \right) \sim_{\zeta_0}
		\begin{pmatrix}
			\tau_2^2 & \quad \\
			\quad & 1
		\end{pmatrix} \ge 0, \quad \mbox{for all} \ \tau  = (\tau_1, \tau_2)  \in \R^2  \  \mbox{with $|\tau| \le \pi - \zeta_0$. }
	\end{align*}
\end{lem}

\begin{proof}
	In fact, writing $r:=|\tau|$, then a direct computation gives
	\begin{align*}
		\He_\tau  \left( \tau_2^2 \, \psi(|\tau|) \right) =  & \, \tau_2^2 \,  \left[ \frac{\psi'(r)}{r} \mathbb{I}_2 + \left( \psi''(r) - \frac{\psi'(r)}{r} \right) \frac{\tau}{r} \left( \frac{\tau}{r} \right)^\T \right] + 2 \psi'(r) \frac{\tau}{r} (0,\tau_2) \\
		&+ 2 \psi'(r) \left(
		\begin{array}{c}
			0 \\
			\tau_2 \\
		\end{array}
		\right)
		\left( \frac{\tau}{r} \right)^\T + 2 \psi(r) \left(
		\begin{array}{cc}
			0 & 0 \\
			0 & 1 \\
		\end{array}
		\right).
	\end{align*}
	
	Next, it is easy to prove that for $\xi  = (\xi_1, \xi_2) \in \R^2$ satisfying $|\xi| = 1$,
	\[
	\frac{1}{1 + \pi^2} \big( \tau_2^2 + \xi_2^2 \big) \le \tau_2^2 \xi_1^2+  \xi_2^2 \le \tau_2^2 + \xi_2^2, \quad \tau_2^2 \le \pi.
	\]
	Hence it remains to show that for  $\pi - |\tau| \ge  \zeta_0$,  we have
	\begin{equation}\label{stare}
		\begin{aligned}
			\tau_2^2 \xi_1^2+  \xi_2^2 &\sim \tau_2^2 +\xi_2^2 \sim_{\zeta_0} \langle \He_\tau  \left( \tau_2^2 \, \psi(|\tau|) \right) \, \xi, \xi \rangle \\[2mm]
			&= \left[ \frac{\psi'(r)}{r} + \left( \psi''(r) - \frac{\psi'(r)}{r} \right) r_0^2 \right] \tau_2^2 + 4 \psi'(r) \, r_0 \,  \xi_2 \, \tau_2 + 2 \psi(r) \,  \xi_2^2,
		\end{aligned}
	\end{equation}
	where $r_0 =  \xi \cdot \frac{\tau}{|\tau|}  \in [-1, 1]$. To prove this, we write RHS of \eqref{stare} as
	\begin{align}\label{stare2}
		\begin{pmatrix}
			\tau_2 &  \xi_2
		\end{pmatrix}
		\begin{pmatrix}
			\frac{\psi'(r)}{r} + \left( \psi''(r) - \frac{\psi'(r)}{r} \right) r_0^2 & 2 \psi'(r) \, r_0 \\
			2 \psi'(r) \, r_0 & 2 \psi(r)
		\end{pmatrix}
		\begin{pmatrix}
			\tau_2 \\
			\xi_2
		\end{pmatrix}.
	\end{align}
	Note that we have uniformly for all $0 \le r \le \pi - \zeta_0$ that  the diagonal entries of the last matrix
	\begin{align}\label{stare3}
		\frac{\psi'(r)}{r} + \left( \psi''(r) - \frac{\psi'(r)}{r} \right) r_0^2  \sim_{\zeta_0}  1, \quad
		\psi(r)  \sim_{\zeta_0} 1
	\end{align}
	by \eqref{N32ei1}, \eqref{Ii2} respectively, and the determinant
	\begin{align}\label{stare4}
		2 \left\{ \psi(r) \frac{\psi'(r)}{r} (1 - r_0^2) + \left( \psi(r) \, \psi''(r) - 2 \psi'(r)^2 \right) r_0^2 \right\}  \sim_{\zeta_0}  1
	\end{align}
	by Lemma \ref{n32l}, which implies its eigenvalues $ \sim_{\zeta_0}  1$, thereby proving \eqref{stare}.
\end{proof}

Recall that $\var$ denotes the unique solution of $\tan{r} = r$ in the interval $(\pi, \, 1.5 \, \pi)$. We will use repeatedly the following simple observation:

\begin{lem}
	We have for any $ 0 < r \neq \pi < \vartheta_1$ that:
	\begin{gather}
		\Upsilon(r) = \frac{1}{\psi(r)}, \quad \Upsilon'(r) = - \frac{\psi'(r)}{\psi(r)^2}, \quad\Upsilon''(r)= \frac{-\psi(r)\psi''(r)+2\psi'(r)^2}{\psi(r)^3}. \label{relUp}
	\end{gather}
	Moreover,
	\begin{equation}\label{mapL_p}
		\psi'(r) \sim \frac{r}{(\pi - r)^2}, \qquad |\psi(r)| \sim \frac{\var - r}{|\pi - r|}, \quad \mbox{for all $0 < r \ne \pi < \var$.  }
	\end{equation}
\end{lem}
\begin{proof}
The equations in \eqref{relUp} are trivial. For \eqref{mapL_p}, the first claim (resp. the second one) is a direct consequence of the second equality in \eqref{N32ei1} (resp. \eqref{EFs}, the fact that $\psi(\var) = 0$ and $\psi'(\var) > 0$).
\end{proof}

Now we turn to the property of $\psi$ on $(\pi, \var)$, which plays crucial roles in the proof of the second analytic-diffeomorphism in Theorem \ref{mapL}.

\begin{lem} \label{Ll34}
	It holds for any $r \in (0, \ +\infty) \setminus \{ j \pi; \, j\in\nn^*  \}$ that:
	\begin{align} \label{32ei3}
		\K_1 := \frac{\psi'(r)}{r} = 4 \sum_{j = 1}^{+\infty} \left[(j \, \pi)^2 - r^2 \right]^{-2} > 0.
	\end{align}
	Moreover, we have for any $\pi < r < \var$ that:
	\begin{gather}
		0 >  \frac{1 -  r \cot{r}}{r^2} = \psi(r) = - 2 \left[ \frac{1}{r^2 - \pi^2} - \sum_{j = 2}^{+\infty} \frac{1}{(j \, \pi)^2 - r^2} \right], \label{32ei2} \\[1mm]
		\K_2 := r^{-1} \left( \frac{\psi'(r)}{r} \right)'  = -16  \left[ \frac{1}{[r^2 - \pi^2]^3} - \sum_{j = 2}^{+\infty} \frac{1}{[(j \, \pi)^2 - r^2]^3} \right] < 0, \label{32ei4}\\[2mm]
		2 \psi(r) \, \K_2 - 4 \K_1^2 < 0, \qquad 5\K_1 + \K_2 \, r^2 <0. \label{32ei5}
	\end{gather}
	
\end{lem}
\begin{proof}
	Using \eqref{N32ei1}, we get \eqref{32ei3} and \eqref{32ei2} without difficulties. The inequality \eqref{32ei4} follows from \eqref{32ei2} and the observation that $(j \, \pi)^2 - r^2 > r^2 - \pi^2$ for each $j \geq 2$.
	
	From \eqref{32ei2} and \eqref{32ei4}, we see that
	\[
	2 \psi(r) \, \K_2 < 2 \times 2 \times 16 \frac{1}{r^2 - \pi^2} \times \frac{1}{(r^2 - \pi^2)^3} = 4 \times \left( \frac{4}{(r^2 - \pi^2)^2} \right)^2 < 4 \, \K_1^2,
	\]
	where we have used \eqref{32ei3} in the last inequality, which proves the first assertion of \eqref{32ei5}. To prove the second one, it suffices to employ \eqref{32ei2} and the following:
	\begin{align}  \label{psi_sum}
		5 \K_1 + \K_2 \, r^2 = 4 \, \frac{\psi'(r)}{r} + \psi''(r) = 2 \, \psi(r) \, \frac{r^2\csc^2r - 1}{r^2}, \quad 0 < r \neq \pi < \vartheta_1,
	\end{align}
	which is due to the trivial facts that
	\begin{equation*}
		\psi'(r)=\frac{r^2 \csc ^2r+r \cot r-2}{r^3},\quad \psi''(r)=-2\,\frac{r^3\cot r\csc ^2r + r^2 \csc ^2r+r\cot r-3}{r^4}.
	\end{equation*}
	This finishes the proof of Lemma \ref{Ll34}.
\end{proof}

Recall that (cf. \eqref{nABn3}) $h(r) =  r^2 + r \sin{r} \cos{r} - 2 \sin^2{r}$ for $r > 0$. Notice that $h(j \, \pi) = (j \, \pi)^2$ ($j \in \N^*$) and $h(r) = \psi'(r) \, r^3 \sin^2r$ for $r \not\in \{ j \, \pi; \, j \in \Z^*\}$. A direct consequence of \eqref{32ei3} is the following:

\begin{cor} \label{nCn1}
	We have $h(r) > 0$ for all $r > 0$.
\end{cor}

Now, we give the

\subsection{Proof of the two key analytic-diffeomorphisms}

\begin{proof} There are four steps.
	
	{\em Step 1. The Jacobian determinant of $\Lz$ is positive on $\Oz_{+,1}$ and negative on $\Oz_{-,4}$, respectively.}
	Notice that (cf. \eqref{32N96}) for $0<r = |(v_1, v_2)|\neq\pi<\var$ we have
	\begin{align*}
		\Lambda(v_1, v_2) = \nabla_v[v_2^2 \,\psi(|v|)] = \left(\frac{\psi'(r)}{r} \, v_1  v_2^2,  \ v_2 \left[ \frac{\psi'(r)}{r} \, v_2^2 + 2 \, \psi(r) \right] \right).
	\end{align*}
	Then  the Jacobian matrix of $\Lambda$ at $v = (v_1, v_2)$, saying $\JA(v, r)$, equals by a routine calculation
	\begin{align} \label{nJAn}
		\left(
		\begin{array}{cc}
			v_2^2 \left[ \frac{\psi'(r)}{r} + \left( \frac{\psi'(r)}{r} \right)' \frac{v_1^2}{r} \right] \qquad & \qquad v_1 v_2 \left[ 2 \frac{\psi'(r)}{r} + \left( \frac{\psi'(r)}{r} \right)' \frac{v_2^2}{r}  \right] \\
			\mbox{} \qquad & \qquad \mbox{} \\
			v_1 v_2 \left[ 2 \frac{\psi'(r)}{r} + \left( \frac{\psi'(r)}{r} \right)' \frac{v_2^2}{r}  \right] \qquad & \qquad  2 \psi(r) + v_2^2 \left[ 5 \frac{\psi'(r)}{r} + \left( \frac{\psi'(r)}{r} \right)' \frac{v_2^2}{r} \right]  \\
		\end{array}
		\right).
	\end{align}
	Notice that $\JA(v, r) = \He_v  \left( v_2^2 \, \psi(|v|) \right) $, which is positive definite on $\Omega_{+,1}$ by Lemma \ref{Hetau1}. This in particular gives the first claim. For the second one, recalling (see \eqref{32ei3}, \eqref{32ei4}, and \eqref{nK3N})
	\begin{align*}
		\K_1 := \frac{\psi'(r)}{r}, \quad \K_2 := r^{-1} \left( \frac{\psi'(r)}{r} \right)', \quad \K_3 := 2 \psi(r) + \K_1 v_2^2,
	\end{align*}
	we find
	\begin{equation*}
		\det \JA(v, r) = v_2^2 \left\{ (\K_1 + \K_2 v_1^2) \, [\K_3 + v_2^2 (4 \K_1 + \K_2 v_2^2)] - v_1^2 (2 \K_1 + \K_2 v_2^2)^2 \right\} := v_2^2 \, \K,
	\end{equation*}
	with the observation that
	\begin{align}
		\K(v_1,v_2) &= (2 \psi(r) + \K_1 v_2^2) \cdot (\K_1 + v_1^2 \K_2) + 4 \K_1^2 v_2^2 + \K_1 \K_2 v_2^4 - 4 \K_1^2 v_1^2  \nonumber \\[1mm]
		&= 2  \psi(r) \K_1 + v_1^2 \left[ 2 \psi(r) \K_2 - 4 \K_1^2 \right] + v_2^2 \, \K_1 \left( 5 \K_1 + \K_2 \, r^2 \right). \label{32nK0}
	\end{align}
	Hence from \eqref{32ei3}, \eqref{32ei2}, and \eqref{32ei5}, we conclude that $\det \JA(v, r)$ is negative on $\Oz_{-,4}$.
	
	{\em Step 2. The analytic map $\Lambda$ is a diffeomorphism from $\Omega_{+, 1}$ onto $\R^2_{>,+}$. } First we claim that $\Lambda$ is from $\Omega_{+, 1}$ into $\R^2_{>,+}$, namely,
	\begin{align*}
		v_2 \left( 2 \psi(r) + \psi'(r) \frac{v_2^2}{r} \right) > \frac{2}{\sqrt{\pi}} \sqrt{\psi'(r) \frac{v_2^2}{r}  v_1}, \quad v_1, v_2 > 0 \mbox{  with } r = \sqrt{v_1^2 + v_2^2} < \pi.
	\end{align*}
	This is a direct consequence of the inequality $\psi(r) > \sqrt{\frac{\psi'(r)}{r}}$ for  $0 \leq r < \pi$; see \eqref{iN38}.
	
	Notice that $\Omega_{+ , 1}$ and $\R^2_{>,+}$ are both connected and simply connected and we have shown the Jacobian determinant of $\Lambda$ vanishes nowhere in $\Omega_{+, 1}$. Then by Hadamard's theorem (see for example \cite[\S~ 6.2]{KP02}), it remains to prove that $\Lambda$ is proper, that is, whenever $\{ v^{(j)} \}_{j = 1}^{+\infty} \subseteq \Omega_{+, 1}$ satisfies $v^{(j)} = (v^{(j)}_1, v^{(j)}_2)  \to \partial \Omega_{+, 1}$ then $\Lambda(v^{(j)})  \to \partial \R^2_{>,+}$. To see this, by contradiction, assume that $\Lambda$ is not proper. Then we can find a sequence $\{ v^{(j)} \}_{j = 1}^{+\infty} \subseteq \Omega_{+, 1}$ such that $v^{(j)} \to \partial \Omega_{+, 1}$ but $\{\Lambda(v^{(j)})\}_{j = 1}^{+\infty}$ stays in a compact set in  $\R^2_{>,+}$. After picking subsequences, we can assume further that $v^{(j)} \to v^{(0)} \in \partial \Omega_{+, 1}$ and $\Lambda(v^{(j)}) \to  (a_1, a_2) \in \R^2_{>,+}$. Recalling the definition of  $\psi$ (see \eqref{EFs}), we always obtain a contradiction in each of the following 4 possible cases:
	
	\begin{enumerate}[(i)]
		
		\item If $ v^{(0)} \in [0, \pi) \times \{0\} $, then $\Lambda(v^{(j)})  \to 0 \in  \partial \R^2_{>,+}$.
		
		\item If $ v^{(0)} \in \{v; \, |v| = \pi, v_1 \ge 0, v_2 > 0\}$,  by the fact that (cf. \eqref{N32ei1})
		\begin{align} \label{nnp}
			\lim_{r \to \pi^-} (\pi - r) \, \psi(r) = \frac{1}{\pi}, \qquad \lim_{r \to \pi^-} (\pi - r)^2 \, \psi'(r) = \frac{1}{\pi},
		\end{align}
		we have  $\Lambda(v^{(j)})  \to \infty \in  \partial \R^2_{>,+}$.
		
		\item If $v^{(0)} \in \{0\} \times (0,\pi) $, then $\Lambda(v^{(j)})  \to \{ 0 \} \times (0, +\infty)   \subseteq  \partial \R^2_{>,+}$.
		
		\item Assume $ v^{(0)} = (\pi,0) \in \partial \Omega_{+, 1}$. Set $\Lambda(v^{(j)})  = (u^{(j)}_1, u^{(j)}_2)$. Then we have
		\begin{align*}
			a_1 = \lim_{j \to +\infty} u^{(j)}_1 =  \lim_{j \to +\infty}  \frac{v^{(j)}_1}{|v^{(j)}|} \left( v^{(j)}_2 \right)^2 \psi'(|v^{(j)}|)   =  \lim_{j \to +\infty}  \frac{1}{\pi} \left( \frac{v^{(j)}_2}{\pi - |v^{(j)}|} \right)^2,
		\end{align*}
		where we have used \eqref{nnp} in the last equality. Using \eqref{nnp} again, it
		implies that
		\begin{align*}
			\lim_{j \to +\infty} u^{(j)}_2 &=  \lim_{j \to +\infty}  \left[ 2  v^{(j)}_2 \psi(|v^{(j)}|) + \frac{v^{(j)}_2}{|v^{(j)}|} \left( v^{(j)}_2 \right)^2 \psi'(|v^{(j)}|)  \right] \\
			&=  \lim_{j \to +\infty}  \frac{2}{\pi} \frac{v^{(j)}_2}{\pi - |v^{(j)}|}  = \frac{2}{\sqrt{\pi}} \sqrt{a_1},
		\end{align*}
		and hence $ (a_1, \frac{2}{\sqrt{\pi}} \sqrt{a_1}) \in \partial \R^2_{>,+}$.
	\end{enumerate}
	
	{\em Step 3. $\Lambda$ is a diffeomorphism from $\Omega_{-, 4}$ onto $\R^2_{<, +}$.} First, we need to show $\Lambda$ is from $\OC$ into $\RS$, i.e,
	\begin{align*}
		2 \psi(r) +  \frac{\psi'(r)}{r} v_2^2 > - \frac{2}{\sqrt{\pi}} \sqrt{\frac{\psi'(r)}{r} v_1}, \quad \forall \, (v_1, v_2) \in \OC.
	\end{align*}
	Using \eqref{32ei3}, it remains to prove that $\psi(r) > - \sqrt{\frac{\psi'(r)}{r}}$, which is obvious by \eqref{32ei2} and \eqref{32ei3}.
	
	Using an argument similar to that in  {\it Step 2}, it is easily seen that the smooth function $\Lambda$ is proper. Moreover, we have obtained in {\it Step 1} that the Jacobian determinant of $\Lambda$ vanishes nowhere on $\Omega_{- , 4}$. From these and Hadamard's theorem the desired conclusion follows, since both $\Omega_{- , 4}$ and $\R^2_{<,+}$ are connected and simply connected.
	
	{\em Step 4. The $3 \times 3$ matrix ${\He}_{\theta} \, \phi(g_u ; \theta)$ is nonsingular, negative definite on $\Omega_{+,1}$, and has exactly two positive eigenvalues on $\Omega_{-,4}$}. Indeed, setting $\tilde{\theta} = \left(\theta_1, \theta_2\right) := \Lambda^{-1}(\tilde{u})$, then a direct calculation gives
	\begin{equation} \label{Hes_phi}
		{\He}_{\theta} \, \phi(g_u ; \theta)=-\left(\begin{array}{cc}
			\mathrm{J}_{\Lambda}(\widetilde{\theta},|\theta|)  & \mathbb{O}_{2 \times 1} \\[1mm]
			\mathbb{O}_{1 \times 2} & 2 \psi(|\theta|)+\theta_2^2 \frac{\psi^{\prime}(|\theta|)}{|\theta|}=\K_3(\tilde{\theta})
		\end{array}\right).
	\end{equation}
	Note that we have $\K_3(\tilde{\theta})>0$ on $\Oz_{+,1}$ by \eqref{32ei3}, and $\K_3(\tilde{\theta}) < 0$ on $\Oz_{-,4}$ (cf. \eqref{O-4}). Moreover, from {\it Step 1} we know the matrix $\mathrm{J}_{\Lambda}(\widetilde{\theta},|\theta|)$ is positive definite on $\Oz_{+,1}$ and $\mathrm{det\,J}_{\Lambda}(\widetilde{\theta},|\theta|) =\tz_2^2\, \K(\tilde{\theta}) <0 $ on $\Omega_{-,4}$. Then the claim of {\em Step 4} follows easily and the proof of Theorem \ref{mapL} is completed.
\end{proof}

The above proof and the fact that $\He_\theta \, \phi(g; \theta) = |x|^2 \, \He_\theta \, \phi(g_u; \theta)$ give the following lemma immediately.
\begin{lem}
	Let $g$ and $\tz$ be given as in Assumption (A) (cf. \eqref{n78n}). We have
	\begin{equation} \label{detph}
		\det(- \He_\theta \,\phi(g;\theta)) = \tz_2^2\, \K(\tz_1,\tz_2) \,\K_3(\tz_1,\tz_2) \, |x|^6.
	\end{equation}
\end{lem}

\begin{remark}\label{Lrmk}
Suppose Assumption (A)  holds. Then from \eqref{32N96} it follows that
	\begin{gather}
		u_1\sim\tz_1\, \tz_2^2\, |\pi-|\tz|\,|^{-2};\label{u1sim}\\[1mm]
		u_2\sim|\tz_2|^3 \, |\pi-|\tz|\,|^{-2}+|\tz_2|\,|\pi-|\tz|\,|^{-1},\quad {\rm as}\,\,|\tz|<\pi;\label{u2sim1}\\[1mm]
		\tz_1\sim1 \,\,{\rm and}\,\, \frac{\ep}{|\pi-|\tz|\,|} \sim 2\,|\psi(|\theta|)| = \frac{u_1}{\theta_1} + \frac{u_2}{|\tz_2|} \sim \frac{\tz_2^2}{(\pi - |\tz|)^2} + \frac{u_2}{|\tz_2|},\quad {\rm as}\,\,|\tz|>\pi,\label{u2sim2}
	\end{gather}
where we have used   \eqref{mapL_p} with $r = |\theta|$.  These estimates will be used in the proof of Lemma \ref{asyin} below.
\end{remark}

Recall that $\mathrm{K}_3$ and $\Omega_{-,4}$ are defined by \eqref{nK3N} and \eqref{O-4}, respectively. The following characterization of $\Omega_{-,4}$ will be used in the proof of Theorem \ref{tmm} (cf. Subsection \ref{ss53} below):

\begin{lem}
	We have
	\begin{align} \label{ChaO-4}
		\Omega_{-,4} = \left\{(v_1, v_2); \, v_2 < 0 < v_1, \ \mathrm{K}_3(v_1, v_2) < 0, \ \pi \neq r  = \sqrt{v_1^2 + v_2^2}  < \vartheta_1 \right\}.
	\end{align}
\end{lem}

\begin{proof}
	Indeed, $\Omega_{-,4}$ is obviously contained in RHS. Conversely, given a point $(v_1,v_2)$ in RHS, since it deduces from \eqref{N32ei1} that
		\begin{align*}
			0>\mathrm{K}_3\left(v_1, v_2\right) &=4\left[\sum_{j=1}^{+\infty} v_2^2\left((j \pi)^2-r^2\right)^{-2}+\sum_{j=2}^{+\infty}\left((j \pi)^2-r^2\right)^{-1}-\frac{1}{r^2-\pi^2}\right]\\
			&>4\left[v_2^2\, (\pi^2-r^2)^{-2}-\frac{1}{r^2-\pi^2}\right]=4 \,\frac{\pi^2-v_1^2}{\left(r^2-\pi^2\right)^2},
		\end{align*}
		then $r>v_1>\pi$ whenever $v_1>0>v_2$, which implies that $(v_1,v_2)\in\Omega_{-,4}$.
\end{proof}

The expression of the function $\K$ (cf. \eqref{32nK0}) in fact can be simplified as the following lemma shows, which will be used in Section \ref{sec11}.
\begin{lem} \label{K_form}
	Suppose that $v  = (v_1, v_2)  \in \rr^2$ with $0<|v|\neq\pi < \var$. Then
	\begin{align}\label{32nK}
		\K(v_1,v_2) = \frac{2}{\Upsilon(|v|)^3 \, |v|} \left[ \frac{-\Upsilon''(|v|)}{|v|} v_1^2 +  \frac{-\Upsilon'(|v|)}{\sin^2(|v|)} v_2^2 \right].
	\end{align}
\end{lem}

\begin{proof} Write $r:=|v|$. Then from \eqref{32nK0} and a routine computation it follows that
	\begin{align}
		\K(v_1,v_2) &= 2  \psi(r) \K_1 + v_1^2 \left[ 2 \psi(r) \K_2 - 4 \K_1^2 \right] + v_2^2 \, \K_1 \left( 5 \K_1 + \K_2 \, r^2 \right) \nonumber\\[1mm]
		&=2 v_1^2 \left[ \frac{\psi(r)}{r^2} \K_1 + \psi(r) \K_2 - 2 \K_1^2 \right] + v_2^2 \, \K_1 \left[ 2 \frac{\psi(r)}{r^2} + 5 \K_1 + \K_2 \, r^2 \right] \nonumber\\[1mm]
		&\xlongequal{\eqref{psi_sum}}2 v_1^2 \left[ \psi(r)\,\left(\frac{\K_1}{r^2} + \K_2 \right) -2 \frac{\psi'(r)^2}{r^2}\right] + v_2^2 \, \frac{\psi'(r)}{r}\left[2\frac{\psi(r)}{r^2} + 2\frac{\psi(r)}{r^2} (r^2\csc^2 r -1) \right] \nonumber\\[1mm]
		&=2 v_1^2 \left[\frac{\psi(r)\psi''(r)-2\psi'(r)^2}{r^2} \right] + 2 v_2^2 \left[ \frac{\psi(r)\psi'(r)}{r \sin^2 r}\right] \nonumber\\[1mm]
		&\xlongequal{\eqref{relUp}}\frac{2}{\Upsilon(r)^3 \, r} \left[ \frac{-\Upsilon''(r)}{r} v_1^2 +  \frac{-\Upsilon'(r)}{\sin^2(r)} v_2^2 \right] \nonumber.
	\end{align}
	This is the desired result.
\end{proof}

\section{Proof of Proposition \ref{nPn1} and Theorem \ref{tmm} } \label{s4}
	\setcounter{equation}{0}
	In this section, we will show the more useful integral expression for the heat kernel \eqref{ehk4}, which plays a role in overcoming our major obstacle in Section \ref{bigchr} below. Furthermore, we give some basic properties for the ``intrinsic distance'' $\rD$, as well as the proof of Theorem \ref{tmm}. Let us begin with the
	
	\subsection{Proof of Proposition \ref{nPn1}} \label{sec51}
	
	\begin{proof}
		Let $\vartheta_k \, (k \in \N^*)$ denote the unique solution of $\tan{r} = r$ in $(k \pi, (k + \frac{1}{2}) \pi)$. The following equality, which is a counterpart of \eqref{defV} for $r/\sinh{r}$, comes from the property of the Bessel function $J_{\frac{3}{2}}$ (cf. \cite[\S~8.544, \S~8.464.3]{GR15})
		\begin{align}\label{J32}
			\frac{r^3}{r \cosh{r} - \sinh{r}} &= 3\, \prod_{k = 1}^{+\infty} \left( 1 + \frac{r^2}{\vartheta_k^2} \right)^{-1}.
		\end{align}
		Then taking logarithmic derivative of the equality above, a direct calculation gives:
		\begin{align}\label{J33}
			\frac{r^2}{r \coth{r} - 1} &= 3 + 2 \, \sum_{k = 1}^{+\infty} \frac{r^2}{\vartheta_k^2 + r^2} =: \widetilde{\Upsilon}(r).
		\end{align}
		
		Next, let $A$ be a real, positive definite $q \times q$ matrix, and $Y \in \mathbb{C}^q$. Recall the well-known formula (cf. for instance \cite[Theorem~7.6.1]{H90})
		\medskip
		\begin{align}\label{Hor90}
			\int_{\R^q} e^{-\frac{1}{2} \langle A \, s, \, s \rangle + i \langle Y, \, s \rangle} ds = (2\pi)^{\frac{q}{2}} \, \frac{1}{\sqrt{\det{A}}} \,  e^{-\frac{1}{2} \langle A^{-1} \, Y, \, Y \rangle}.
		\end{align}
		\medskip
		Now applying this formula with $q=2$, $Y = \frac{|x|^2 \, \ww}{2} (\lz_2, \lz_3)$ and $A = \frac{1}{2} \frac{|\lz|^2}{|\lz|\coth|\lz|-1} |x|^2 \, \ww^2 \, \I_2$ to \eqref{ehk2'}, we obtain the desired result
		\begin{align*}
			p(g) = \frac{1}{4\pi} \, |x|^2 \, \ww^2\, e^{-\frac{|x|^2}{4}} \int_{\R^2} \rP\left( s|x|\,\ww,\frac14|x|^2(u + 2\ww\, s_1 \, e_2 + 2\ww\, s_2 \, e_3)\right) \, ds,
		\end{align*}
		where
		\begin{align}\label{defP2}
			\rP(X,T) = \int_{\R^3}  \, \vv(\lambda) \, e^{-\frac{1}{4} \Ga((X,T);\lambda)} \, d\lambda,
		\end{align}	
		with
		\begin{gather}
			\vv(\lambda) = \frac{|\lambda|^3}{|\lambda| \cosh{|\lambda|} - \sinh{|\lambda|}} = 3\, \prod_{k = 1}^{+\infty} \left( 1 + \frac{\lambda \cdot \lambda}{\vartheta_k^2} \right)^{-1}, \label{defcV} \\
			\Ga((X,T);\lambda) = \widetilde{\Upsilon}(|\lambda|) \, |X|^2  - 4i \,  T \cdot \lambda = 3|X|^2 + \sum_{k = 1}^{+\infty} \frac{2 |\lambda|^2}{\vartheta_k^2 + |\lambda|^2} |X|^2 - 4i \, T \cdot \lambda. \label{MFF2}
		\end{gather}
		
		Notice that \eqref{defP2}-\eqref{MFF2} have been defined in Subsection \ref{ideaS}.
	\end{proof}
	
	\subsection{The ``intrinsic distance'' $\rD$ associated to $\rP$} \label{sec52}
	
	We are in a position to provide some basic properties of the squared ``intrinsic distance'' associated to $\rP$,
	\begin{align*}
		\rD(X,T)^2 := \sup_{\tau \in B_{\R^3}(0, \vartheta_1)}   \Gamma((X,T);\tau), \qquad (X,T) \in \R^2 \times \R^3,
	\end{align*}
	where the smooth function $\Gamma((X,T);\cdot)$ is defined by
	\begin{align} \label{MFF3}
		\Gamma((X,T);\tau) : = \Ga((X,T);i\tau) = \Upsilon(|\tau|) |X|^2 + 4T \cdot \tau , \qquad \tau \in B_{\R^3}(0, \vartheta_1),
	\end{align}
	with
	\begin{align}\label{deD}
		\Upsilon(r):= \frac{r^2}{1 - r \cot{r}} (= \frac{1}{\psi(r)}, \ r \neq \pm \pi) = 3 - 2 \sum_{k = 1}^{+\infty} \frac{r^2}{\vartheta_k^2 - r^2}, \,\,   - \vartheta_1 < r < \vartheta_1,
	\end{align}
	where we have used \eqref{J33} in the last equality with a complexification. Hence $-\Upsilon$ is operator convex on $(-\var, \  \var)$ in the sense of \cite{B97}.
	
From the above series expansion,  one has immediately
\begin{equation}
	\begin{gathered}\label{Ups_ser}
		\Upsilon^\prime(r) = - 4 \sum_{k=1}^{+\infty} \frac{\vtz_k^2 \, r}{(\vtz_k^2-r^2)^2},\quad \Upsilon^{\prime\prime}(r) = - 4 \sum_{k=1}^{+\infty} \frac{\vtz_k^2 \,  (\vtz_k^2+3r^2)}{(\vtz_k^2-r^2)^3} \ (< 0), \\
		\Upsilon^{\prime\prime\prime}(r) = - 48 \sum_{k=1}^{+\infty} \frac{\vtz_k^2 \, (\vtz_k^2+r^2) \, r}{(\vtz_k^2-r^2)^4}, \quad r\in(-\vtz_1,\vtz_1).
	\end{gathered}
\end{equation}
In particular, we get the following counterpart of the key function $\mu(r) = (2 r - \sin{2 r})/(2 \sin^2{r})$ ($-\pi < r < \pi$) in the setting of Heisenberg group as well as generalized Heisenberg-type groups (cf. \cite[Lemma~3, p. 112]{G77}, \cite[Theorem~1.36]{BGG00} and \cite[Section 7]{LZ22}):
\begin{lem} \label{lmDp}
	The function $r \mapsto - \Upsilon^\prime(r)$ is an odd increasing diffeomorphism between $(-\vartheta_1, \vartheta_1)$ and $\R$.
\end{lem}
We denote the smooth inverse function of $-\Upsilon'(\cdot)$ by $\cZ(\cdot)$, that is,
\begin{equation} \label{cZd}
	-\Upsilon^\prime(\cZ(\rho)) := \rho, \quad \rho \in \R.
\end{equation}And we define a smooth function $\Phi(\cdot)$ on the real line by setting
\begin{align} \label{defUps}
	\Phi(\rho) := \Upsilon(\cZ(\rho)) + \rho \, \cZ(\rho),  \quad  \rho \in \R.
\end{align}
The following result (especially the assertions  (ii) and (iii) below) says that $\rD$ can be imagined as a Carnot--Carath\'eodory distance on $\R^2 \times \R^3$. Indeed it comes from the same spirit of the characterization for the squared Carnot--Carath\'eodory distance on GM-groups, even on GM-M\'etivier groups (cf. \cite{Li20} and \cite{LZ22}). So one can consider that we are in a special case of GM-groups.

\begin{prop}\label{lD} It holds that:
	{\em\begin{compactenum}[(i)]
			\item The smooth function $\Gamma((X,T); \cdot)$ is concave in $B_{\R^3}(0,\vartheta_1)$.
			
			\item  The function $\rD(X,T)^2$ has the following explicit expression:
			\begin{align}\label{eD3}
				\rD(X,T)^2 =\begin{cases}
					4\vartheta_1 |T|, & \mbox{if \ } |X| = 0,  \\
					\Upsilon(|\tau^*|)|X|^2 + 4T \cdot \tau^* = \Phi\left( \frac{4|T|}{|X|^2} \right) |X|^2 , & \mbox{if \ } |X| \ne 0,
				\end{cases}
			\end{align}
			where $\tau^* := \tau^*(X,T)$ denotes the unique critical point of $\Gamma((X,T); \cdot)$ in $B_{\R^3}(0,\vartheta_1)$, i.e.,
			\begin{align}\label{Dts}
				\frac{\Upsilon^\prime(|\tau^*|)}{|\tau^*|} |X|^2 \, \tau^* + 4T = 0, \quad {\rm so} \quad \tau^* = \cZ\left( \frac{4|T|}{|X|^2} \right) \, \widehat{T},
			\end{align}
			with the convention $\widehat{T} = 0$ for $T = 0$, and $\widehat{T} = T/|T|$ otherwise.
			\item $\rD(X, T)^2 > 0$ for any $(X,T) \ne (0,0)$.  Moreover, the following properties hold
			for all $(X, T) \in \R^2 \times \R^3$ :
			\begin{equation} \label{D_scal}
				\rD(X,T)^2 \sim |X|^2 + |T|, \qquad	\rD\left( \frac{X}{\sqrt{h}}, \frac{T}{h} \right)^2 = \frac{1}{h} \, \rD(X,T)^2, \qquad \forall \, h > 0.
			\end{equation}
			
			\item $\rD^2$ is continuous on $\R^2 \times \R^3$.
			
	\end{compactenum}}
\end{prop}

\begin{proof}  For the reader’s convenience, we provide a direct proof.
	
	First, the concavity can be directly verified by showing that $\He_{\tau} \Gamma((X, T), \tau) \le 0$ from the series expression \eqref{deD}. Now we prove (ii). Recall that for a smooth concave function, the maximizer is equivalent to the critical point (cf. \cite[IMPORTANT, p.146]{NP18}). Through a simple computation we have
	\begin{align*}
		\rD(X,T)^2 = \begin{cases}
			4\vartheta_1 |T|, & \mbox{if \ } |X| = 0 , \\
			\Upsilon(|\tau^*|)|X|^2 + 4T \cdot \tau^*, & \mbox{if \ } |X| \ne 0,
		\end{cases}
	\end{align*}
	where $\tau^*$ is defined by \eqref{Dts} (this is clearly consistent with the one defined in Proposition \ref{lP} (iii)). From this and \eqref{defUps} we obtain (ii).
		
		Next we show (iii). It follows from \eqref{deD} that $\Upsilon(r) \le 3$. Hence \eqref{MFF3} and the definition of $\rD^2$ yield that $\rD(X, T)^2 \le 3 |X|^2 + 4 \var \, |T|$. Next, notice that $\Gamma((X, T); 0) = 3 |X|^2$, and $\Gamma((X, T); \pi \, \widehat{T}) = 4 \pi |T|$ since $\Upsilon(\pi) = 0$ (cf. the first equality in \eqref{deD}). In conclusion, we get that $\rD(X, T)^2 \sim |X|^2 + |T|$, which implies immediately the first claim.  On the other hand, the scaling property $\rD(X/\sqrt{h}, T/h)^2 = \rD(X, T)^2/h$ is clear from the definition of $\rD^2$.
		
			We are left with the proof of (iv). Notice that the continuity of $\rD^2$ at $(X_0,T_0)$ where $X_0\neq0$ follows from \eqref{eD3} and the smoothness of $\Phi$. And the continuity at $(X_0,T_0)=(0,0)$ follows easily from the first formula of \eqref{D_scal}. For the point $(X_0,T_0)=(0,T_0)$ where $T_0\neq0$, initially by \eqref{eD3} we have $\rD(0,T_0)^2 = 4\var |T_0|$. Supposing $(X,T)\to(0,T_0)$, we consider two cases. If $X=0$, then by the first equality in \eqref{eD3} we have $\rD(X,T)^2=4\var|T|\to \rD(X_0,T_0)^2$. If $X\neq0$, by \eqref{Dts} we have $|X|^2 = -4|T|/\Upsilon'(|\tau^*|)$ and $|\tau^*|=\cZ(4|T|/|X|^2) \to \vartheta_1^-$ (from Lemma \ref{lmDp} and the fact that $|T|/|X|^2 \to +\infty$). Then  noticing that $\Upsilon(r) \sim - \vartheta_1/(\vartheta_1 - r)$ and $\Upsilon'(r) \sim - \vartheta_1/(\vartheta_1 - r)^2$ for $0 < \vartheta_1 - r$ small enough, by \eqref{eD3} we obtain
				\[
				\rD(X,T)^2 = 4|T|\left( |\tau^*|- \frac{\Upsilon(|\tau^*|)}{\Upsilon'(|\tau^*|)} \right) \to 4\var |T_0| = \rD(X_0,T_0)^2
				\]
				as well.
		\end{proof}

\begin{remark}\label{smoD}
Combining \eqref{eD3} with \eqref{aPsi} below we can see that $\rD^2$ is not $C^1$ at $\{(0,T); \, T \ne 0\}$. Moreover, $\rD^2$ is not $C^1$ at $(0,0)$ by the first estimate of \eqref{D_scal}. In conclusion, $\rD^2$ is smooth on $(\R^2 \setminus \{0\}) \times \R^3$ and is not $C^1$ on $\{0\} \times \R^3$.
\end{remark}
		
		Before providing the proof of Theorem \ref{tmm}, we recall that (cf. \eqref{defs}) $\varphi_1(r) = \frac{r^2 - \sin^2{r}}{r - \sin{r} \cos{r}} > 0$ for $r > 0$, and state the following simple observation:
		
		\begin{lem} \label{0zl1}
			$\varphi_1 > 0$ is strictly increasing on $(0, \ +\infty)$.
		\end{lem}
		
		\begin{proof}
			A simple computation shows
			\begin{equation*}
				\varphi_1'(r) = 8\ \frac{(r \cos{r} - \sin{r})^2}{(2 r - \sin(2r))^2}> 0, \quad \forall \, r \in (0, \ +\infty)\setminus\{\vtz_k;k\in\nn^*\},
			\end{equation*}
			which implies the monotonicity of $\varphi_1$.
		\end{proof}
		
		\subsection{Proof of Theorem \ref{tmm}} \label{ss53}
		
		\begin{proof}
			Let us begin with the proof of \eqref{dEn2}.
			Recall that the Assumption (A) (cf. \eqref{n78n}) says that $\ww = \theta_2 \, \psi(|\theta|) > 0$, $\ss = (\ss_1, \ss_2) = (-1, 0)$, and $u = (u_1, u_2, 0)$ with
			\[
			u_1 = \frac{\psi'(|\theta|)}{|\theta|} \, \theta_2^2 \, \theta_1 > 0, \qquad u_2 = \theta_2 \, \left[ \frac{\psi'(|\theta|)}{|\theta|} \, \theta_2^2 + 2 \, \psi(|\theta|) \right] = \theta_2 \, \K_3(\theta_1, \theta_2) > 0.
			\]
			It follows from  $\Upsilon(r) = 1/\psi(r)$ ($0 < r \neq \pi < \var$) that
			\begin{align}\label{crieq}
				\frac{\Upsilon^\prime(|\theta|)}{|\tz|} \, \ww^2 \, |\ss|^2 \, \theta + 4 \times \frac{1}{4} \Big(u + 2 \, \ww \,\ss_1  e_2 + 2 \, \ww \,\ss_2 e_3 \Big)
				= 0.
			\end{align}
			Combining this with \eqref{Dts} and \eqref{eD3}, we obtain that
			\begin{align*}
				\cD(u;\ss) = \Upsilon(|\theta|) \ww^2 |\ss|^2+ u \cdot \theta + 2\ww \,\ss_1 \theta_2 = -\theta_2^2 \, \psi(|\theta|) +   u \cdot \theta,
			\end{align*}
			which is the first equality in \eqref{dEn2}.
			
			Next, observe that $u \cdot \theta = \theta_2^2 \, ( |\theta| \, \psi'(|\theta|) + 2 \, \psi(|\theta|) )$. Then to obtain the second equality in \eqref{dEn2}, it suffices to use the first one and
			\[
			\psi(r) + r \psi'(r) = (r \psi(r))' = \left( \frac{1 - r \cot{r}}{r} \right)' = \frac{1}{\sin^2{r}} - \frac{1}{r^2}.
			\]
			To show the third one, we make use of the second one and the following basic equality
			\[
			\varphi_1(r) = \frac{\varphi_0(r)}{r^2 \psi'(r) + 2r \psi(r)}, \quad
			\mbox{with} \quad
			\varphi_0(r) := \left( \frac{r}{\sin{r}} \right)^2 - 1.
			\]
			For the fourth one, it suffices to use the second one, the fact that $u_1/\theta_1 = \psi'(|\theta|) \, \theta_2^2/|\theta|$ and $\varphi_2(r) = \varphi_0(r)/(r^2 \, \psi'(r))$.
			The last one follows from the fourth one, the fact that
			\[
			u_1 \, ( u_1 + u_2 \frac{\theta_2}{\theta_1} ) = u_1^2 \frac{|\theta|^2 \, \psi'(|\theta|) + 2 \, |\theta| \, \psi(|\theta|)}{\psi'(|\theta|) \, \theta_1^2} = \left( u_1 \frac{|\theta|}{\theta_1} \right)^2 \frac{\frac{\varphi_0(|\theta|)}{\varphi_1(|\theta|)}}{|\theta|^2 \, \psi'(|\theta|)} = \left( u_1 \frac{|\theta|}{\theta_1} \right)^2 \frac{\varphi_2(|\theta|)}{\varphi_1(|\theta|)}
			\]
			and $\varphi_3 = \sqrt{\varphi_1 \, \varphi_2}$.	
			
			Turning to the first assertion. Fix $u$ as in Assumption (A). Notice that from Proposition \ref{lD} (iii)-(iv), the continuous function $\cD(u;\cdot)$ must attain its minimum at some point $s^* \in \R^2$. Then we only need to show that $s^*=\ss$. This proceeds in three steps as follows:

			\paragraph{Step 1: It holds that $s^* = (s_1^*, s_2^*) \ne 0$.} We argue by contradiction. If $s^* = 0$,
			it follows from \eqref{eD3} that $\cD(u;0) = \rD(0, u/4)^2 = \vartheta_1 \, |u|$. Next, recalling $\tan{\var} = \var$, a direct calculation yields that
			\[
			\varphi_1(\var) = \frac{\var^2 - \sin^2{\var}}{\var - \sin{\var} \, \cos{\var}} = \frac{\tan^2{\var} - \sin^2{\var}}{\tan{\var} - \sin{\var} \, \cos{\var}} = \tan{\var} = \var.
			\]
			Thus the fact that $\varphi_1$ is strictly increasing on $(0, +\infty)$ (cf. Lemma \ref{0zl1})  implies that
			$\cD(u;0) =  \varphi_1(\vartheta_1) \, |u|>\frac{\varphi_1(|\theta|)}{|\tz|} u\cdot\tz = \cD(u;\ss)$. This leads to a contradiction.
			
			It follows from \eqref{eD3} and \eqref{Dts} that $\rD^2$ is smooth on $(\R^2 \setminus \{ 0 \}) \times \R^3$. Then $\cD(u; \cdot) \in C^{\infty}(\R^2 \setminus \{ 0 \})$.
			Since its minimum point $s^* \ne 0$, using \eqref{eD3} again, we obtain
			\begin{align}\label{gras}
				0= \partial_{s_1} \cD(u;s^*) = 2 \, \ww \, ( \Upsilon(|\tau_*|) \, \ww \, s^*_1 + \tau_{*2}), \quad
				0= \partial_{s_2} \cD(u;s^*) = 2 \, \ww \, ( \Upsilon(|\tau_*|) \, \ww \, s^*_2 + \tau_{*3}),
			\end{align}
			with $\tau_*= (\tau_{*1}, \tau_{*2}, \tau_{*3}) = \tau(s^*)$, where the smooth mapping $\tau(s) = (\tau_1(s),\tau_2(s), \tau_3(s))$ on $\R^2 \setminus \{ 0 \}$ is defined by
			\begin{align}\label{gratau}
				\frac{\Upsilon^\prime(|\tau(s)|)}{|\tau(s)|} \, \ww^2 \, |s|^2 \, \tau(s) + u + 2 \, \ww \, s_1  e_2 + 2 \, \ww \, s_2 e_3 = 0.
			\end{align}
			In particular,
			\begin{align}  \label{nABn2}
				\frac{\Upsilon^\prime(|\tau(s)|)}{|\tau(s)|} \, \ww^2 \, |s|^2 \, \tau_3(s) + 2 \, \ww \, s_2 = 0.
			\end{align}
			
			\paragraph{Step 2: It holds that $|\tau_*| \ne \pi$.} We argue by contradiction again. Assume that $|\tau_*| = \pi$. Observing that $\Upsilon(|\tau_*|) = \Upsilon(\pi) = 0$, then \eqref{gras} implies that $\tau_{*2} = \tau_{*3} = 0$. Inserting this into \eqref{gratau} with $s = s^*$ and noticing $\Upsilon^\prime(\pi) = -\pi$ (which can be checked by \eqref{relUp} and \eqref{nnp} for example), we get that $s^*_2 = 0$ and $\pi u_2^2 = 4 u_1$, which contradicts with our Assumption (A) (cf. \eqref{n78n}).
			
			\paragraph{Step 3: We have $\tau_* = \theta$ and $s^* = \ss$.}
			
			Since $s^* \ne 0$ and $0 < |\tau_*| \ne \pi < \vtz_1$,
			it follows from \eqref{gras} and $\Upsilon = 1/\psi$ that
			\begin{align} \label{nABn1}
				\ww \, s^*_1 = - \psi(|\tau_*|) \, \tau_{*2}, \quad \ww \, s^*_2 = - \psi(|\tau_*|) \, \tau_{*3}, \quad \ww^2 \, |s^*|^2 =  \psi(|\tau_*|)^2 \, \Big( (\tau_{*2})^2 + (\tau_{*3})^2 \Big).
			\end{align}
			Substituting this into \eqref{gratau} with $s = s^*$, together with $\Upsilon' = - \psi'/\psi^2$, we can write
			\begin{align*}
				u =  \frac{\psi^\prime(|\tau_*|)}{|\tau_*|} \Big( (\tau_{*2})^2 + (\tau_{*3})^2 \Big) \, \tau_* + 2 \,  (\tau_{*2} \ e_2 + \tau_{*3} \ e_3) \,  \psi(|\tau_*|).
			\end{align*}
			The equation together with the fact that $\psi'(r)/r > 0$ for $0 < r \ne \pi < \var$ (cf. \eqref{32ei3}), we get that $\tau_{*1} > 0$,  $\tau_{*2} \ne 0$, $\tau_{*3} = 0$  (otherwise we would yield from $u_3 = 0$ that $\frac{\psi^\prime(|\tau_*|)}{|\tau_*|} \Big( (\tau_{*2})^2 + (\tau_{*3})^2 \Big) + 2 \, \psi(|\tau_*|) = 0$, which leads to $u_2 = 0$, a contradiction!),  and $\Lambda(\tau_{*1},\tau_{*2}) = (u_1,u_2)$. In particular, recalling the definition of $\K_3$ (cf. \eqref{nK3N}), we have $u_2 = \tau_{*2} \, \K_3(\tau_{*1}, \tau_{*2}) > 0$.
			
			At this point, \eqref{nABn1} leads to $s_1^* = -\frac{\tau_{*2}}{\ww} \psi(|\tau_*|)$, $s_2^* = 0$  and $\ww^2 |s^*|^2 = \psi(|\tau_*|)^2 \tau_{*2}^2 $.  Taking gradient w.r.t. $s$ on both sides of \eqref{nABn2}, we obtain $\nabla_{s^*} \tau_3(s^*) = (0, -\frac{2}{\ww |s^*|^2} \frac{|\tau_*|}{\Upsilon'(|\tau_*|)})$. Then using the expression of $\partial_{s_2} \cD(u;s^*)$ (cf.  \eqref{gras}), a direct computation gives  $\partial_{s_1s_2}^2 \cD(u;s^*) = 0$ and
			\begin{align*}
				\partial_{s_2}^2 \cD(u;s^*) = 2 \, \ww^2 \,  \Upsilon(|\tau_*|) + 2 \, \ww \, \partial_{s_2} \tau_3(s^*) = \frac{2 \, \ww^2 \, |\tau_*|}{(\tau_{*2})^2 \,  \psi(|\tau_*|) \, \psi^\prime(|\tau_*|) } \, \K_3(\tau_{*1}, \tau_{*2}),
			\end{align*}
			which is non-negative, by the fact that the Hessian matrix at a minimum point  is positive semi-definite.
			
			To summarize, we get that
			\begin{gather*}
				\tau_* = (\tau_{*1},\tau_{*2}, 0), \quad 0 < |\tau_*| \ne \pi < \var, \quad \tau_{*1} > 0, \quad \tau_{*2} \, \K_3(\tau_{*1}, \tau_{*2}) > 0, \\[1mm]
				\psi(|\tau_*|)\,\K_3(\tau_{*1}, \tau_{*2}) > 0, \quad  \Lambda(\tau_{*1},\tau_{*2}) = (u_1,u_2).
			\end{gather*}
			Using the characterization of $\Omega_{-, 4}$ (cf. \eqref{ChaO-4}), one can easily verify that $(\tau_{*1},\tau_{*2})\in\Omega_{+,1}\cup\Omega_{-,4}$.
			Then by Theorem \ref{mapL} we see that $\tau_* = \theta$. Therefore, $s_1^* = -\frac{\tau_{*2}}{\ww} \psi(|\tau_*|) = -1$ and $s_2^* = 0$, namely,
			$s^*=\ss$, which finishes the proof of Theorem \ref{tmm}.
		\end{proof}
		
		Recall that $\ww = \theta_2 \, \psi(|\theta|)$ and $u_2 = \theta_2 \, \K_3(\theta_1, \theta_2)$. The above proof also gives the following:
		
		\begin{cor} \label{C54n}
			Let $\theta, u$ be as in Assumption (A) (cf. \eqref{n78n}). Then
			\begin{align}\label{ps2D}
				\partial_{s_1s_2}^2 \cD(u; \ss) = 0, \ \mbox{and} \quad
				\partial_{s_2}^2 \cD(u; \ss) =  2 \, \frac{\psi(|\tz|) \, |\tz|}{\psi'(|\tz|)} \, \K_3(\tz_1,\tz_2) = 2 \, \frac{u_2  \, \ww \, |\theta|}{\theta_2^2 \, \psi'(|\theta|)} > 0.
			\end{align}
		\end{cor}
		
\section{The complete answer to the Gaveau--Brockett problem on $N_{3,2}$} \label{ssd}
\setcounter{equation}{0}

Now we are in a position to determine the expression for the sub-Riemannian distance from the origin to any given point. As indicated in Introduction and Remark \ref{Rem21}, the main result (i.e., Theorem \ref{RLT1}) has been proved in \cite{Li20,LZ21}, via two different methods. Here our third method use only the heat kernel. Although we can apply our uniform heat kernel asymptotic at infinity, namely Theorem \ref{bigL1} below (combining with Remark \ref{Vara2}), to obtain Theorem \ref{RLT1}, it turns out that a simpler argument in the proof below works as well.  Indeed the argument is closely related to one of basic ideas in \cite{Li20, LZ212}.  Also note that at least in this very broad framework of step-two groups, this method looks much more practical and elementary than the classical one.

\subsection{Proof of Theorem \ref{RLT1}}
\begin{proof}

	Combining the scaling property \eqref{ehk} with \eqref{ehk4}, we obtain that
	\[
	p_h(g_u) = \frac{\mathbf{C} \, \ww^2}{4\pi h^{\frac{11}{2}}} e^{-\frac{1}{4h}} \int_{\R^2}  \rP\left( \frac{s\,\ww}{\sqrt{h}},\frac{1}{4h}(u + 2\ww \, s_1 \, e_2 + 2\ww \, s_2 \,  e_3)\right) \, ds.
	\]
	Then it follows from (ii) of Proposition \ref{lP} and the second equation of \eqref{D_scal} that
	\begin{equation} \label{smhh}
		p_h(g_u) \sim \frac{\ww^2}{h^{4}} \int_{\R^2} \frac{e^{-\frac{\cD(u;s) + 1}{4h}}}{(h + \cD(u;s)) (h^2 + \ww^2 \, |s|^2 \, \cD(u;s))^{\frac{1}{4}}} \, ds.
	\end{equation}
	
	Recall that $\cD(u; s) \ge \max\{\cD(u; \ss), c\,(\ww^2 |s|^2 + u_1)\}$ for some constant $c>0$
 (cf.  Theorem \ref{tmm} and the first equation in \eqref{D_scal}). Then
		\[
		p_h(g_u) \lsim \frac{\ww}{h^{4}} e^{-\frac{\cD(u;\ss) + 1}{4h}} \int_{\R^2} \frac{ds}{(\ww^2 |s|^2 + u_1) |s|} \lesssim {u}^{-\frac{1}{2}}_1 \, h^{-4} \,  e^{-\frac{\cD(u;\ss) + 1}{4h}}.
		\]
On the other hand, for any given $1 > \varsigma_0 > 0$, by the continuity of $\cD(u; \cdot)$, there exists a ball $B_{u,\varsigma_0}$ with center $\ss$ and radius $< 1$ where $\cD(u; s) \le \cD(u; \ss) + \varsigma_0$. Hence
		\[
		p_h(g_u) \gtrsim_{u,\varsigma_0} \frac{\ww^2}{h^{4}} e^{-\frac{\cD(u;\ss) +\varsigma_0 + 1}{4h}} (h + \cD(u; \ss) + 1)^{-1} \left[h^2 + 4 \ww^2 \,  (\cD(u; \ss) + 1) \right]^{-\frac{1}{4}}.
		\]
		
		It follows from Varadhan's formulas that $\cD(u;\ss) + 1 \le d(g_u)^2 \le \cD(u;\ss) +\varsigma_0 + 1$, which implies immediately the desired result.
\end{proof}

Up to a slight modification, the proof below is extracted from  \cite{Li20, LZ21}.

\subsection{Proof of Corollary \ref{RLT2}}

\begin{proof}
We first show (ii). Fix an $\alpha > 0$ and let $g_{\alpha} = (e_1, 4^{-1} (\frac{\alpha^2}{\pi}, \frac{2}{\pi} \alpha, 0))$. From the proof of Theorem \ref{mapL}, we can pick a sequence $\{v^{(j)} = (v_1^{(j)}, v_2^{(j)})\} \subseteq \Omega_{+ , 1}$ such that $v^{(j)} \to (\pi,0) \in \partial  \Omega_{+ , 1}$ and
	\[
	\frac{v_2^{(j)}}{\pi - |v^{(j)}|} \to \alpha, \quad u^{(j)} = \Lambda(v^{(j)}) \to \left(\frac{\alpha^2}{\pi}, \frac{2}{\pi} \alpha\right), \quad\mbox{as}\,\, j \to +\infty.
	\]
	Hence $g_{u^{(j)}} \to g_\alpha$ as $j \to +\infty$.  Since $d(\cdot)^2$ is continuous, it yields from Theorem \ref{RLT1} that
	\[
	d(g_\alpha)^2 = \lim_{j \to +\infty} d(g_{u^{(j)}})^2 = \varphi_1(\pi) \, \frac{\alpha^2}{\pi} \times \frac{\pi}{\pi} + 1 = 1 + \alpha^2.
	\]
	
	To show (iii), as before we first fix a $\beta > 0$.
	Let $0 < \kz < \frac{2}{\sqrt{\pi}} \sqrt{\beta}$, $u(\kz) := (\beta, \kz)$, and $v(\kz) = (v_1(\kz),v_2(\kz)) := \Lambda^{-1} (u_\kz)$. Since $\{u(\kz)\} \subseteq \R^2_{<, +}$, by Theorem  \ref{mapL} (2) and \eqref{32N96} we have $\{v(\kz)\} \subseteq \Omega_{- , 4}$ and
	\begin{gather}
		v_2(\kz) < 0 < \pi < v_1(\kz) < |v(\kz)| < \vartheta_1, \nonumber \\[1mm]
		\beta = \frac{\psi'(|v(\kz)|)}{|v(\kz)|} \, v_1(\kz) \, v^2_2(\kz), \qquad \kz = v_2(\kz) \left( \frac{\psi'(|v(\kz)|)}{|v(\kz)|} \, v^2_2(\kz) + 2 \, \psi(|v(\kz)|) \right). \label{ncin}
	\end{gather}
	
	By the compactness of $\overline{B_{\R^2}(0, \vartheta_1)}$, up to subsequences, we may take $\kz_j \to 0^+$ as $j \to +\infty$  such that the corresponding $v(\kz_j) \to v(0) := (v_1(0), v_2(0))$. Obviously, $v_2(0) \leq 0$ and $\pi \le v_1(0) \le r:=|v(0)| \le \vartheta_1$.
	
	We claim that $v_2(0) \neq 0$, so $v(0) \notin \{(\pi, 0), (\vartheta_1, 0)\}$ and $\pi < r < \vartheta_1$ (by \eqref{ncin} and the fact that $\psi(\vtz_1)=0,\, \psi'(\var)>0$). Were this not the case, it would follow that $v_2(\kz_j) \to 0^-$. Then the first equation in \eqref{ncin} implies that
	\[
	\lim_{j \to +\infty} \frac{\psi'(|v(\kz_j)|)}{|v(\kz_j)|} = +\infty, \quad \mbox{so } \ |v(\kz_j)| \to \pi^+ \  \mbox{and } \ v_1(\kz_j)
	\to \pi^+,
	\]
	since $\pi < v_1(\kz_j) <  |v(\kz_j)| < \vartheta_1$ and $\psi'(\rho) \to +\infty$ ($\pi < \rho < \vartheta_1$) only if $\rho \to \pi^+$  (cf. \eqref{N32ei1}). Notice  that the limit \eqref{nnp} remains valid as $r \to \pi^+$, then using \eqref{ncin} again we infer that
	\begin{align*}
		\lim_{j \to +\infty} \frac{1}{\pi} \left( \frac{v_2(\kz_j)}{|v(\kz_j)| - \pi} \right)^2 = \beta, \quad 0 = - \frac{2}{\pi} \lim_{j \to +\infty} \frac{v_2(\kz_j)}{|v(\kz_j)| - \pi} = \frac{2}{\sqrt{\pi}} \sqrt{\beta} > 0.
	\end{align*}
	This leads to a contradiction.
	
	In conclusion, by the continuity of  $d(\cdot)^2$ and the last equality in \eqref{dEn}, we obtain that $d(g(\beta))^2 =  \varphi_3(r) \, \beta + 1$, where $\pi < r < \vartheta_1$ satisfies
	\begin{align} \label{NICN}
		\beta = \frac{\psi'(r)}{r} \, v_1(0) \, v_2(0)^2, \qquad  \frac{\psi'(r)}{r} \, v_2(0)^2 + 2 \, \psi(r) = 0.
	\end{align}
	That is
	\begin{align*}
		\beta = - 2 \, \psi(r) \sqrt{r^2 - v_2(0)^2} = - 2 \, \psi(r) \, \sqrt{r^2 + 2 \, r \, \frac{\psi(r)}{\psi'(r)}},
	\end{align*}
	which, together with Remark \ref{Rem21} (1), implies (iii).
	
	Noticing that $\rho^2 \, \psi'(\rho) + 2 \rho \, \psi(\rho) = (\rho^2 \, \psi(\rho))'= \mu(\rho)$, then a similar and simpler argument gives (iv).

	Finally we are left with the proof of (i).  Applying the continuity of $d(g)^2$ again with \eqref{orthod}, it follows from the scaling property (cf. \eqref{scap})  and (iv) that
		\[
		d(g_*)^2 = d(0,e_2)^2 = \lim_{\varepsilon \to 0^+} d(\varepsilon e_1, e_2)^2
		= \lim_{\varepsilon \to 0^+} \varepsilon ^2 \, d(g(4/\varepsilon^2)^*)^2
		= \lim_{\varepsilon \to 0^+} \left( \frac{r(\varepsilon)}{\sin{r(\varepsilon)}} \right)^2 \varepsilon ^2,
		\]
		where $r(\varepsilon)$ is the unique solution of $\mu(r(\varepsilon)) = 4/\varepsilon^2$.  From this and the facts that
				\[
				\lim_{\rho \to \pi^-} (\rho - \pi)^2 \mu(\rho) = \pi, \qquad
				\lim_{\rho \to \pi^-} (\rho - \pi)^2 \left( \frac{\rho}{\sin{\rho}} \right)^2 = \pi^2,
				\]
				the assertion (i) follows easily.
			\end{proof}

\section{Uniform asymptotics for the simplest case:  $|\tz|\le 3$ and $\tz_2|x|\to+\infty$} \label{s44}
\setcounter{equation}{0}

Recall that $d(g)^2 = \phi(g; \theta)$. In this section we establish the following theorem:
\begin{theo} \label{asyab}
	Let $|\tz|\le 3
	$ and $\tz_2 \, |x| \ge\nzz$ with $\nzz\gg1$. Then
	\begin{equation}\label{hk_as1}
		p(g) =  (8\pi)^{\frac{3}{2}} \, e^{-\frac{d(g)^2}{4}} \, \cV(i\theta)
		\, \frac{1}{\sqrt{\det(- \He_\theta \, \phi(g;\theta))}} \,  (1 + o_{\nzz}(1)).
	\end{equation}
\end{theo}

Essentially, Theorem \ref{asyab} is a special case of \cite[Theorem~2.2]{Li20} (up to a slight modification), which is based on the method of stationary phase and the operator convexity. For the reader’s convenience, we provide a direct proof.
Let us begin with the following crucial lemma:
	
	\begin{lem} \label{rephi}
		Let $|\tz|\le 3$. Then there exists a constant $c > 0$ such that
		\[
		\Re (\phi(g;\theta - i\lambda) - \phi(g;\theta)) \ge c \, \frac{- \lambda^\T\, \He_\theta \, \phi(g;\theta) \lambda}{1 + |\lambda|^2}, \quad \forall \, \lambda \in \R^3.
		\]
	\end{lem}
	
	\begin{proof}
		In fact, from the definition of the reference function (cf. \eqref{defref2}) and \eqref{EFs}, we have
		\[
		\phi(g;\theta - i\lambda) = |x|^2 - 2 \sum_{j = 1}^{+\infty} \frac{(\theta_2 - i \lambda_2)^2 + (i\lambda_3)^2}{j^2 \pi^2 - (|\theta|^2 - 2 i \lambda \cdot \theta - |\lambda|^2) } |x|^2 + 4t \cdot (\theta - i \lambda).
		\]
		Hence $\Re (\phi(g;\theta - i\lambda) - \phi(g;\theta))/|x|^2$ equals
		\begin{align*}
			& 2 \sum_{j = 1}^{+\infty}
			\Re \left[ \frac{\lambda_2^2 + \lambda_3^2 - \theta_2^2 + 2i \lambda_2 \theta_2}{j^2 \pi^2 - |\theta|^2 + |\lambda|^2 + 2i \lambda \cdot \theta} +  \frac{  \theta_2^2 }{j^2 \pi^2 - |\theta|^2 }\right]  \\[1mm]
			&\qquad= 2 \sum_{j = 1}^{+\infty}   \Re \left[
			\frac{(\lambda_2^2 + \lambda_3^2 - \theta_2^2 + 2i \lambda_2 \theta_2)(j^2 \pi^2 - |\theta|^2) + \theta_2^2 (j^2 \pi^2 - |\theta|^2 + |\lambda|^2 + 2i \lambda \cdot \theta)}{(j^2 \pi^2 - |\theta|^2 + |\lambda|^2 + 2i \lambda \cdot \theta)(j^2 \pi^2 - |\theta|^2)} \right]  \\[1mm]
			&\qquad= 2 \sum_{j = 1}^{+\infty} \left\{\frac{\theta_2^2 (j^2 \pi^2 - |\theta|^2 + |\lambda|^2) \, |\lambda|^2 + \lambda_2^2 \, |\lambda|^2 \, (j^2 \pi^2 - |\theta|^2)}{[(j^2 \pi^2 - |\theta|^2 + |\lambda|^2)^2 + 4 (\lambda \cdot \theta)^2](j^2 \pi^2 - |\theta|^2)}  \right.\\[1mm]
			&\qquad\qquad +\left.\frac{ \lambda_3^2 (j^2 \pi^2 - |\theta|^2) (j^2 \pi^2 - |\theta|^2 + |\lambda|^2) + [\lambda_2 (j^2 \pi^2 - |\theta|^2) + 2 \theta_2 (\lambda \cdot \theta)]^2 }{[(j^2 \pi^2 - |\theta|^2 + |\lambda|^2)^2 + 4 (\lambda \cdot \theta)^2](j^2 \pi^2 - |\theta|^2)}\right\}.
		\end{align*}
		
		Notice that every term in the series is non-negative, and thus its value is greater than the single term $j_0 = 2$. Thus,
			\begin{align*}
				&\Re (\phi(g;\theta - i\lambda) - \phi(g;\theta))\\[1mm]
				&\qquad\gtrsim  \frac{\theta_2^2 (1 + |\lambda|^2 ) |\lambda|^2 + \lambda_2^2 |\lambda|^2 + \lambda_3^2 (1 + |\lambda|^2) + [\lambda_2 (j_0^2 \pi^2 - |\theta|^2) + 2 \theta_2 (\lambda\cdot\theta)]^2}{(1 + |\lambda|^2)^2} \, |x|^2.
			\end{align*}
			Next, it follows from $j_0^2\pi^2-|\tz|^2\sim1$ and the following elementary inequality
   \[
   (a + b)^2 \ge \frac{a^2}{1 + c} - \frac{b^2}{c}, \qquad a, b \in  \R, \  c > 0,
   \]
   that $\theta_2^2 (1 + |\lambda|^2 ) |\lambda|^2 + [\lambda_2 (j_0^2 \pi^2 - |\theta|^2) + 2 \theta_2 (\lambda\cdot\theta)]^2 \gtrsim \theta_2^2 (1 + |\lambda|^2 ) |\lambda|^2 + \lambda_2^2$.
Consequently,
			\begin{align*}
				\Re (\phi(g;\theta - i\lambda) - \phi(g;\theta)) \gsim \frac{\theta_2^2  \lambda_1^2  + (\lambda_2^2 + \lambda_3^2) }{1 + |\lambda|^2} \, |x|^2.
			\end{align*}
		
		Therefore, to finish the proof of Lemma \ref{rephi}, it suffices to show that for all $|\tz|\le3$,
		\begin{align}\label{Hetau11}
			- \He_\theta \, \phi(g; \theta) \sim \begin{pmatrix}
				\theta_2^2 |x|^2 & 0 & 0 \\
				0 & |x|^2 & 0\\
				0 & 0 & |x|^2 \\
			\end{pmatrix}.
		\end{align}		
		In fact, by \eqref{EFs} and \eqref{32ei3}, $\K_3(\tz_1,\tz_2)\sim1$. Then from this,  \eqref{Hes_phi}, Lemma \ref{Hetau1}, and the fact that $\He_\theta \, \phi(g; \theta) = |x|^2 \, \He_\theta \, \phi(g_u; \theta)$, the estimate \eqref{Hetau11} follows.
	\end{proof}
	
	Next, using the trivial inequality
	\[
	\big| 1 + (i \tau + \lambda) \cdot (i \tau + \lambda) \big| \ge (1 - |\tau|^2) \, (1 + |\lambda|^2), \quad \tau, \lambda \in \R^3 \  \mbox{with} \  |\tau| < 1,
	\]
	it follows from \eqref{defV}
	that
	
	\begin{lem}
		It holds that
		\begin{align} \label{nAnz1}
			\big| \cV(\lambda + i \theta) \big| \le \cV(i \theta) \, \cV(\lambda) = \frac{|\theta|}{\sin{|\theta|}} \, \cV(\lambda) \le \frac{|\theta|}{\sin{|\theta|}}, \quad |\theta| < \pi, \  \lambda \in \R^3.
		\end{align}
	\end{lem}

	We are in a position to give the
	
	\subsection{Proof of Theorem \ref{asyab}}
	
	\begin{proof}
		Under our assumptions, first it deduces from \eqref{Hetau11} and \eqref{32N96} that $\theta$ is a nondegenerate critical point of $\phi(g;\cdot)$ in $B_{\R^3}(0, \pi)$, i.e. the open ball in $\R^3$ with the center $0$ and the radius $\pi$.
		
		Next, notice that the integrand in \eqref{ehk2'} is holomorphic on $\R^3 + i \, \overline{B_{\R^3}(0, 3)} = \{ \lambda + i \tau; \, \lambda \in \R^3, \  \tau \in \overline{B_{\R^3}(0, 3)} \} \subset \C^3$, and decays exponentially from Lemma \ref{rephi} and \eqref{nAnz1}. Hence	we can deform the contour from $\R^3$ to $\R^3 + i \theta$ in \eqref{ehk2'}, and get
		\begin{align}\label{ehk1}
			p(g) = e^{-\frac{d(g)^2}{4}} \int_{\R^3} \cV(\lambda + i \theta) \, \exp\left\{-\frac{1}{4} \Big( \phi(g; \theta - i \lambda) -  \phi(g; \theta) \Big) \right\} \, d \lambda.
		\end{align}
		Since $\theta_2 \, |x| \gg 1$ (so $|x| \gg 1$), we split $\rr^3$ into the following three regions:
		$$\begin{gathered}
			\diamondsuit_1 := \{\lambda \in \R^3; \, |\lambda_1| \le (\theta_2 |x|)^{-\frac{3}{4}}, \, |\lambda^\prime| \le  (\theta_2 |x|)^{\frac{1}{4}} \, |x|^{-1}\}, \\[1mm]
			\diamondsuit_2 := \{\lambda \in \R^3; \,  |\lambda|  \le 1\}\setminus\diamondsuit_1, \quad \mbox{and}\quad \diamondsuit_3 := \{\lambda \in \R^3; \,  |\lambda|  > 1\}.
		\end{gathered}$$
		Then $p(g)=e^{-\frac{\phi(g;\theta) }{4}}\sum_{l=1}^3\J_l$, where
		\begin{align*}
			\J_l := \int_{\diamondsuit_l}\cV(\lambda + i \theta) \, \exp\left\{-\frac{1}{4}\Big( \phi(g; \theta - i \lambda) -  \phi(g; \theta) \Big) \right\}  \, d \lambda, \qquad l = 1,  2, 3.
		\end{align*}
		
		We begin with the estimate of the leading term $\J_1$.  It can be checked that
		\begin{gather*}
			\cV(\lambda + i \theta) =  \cV(i \theta) \, \big( 1 + O((\theta_2 |x|)^{-\frac{3}{4}}) \big),  \\
			\begin{aligned}
				\phi(g; \theta - i \lambda) -  \phi(g; \theta) &= - \frac{1}{2} \lambda^\T \He_{\theta} \, \phi(g; \theta) \lambda +   O\big( |\lambda| \, |x|^2( |\lambda^\prime|^2 + \theta_2^2 \, |\lambda|^2 ) \big) \\
				&= - \frac{1}{2}\lambda^\T \He_{\theta} \, \phi(g; \theta) \lambda  +   O((\theta_2 |x|)^{-\frac{1}{4}}),
			\end{aligned}
		\end{gather*}
		for all $\lz\in\diamondsuit_1$.
		As a result, the standard Laplace's method shows that:
		\begin{align*}
			\J_1 &=  \cV(i\theta)  \int_{\diamondsuit_1} e^{\frac{1}{8}\lambda^\T \He_{\theta} \, \phi(g; \theta) \lambda}  \, d\lambda \,
			( 1 + O((\theta_2 |x|)^{-\frac{1}{4}}) ), \\
			&= \cV(i\theta)   \left(\int_{\R^3} - \int_{(\diamondsuit_1)^c} \right) ( 1 + O((\theta_2 |x|)^{-\frac{1}{4}}) ), \\
			&=  \cV(i\theta) \, \frac{(8\pi)^{\frac{3}{2}}}{\sqrt{\det(-\He_{\theta} \, \phi(g; \theta))}} \, ( 1 + O((\theta_2 |x|)^{-\frac{1}{4}}) ).
		\end{align*}
		
		For $\J_2$,  it follows from Lemma \ref{rephi} that
		$$\Re (\phi(g;\theta - i\lambda) - \phi(g;\theta))\gsim-\lambda^\T \He_{\theta} \, \phi(g; \theta) \lambda\gsim (\theta_2 |x|)^{\frac{1}{2}},\quad \forall \, \lz \in \diamondsuit_2,
		$$
		which, together with \eqref{nAnz1}, implies that
		\begin{align*}
			|\J_2| \le \cV(i \theta) \,  e^{-c \, (\theta_2 |x|)^{\frac{1}{2}}} \int_{\R^3} e^{c \, \lambda^\T \He_{\theta} \, \phi(g; \theta) \lambda}  \, d\lambda =  o(\J_1).
		\end{align*}
		
		We are left with $\J_3$. Set $\widehat{\lambda} = \lambda/|\lambda|$ for $\lambda \neq 0$.
		Notice that Lemma \ref{rephi} again yields
		$$\Re (\phi(g;\theta - i\lambda) - \phi(g;\theta))\gsim-\widehat{\lambda}^\T \, \He_{\theta} \, \phi(g; \theta) \widehat{\lambda} \gsim (\tz_2|x|)^2,\quad \forall\,\lz\in\diamondsuit_3.
		$$
		Moreover, \eqref{nAnz1} says that $|\cV(i \theta + \lambda)| \le 100\, \cV(i \theta) \, e^{-|\lambda|/2}$. Then using the polar coordinates $\lambda = r\, \gz$  with $\gz \in \mathbb{S}^2$, we have that:
		\begin{align*}
			\cV(i \theta)^{-1} \, |\J_3| & \lesssim e^{-c \, (\theta_2 |x|)^2} \, \int_{\diamondsuit_3}  e^{- |\lambda|/2} \, e^{c \,\widehat{\lambda}^\T \He_{\theta} \, \phi(g; \theta) \widehat{\lambda}}  d \lambda \\[1mm]
			& \lesssim e^{-c \, (\theta_2 |x|)^2} \int_1^{+\infty} r^{2} e^{-r/2 } dr \int_{\mathbb{S}^{2}} e^{c \, \gz^\T \He_{\theta} \, \phi(g; \theta)    \gz  }  d\gz  \\
			& \lesssim e^{-c \, (\theta_2 |x|)^2 }
			\int_{\mathbb{S}^{2}} e^{-c \, (\gz_2^2 + \gz_3^2) \, |x|^2}
			d\gz,
		\end{align*}
		because of \eqref{Hetau11}. Next, using the spherical coordinates $\gamma = (\cos{\rho_1},\, \sin{\rho_1} \cos{\rho_2}, \, \sin{\rho_1} \sin{\rho_2})$, the last integral equals
		\begin{align*}
			\int_0^{2\pi} d\rho_2 \int_0^\pi e^{-c \, |x|^2\sin^2{\rho_1} } \sin{\rho_1} d\rho_1 & = 4\pi \int_0^{\frac{\pi}{2}} e^{-c \,|x|^2 \sin^2{\rho_1} } \sin{\rho_1} d\rho_1  \\
			& \lesssim  \int_0^{\frac{\pi}{2}} \rho_1 e^{-c \, |x|^2 \rho_1^2 } d\rho_1  \lesssim \frac{1}{|x|^2},
		\end{align*}
		where the penultimate  ``$\lesssim$''  from the fact that $\sin{\rho} \sim \rho$ on $[0,\frac{\pi}{2}]$. Hence from \eqref{Hetau11}, we get also $|\J_3| = o(\J_1)$ under our assumption.

  Combining all the estimates obtained above, we complete the proof of Theorem \ref{asyab}.
	\end{proof}
	
	\begin{remark} \label{Vara1}
		From the above proof, one can easily see that the condition $|\theta| \le 3$ in Theorem \ref{asyab} can be weakened to
		$|\theta| \le \az_0$ with $3 \le \az_0 < \pi$ and in such case $o_{\zeta_0}(1)$ should be modified to $o_{\zeta_0, \, \az_0}(1)$ as well. Also note that in this case the choice of $\zeta_0$ depends on $\az_0$.
	\end{remark}					

\section{Preparations for asymptotics in difficult cases} \label{s5}
		\setcounter{equation}{0}
		
		In this section, we study more properties of the functions introduced in previous sections, which will be used to develop the uniform asymptotics of our heat kernel.
		
		The first lemma is concerned with the modified Bessel function of order 0 with an additional parameter $\rho$, which is defined by
		\begin{align}
			I_0(\rho;r) := \frac1{2\pi} \int_{-\rho}^\rho e^{r \cos{\gz}} \, d\gz, \qquad r \in \R,\ 0 < \rho \le \pi.
		\end{align}
		Notice that  $I_0(\pi \,; r)$ is exactly $I_0(r)$ (cf. \eqref{defI0}).
		
		\begin{lem} \label{lasinI0}
			It holds that:
			\begin{gather}
				I_0(r) =  \frac{e^{r}}{\sqrt{2\pi r}} \, (1+O(r^{-1})), \qquad{\rm as}\,\, r \to +\infty, \label{asinI0} \\[1mm]
				I_0(r) \sim e^r \, (1+r)^{-\frac12},\quad \forall \, r > 0,  \label{aI0} \\[1mm]
				I_0(\rho; r) = \frac{e^{r}}{\sqrt{2\pi r}} \, (1+o(1))  = I_0(r)  \, (1+o(1)), \qquad{\rm as}\,\, r\rho^2 \to +\infty. \label{asinI}
			\end{gather}
		\end{lem}
		
		\begin{proof}
			The asymptotics \eqref{asinI0} and \eqref{asinI} follow directly from Laplace's method; see also \cite[\S~8.451.5, p. 920]{GR15} for  \eqref{asinI0}. To show \eqref{aI0}, it suffices to observe that $I_0(r) \sim 1$ whenever $r$ is bounded.
		\end{proof}
		
		Some basic properties of the functions $\Upsilon'$ (cf. \eqref{Ups_ser}), $\cZ$ (cf. \eqref{cZd}) and $\Phi$ (cf. \eqref{defUps}) originated from $\Upsilon$ (cf. \eqref{qdef} or \eqref{deD})
		are collected in the following subsection:
		\subsection{Properties of $\Upsilon', \cZ$ and $\Phi$}
	\begin{lem}\label{pPsi}
		The following conclusions hold:
		{\em\begin{compactenum}[(i)]
				\item $\Phi$ is even, positive and strictly increasing on $[0, \  +\infty)$.
				
				\item It holds uniformly for all $0 < r < \var$  that
				\begin{equation}\label{Ups_est}
					-\Upsilon'(r) \sim \frac{r}{(\var-r)^2}, \quad
					-\Upsilon''(r) \sim \frac{1}{(\var-r)^3}, \quad
					-\Upsilon'''(r) \sim \frac{r}{(\var-r)^4}.
				\end{equation}
				
				\item  We have uniformly for all $\rho \ge 0$ that
				\begin{equation}\label{Z_est}
					0 \le \cZ(\rho) < \var, \quad |\cZ'(\rho)|\lsim (\var-\cZ(\rho))^3, \quad |\cZ''(\rho)|\lsim (\var-\cZ(\rho))^5.
				\end{equation}
				
				\item As $\rho \to +\infty$, we have
				\begin{gather}
					\Phi(\rho) = \vartheta_1\, \rho - 2 \sqrt{\vartheta_1 \, \rho} +2 + O(\rho^{-\frac{1}{2}}), \label{aPsi} \\
					\cZ(\rho) = \vartheta_1 - \sqrt{\frac{\vartheta_1}{\rho}} +  O(\rho^{-\frac{3}{2}}), \quad \cZ^\prime(\rho) = O(\rho^{-\frac{3}{2}}), \quad
					\cZ^{\prime\prime}(\rho) = O(\rho^{-\frac{5}{2}}). \label{asP}
				\end{gather}
				Moreover,
				\begin{align}\label{Phi_est}
					\Phi(\rho)  \sim 1+\rho \sim (\var-\cZ(\rho))^{-2}, \qquad \rho \ge 0.
				\end{align}
		\end{compactenum}}
	\end{lem}
	
	\begin{proof}
		Recall that the odd function $\cZ$ satisfies $\cZ(0) = 0$ and $\cZ(\rho) < \var$. From the definition of $\Upsilon$ (cf. \eqref{deD}), we see that $\Phi(\rho) = \Upsilon(\cZ(\rho)) + \rho \cZ(\rho)$ is even, and $\Phi(0) = \Upsilon(0) = 3 > 0$. Taking derivatives, it is clear that
		\begin{align} \label{deZP}
			\Phi^\prime(\rho) = \cZ(\rho), \quad \cZ^\prime(\rho) = -\frac{1}{\Upsilon^{\prime\prime}(\cZ(\rho))} > 0, \quad \cZ^{\prime\prime}(\rho)=-\frac{\Upsilon^{\prime\prime\prime}(\cZ(\rho))}{(\Upsilon^{\prime\prime}(\cZ(\rho)))^3},
			\quad \forall \, \rho \in \R,
		\end{align}
		which implies item (i).
		
		Notice that item (ii)  (resp. (iii)) is a simple consequence of \eqref{Ups_ser}  (resp. item (ii) and \eqref{deZP}).

		Now we turn to the proof of item (iv).  For $r$ near the point $\vartheta_1$, we first observe that $\Upsilon(r) = r^2 \sin{r}/(\sin{r} - r \cos{r})$. Then
		using the Taylor expansion
		at $\vartheta_1$ (recalling $\tan \vartheta_1 = \vartheta_1$):
		\begin{align*}
			r^2 \sin{r} &= \vartheta_1^2 \sin{\vartheta_1} + 3 \vartheta_1 \sin{\vartheta_1} (r - \vartheta_1) + O(|r - \vartheta_1|^2), \\
			\sin{r} - r \cos{r} &= \vartheta_1 \sin{\vartheta_1} (r - \vartheta_1) + \sin{\vartheta_1} (r - \vartheta_1)^2 + O(|r - \vartheta_1|^3),
		\end{align*}
		a direct calculation shows that:
		\begin{align}\label{aDp3}
			\Upsilon(r) = \frac{\vartheta_1}{r - \vartheta_1} + 2 +\square(r),
		\end{align}
		where $\square(r)$ is an analytic function satisfying $\square(\vtz_1) = 0$. Taking derivative, we obtain that:
		\begin{equation}\label{aDp}
			\begin{gathered}
				\Upsilon^\prime(r) = - \frac{\vartheta_1}{(r - \vartheta_1)^2} + O(1),\quad \Upsilon^{\prime\prime}(r) = \frac{2\vartheta_1}{(r - \vartheta_1)^3} + O(1), \\
				\Upsilon^{\prime\prime\prime}(r) = -\frac{6\vartheta_1}{(r - \vartheta_1)^4} + O(1).
			\end{gathered}
		\end{equation}
		
		As a result, whenever $\rho \to +\infty$, the first asymptotic above gives
		\begin{align}\label{aDp2}
			\rho = -\Upsilon^\prime(\cZ(\rho)) =  \frac{\vartheta_1}{(\vartheta_1 - \cZ(\rho))^2} + O(1), \quad \mbox{so} \quad \vartheta_1 - \cZ(\rho) = O(\rho^{-\frac{1}{2}}).
		\end{align}
		Taking square root,  we get that
		\begin{align*}
			\sqrt{\rho} = \frac{\sqrt{\vartheta_1}}{\vartheta_1 - \cZ(\rho)} + O(\vartheta_1 - \cZ(\rho)) = \frac{\sqrt{\vartheta_1}}{\vartheta_1 - \cZ(\rho)} + O(\rho^{-\frac{1}{2}}),
		\end{align*}
		which implies the first equation of \eqref{asP}.  So we get the other two estimates of \eqref{asP} from \eqref{Z_est}. On the other hand, using \eqref{aDp3},  it turns out that:
		\begin{align*}
			\Phi(\rho) &= \Upsilon(\cZ(\rho)) + \rho \cZ(\rho) = \frac{\vartheta_1}{\cZ(\rho) - \vartheta_1} + 2 + \vartheta_1 \rho  - \sqrt{\vartheta_1 \rho} +  O(\rho^{-\frac{1}{2}}) \\
			&= - \sqrt{\vartheta_1 \rho} \, (1 + O(\rho^{-1})) + 2 + \vartheta_1 \rho  - \sqrt{\vartheta_1 \rho} +  O(\rho^{-\frac{1}{2}}) = \vartheta_1 \rho - 2 \sqrt{\vartheta_1 \rho} +2 + O(\rho^{-\frac{1}{2}}),
		\end{align*}
		which is exactly \eqref{aPsi}.
		
		Finally, \eqref{Phi_est} follows from items (i) and (iv).
	\end{proof}
	
	Recall the even function $\q$ is defined by (cf. \eqref{qdef})
	\begin{align*}		
		\q(r) = \frac{r^2 \, \Upsilon(r)}{-\sin{r} \,  \Upsilon'(r) \, \sqrt{-\Upsilon''(r)}}, \quad -\vtz_1 < r < \vtz_1.
	\end{align*}
	The following simple observation will be used in the proof of Propositions \ref{lP} and \ref{Dsim}, as well as of Theorem \ref{sL1}  below:
	
	\begin{lem} \label{q_lem}
		The even function $\q$ is positive and smooth on $(-\var, \ \var)$. Moreover, we have
		\begin{equation}\label{q_var}
			\q(r) = \frac{\sqrt{2\vtz_1}\, \vtz_1}{-2\sin\vtz_1} (\vartheta_1 -r)^{\frac{5}{2}} \, (1 + O(\var - r)), \quad{\rm as}\,\,r \to \vartheta_1^-,
		\end{equation}
		and
		\begin{equation} \label{qvar2}
			\q(r) \sim (\var - |r|)^\frac52, \qquad
			|\q'(r)| \lsim (\var - |r|)^\frac32, \qquad  -\var < r < \var.
		\end{equation}
	\end{lem}
	\begin{proof}
		The proof is easy. Using \eqref{deD} and \eqref{Ups_ser}, it suffices to observe that $\Upsilon$ is monotonically decreasing on $[0, \
		\var)$, $\Upsilon(0) = 3$, and $\pi$ (resp. $0$)  is a simple zero of $\Upsilon$ (resp. $\Upsilon'$).  Notice that we can use directly \eqref{aDp3} and \eqref{aDp} to deduce \eqref{q_var}.
\end{proof}

In Sections \ref{bigchr} and \ref{schr} below, we need to utilize more information of $\cD$ defined by \eqref{decD} near its unique minimum point. To do this, we introduce

\subsection{The auxiliary function $\bH$} \label{afH}

Let $u, \ww$ be defined as in Assumption (A) (cf. \eqref{n78n}). We set for $s = (s_1,s_2) \in \R^2$
\begin{align}\label{defT}
	\bT(s) &= \bT(u;s) :=\sqrt{u_1^2+(u_2+2\ww\, s_1)^2+4\ww^2 s_2^2} =\sqrt{|u|^2+4\ww^2|s|^2+4u_2\ww\, s_1},
\end{align}
namely, the modulus of the vector  $u + 2\ww s_1 \, e_2 + 2\ww\,  s_2  \, e_3$ in \eqref{decD},
and write
\begin{align}
	\L(s) = \L(u;s) :=\ww^{-2} |s|^{-2}\bT(s), \qquad s \ne 0.
\end{align}
Then \eqref{eD3} implies the following more suitable expression of $\cD$ than \eqref{decD}:
\begin{align}\label{eD4}
	\cD(u;s) = \ww^2|s|^2 \, \Phi(\L(s)), \quad \forall \, s \ne 0.
\end{align}

In what follows, it will be convenient to work in the modified polar coordinates
\begin{equation}\label{pol_cor}
	s =(s_1,s_2)= (-\w \cos{\gz}, -\w \sin{\gz}),\quad \w > 0,\quad -\pi < \gz < \pi,  \  \mbox{for $s \neq 0$}.
\end{equation}
Correspondingly, given a function $f$ on $\rr^2$ we shall write $f_\pl(\w,\gz):=f(s)$.
And for brevity, given a function $g(u;s)$ (resp. $g(u;\w,\gz)$) on $\R^2$ (resp. $(0,\infty)\times(-\pi,\pi)$) with the parameter $u$, we will write it simply $g(s)$ (resp. $g(\w,\gz)$) when there is no confusion. For example, the notation above implies that
$$\cD(u;s)=\cD(u;-\w \cos{\gz},-\w \sin{\gz})=:\cD_\pl(u;\w,\gz)=\cD_\pl(\w,\gz).$$

Since the unique minimizer  of $\cDp(\cdot)$ (which is exactly $(1,0)$) belongs to $(0,\infty) \times \{0\}$, we will see later that it is natural to consider the restriction of $\cDp(\cdot)$ on $(0,\infty) \times \{0\}$. To be more precise,  we set
\begin{align}\label{defH}
	\bH(\w) := \cDp(\w,0), \qquad  \w > 0.
\end{align}
Similarly, after introducing the counterparts of $\bT$ and $\L$:
\begin{equation}\label{defu}
	\begin{gathered}
		\bA(\w) :=  \bT(-\w,0) = \sqrt{|u|^2+4\ww^2\,\w^2-4u_2\,\ww\, \w}, \\
		\u(\w) :=  \L(-\w, 0) = \ww^{-2} \, \w^{-2}\bA(\w),
	\end{gathered}
\end{equation}
we obtain
\begin{align}\label{eD5}
	\bH(\w) = \ww^2 \, \w^2 \, \Phi(\u(\w)), \qquad \forall \, \w > 0.
\end{align}

From the definition, $\bH$ is smooth on $(0, +\infty)$ and $1$ is its minimizer, so $\bH'(1) = 0$.  For the convenience of readers, we list here some basic facts, which can be deduced immediately from their definitions by recalling \eqref{crieq}, \eqref{deD}, \eqref{cZd}, \eqref{defUps} and \eqref{deZP} :
\begin{equation}  \label{nBNn}
	\left\{
	\begin{gathered}
		\bA(1) =  \theta_2^2 \psi'(|\theta|) > 0, \quad \u(1) = \frac{\bA(1)}{\ww^2} = \frac{\psi'(|\theta|)}{\psi^2(|\theta|)} = - \Upsilon'(|\theta|),  \\
		\cZ(\u(1)) = \cZ(- \Upsilon'(|\theta|)) = |\theta|, \quad \cZ'(\u(1)) = -\frac{1}{\Upsilon''(|\theta|)}, \\
		\Phi(\u(1)) = \Upsilon(|\theta|) - |\theta| \Upsilon'(|\theta|) = \Upsilon(|\theta|) + |\theta| \u(1).
	\end{gathered}
	\right.
\end{equation}

In Section \ref{bigchr} below, we will fully utilize the Taylor expansion of order $2$ of $\bH$ at its minimizer $1$ and the fact that

\paragraph{The critical point $1$ is nondegenerate.}  To do this, we start to calculate $\bH''(1)$ explicitly (equality \eqref{Hpp}). Indeed, a simple calculation shows that:
\begin{gather}\label{dws}
	\bA^\prime = 2\ww \, \frac{ 2\ww\,\w - u_2 }{\bA}, \quad \bA^{\prime\prime} = \frac{4u_1^2 \, \ww^2}{\bA^3},  \quad
	\bA^{\prime\prime\prime} = - \frac{12 u_1^2 \, \ww^2}{\bA^4} \, \bA^\prime, \\[1mm]
	\label{dws2}
	\u^\prime =\ww^{-2}(-2 \w^{-3} \bA + \w^{-2} \bA^\prime), \quad \u^{\prime\prime}= \ww^{-2}(6 \w^{-4} \bA - 4 \w^{-3} \bA^\prime + \w^{-2} \bA^{\prime\prime}), \\[1mm]
	\label{dws3}
	\begin{aligned}\u^{\prime\prime\prime} &= \ww^{-2}(-24 \w^{-5} \bA + 18 \w^{-4} \bA^\prime - 6 \w^{-3} \bA^{\prime\prime} + \w^{-2} \bA^{\prime\prime\prime})\\[1mm]
		&= - 6 \, \w^{-1} \, \u'' - 6 \, \w^{-2} \, \u' + \ww^{-2} \, \w^{-2} \, \bA'''. \end{aligned}
\end{gather}
Moreover, using the chain rule and the fact that $\Phi'(r) = \cZ(r)$ (cf. \eqref{deZP}), we get that:
\begin{gather}\label{dws4}
	\bH^\prime =\ww^2 \, [2 \w \, \Phi(\u) + \w^2 \Phi^\prime(\u) \u^\prime] = \ww^2 \, [2 \w \, \Phi(\u) + \w^2  \cZ(\u) \u^\prime], \\[1mm]
	\label{dws5}
	\bH^{\prime\prime} =\ww^2 \, [ 2 \Phi(\u) + 4 \w \, \cZ(\u) \u^\prime + \w^2 \cZ^\prime(\u) (\u^\prime)^2 + \w^2 \cZ(\u) \u^{\prime \prime}], \\[1mm]
	\label{dws6}
	\bH^{\prime\prime\prime} = \cZ(\u) \bA^{\prime\prime\prime} + \ww^2\cZ^\prime(\u) [6 \w \, (\u^\prime)^2 + 3 \w^2 \, \u^{\prime} \u^{\prime\prime}] + \ww^2\w^2 \cZ^{\prime\prime}(\u) (\u^\prime)^3,
\end{gather}
where we have used \eqref{dws3} in the last equality.

In particular, at the minimizer $\w = 1$, by the fact that $\cZ(\u(1)) = |\theta|$ and $\bH^\prime(1) = 0$, it follows from \eqref{dws4} that
\begin{equation}\label{U_prim}
	\u^\prime(1) = - 2 \frac{\Phi(\u(1))}{ |\theta|} = 2\Upsilon'(|\theta|) - 2\frac{\Upsilon(|\theta|)}{|\theta|}
\end{equation}
by the last equation in \eqref{nBNn}.
Using the first equation of \eqref{dws2}, we obtain
\[
\bA^{\prime}(1) = \ww^2 \, \u^{\prime}(1) + 2 \bA(1).
\]
This together with the fact that $\bA''(1) = 4 \, u_1^2 \, \ww^2/\bA(1)^3$ (via the second equality in \eqref{dws}) and $\u(1) = \ww^{-2} \bA(1)$ (cf. \eqref{nBNn}) allows us to yield
\[
\u^{\prime\prime}(1) = -2 \u(1) - 4 \u^\prime(1) + \frac{4 u_1^2}{\bA(1)^3}
\]
from the second equation of \eqref{dws2}. Inserting this into \eqref{dws5}, together again with the fact that $\cZ(\u(1)) = |\theta|$ and $\cZ'(\u(1)) = - 1/\Upsilon''(|\theta|)$ (cf. \eqref{nBNn}), we have
\begin{align*}
	\bH^{\prime\prime}(1)  =
	\left[2 \Phi(\u(1)) - 2 |\theta| \u(1) - \frac{(\u^{\prime}(1))^2}{\Upsilon^{\prime\prime}(|\theta|)} + \frac{4 |\theta| u_1^2}{\bA(1)^3}\right] \ww^2
	= \left[\p(|\theta|) +  \frac{4|\theta| u_1^2}{\bA(1)^3}\right] \ww^2,
\end{align*}
where we have used \eqref{U_prim} and the last equality of \eqref{nBNn} in the last ``=", and set that
\begin{equation}
	\p(r):=2 \, \Upsilon(r) - 4 \, \frac{(\Upsilon(r) - r \Upsilon^\prime(r))^2}{r^2 \, \Upsilon^{\prime\prime}(r)} = 4 \frac{\Upsilon(r)^2}{-\Upsilon''(r)\sin^2r}\label{p_cal}, \quad 0 < r < \vtz_1.
\end{equation}
The second ``=" in \eqref{p_cal} can be verified directly. In fact, from \eqref{relUp}, we see that:
\begin{align*}
	\p(r)&= 2 \frac{r^2 \left[ -\psi \, \psi'' + 2 \, (\psi')^2 \right] - 2 \, ( \psi +  r \, \psi')^2 }{r^2 \, \Upsilon'' \, \psi^4} = 2 \frac{2 \, \psi + r^2 \, (\psi'' + 4 \frac{\psi'}{r})}{-r^2 \, \Upsilon'' \, \psi^3} \\[1mm]
	&=2 \frac{2 \, \psi + 2 \, \psi \, (r^2 \, \csc^2{r} - 1)}{-r^2 \, \Upsilon'' \, \psi^3} = 4 \frac{\Upsilon(r)^2}{-\Upsilon''(r)\sin^2r},
\end{align*}
where the penultimate equality follows from \eqref{psi_sum}.
Remark that $\pi$, the unique zero of $\Upsilon$ on $(0, \  \var)$, is simple. By \eqref{Ups_ser}, it is obvious that $\p > 0$, which shows that $\bH^{\prime\prime}(1)> 0$. In conclusion, by the fact that $\bA(1) = \theta_2^2 \, \psi'(|\theta|)$, $u_1 = \theta_1 \, \theta_2^2 \, \psi'(|\theta|)/|\theta|$ and $\ww = \theta_2 \, \psi(|\theta|)$, we yield the following:

\begin{lem} \label{L64n}
Let $\theta, u$ be as in Assumption (A) (cf. \eqref{n78n}). Then
\begin{align} \label{Hpp}
	\bH''(1) = \partial^2_{s_1} \cD(u;\ss) = 4
	\left(\frac{\tz_1^2}{-\Upsilon'(|\tz|)|\tz|}+\frac{\tz_2^2}{-\Upsilon''(|\tz|)\sin^2|\tz|}\right) > 0.
\end{align}
\end{lem}

It is a simple but vital fact that both $\cDp$ and $\bH$ enjoy some monotonicity properties, which will be used extensively in Section \ref{bigchr}. Actually, one has:
\begin{lem} \label{pcD} The following conclusions hold:
{\em\begin{compactenum}[(i)]
		\item For fixed $\w$, the even function $\gz \mapsto \cDp(\w,\gz)$ is increasing w.r.t. $|\gz|$;
		\item The function $\bH(\cdot)$ is decreasing on the interval $(0, 1)$ and increasing on $(1, + \infty)$.
\end{compactenum}}
\end{lem}

\begin{proof}
The proof of (i) is easy. Since $\w$ is fixed, then by \eqref{defT} the positive function
\begin{equation}\label{A_pol}
	\begin{aligned}
		\bT_\pl(\w,\gz) &= \sqrt{u_1^2+(u_2-2\ww\,\w\cos\gz)^2 + 4 \, \ww^2 \, \w^2 \sin^2\gz} \\
		&= \sqrt{|u|^2 + 4\ww^2\,\w^2 - 4 u_2\ww\, \w\cos{\gz} }
	\end{aligned}
\end{equation}
is increasing w.r.t. $|\gz|$,  so is $\L_\pl(\w, \gz) = \ww^{-2} \w^{-2} \bT_\pl(\w,\gz)$. Thus (i) follows from the fact that $\cDp(\w,\gz) = \ww^2 \, \w^{2} \, \Phi(\L_\pl(\w,\gz))$ and the first claim in Lemma \ref{pPsi}.

To prove (ii), it suffices to show that $\bH^\prime(\w) < 0$ on $(0, 1)$ and $\bH^\prime(\w) > 0$ on $(1, + \infty)$. Since $\bH'(1)=0$ and $\bH^{\prime\prime}(1) > 0$, then $\bH^\prime(\w) < 0$  whenever $1 - \w > 0$ is small enough,  and $\bH^\prime(\w) > 0$ for $\w - 1 > 0$ small enough. As a result,  by the smoothness of $\bH$, it remains to show that  there is no $\w_* \in (0, 1) \cup (1, + \infty)$ such that $\bH^\prime(\w_*) = 0$. If this is not the case, from (i) it will deduce that
\begin{align*}
	\partial_{\w} \cDp(u;\w_*,0) = 0, \quad  \partial_{\gz} \cDp(u;\w_*,0) = 0, \quad \partial^2_{\gz} \cDp(u;\w_*,0) \ge 0,
\end{align*}
which is equivalent to
\begin{align*}
	\partial_{s_1} \cD(u;-\w_*,0) = 0, \quad  \partial_{s_2} \cD(u;-\w_*,0) = 0, \quad \partial^2_{s_2} \cD(u;-\w_*,0) \ge 0.
\end{align*}
Then following the argument in the proof of Theorem \ref{tmm}, we have $(-\w_*,0) = (-1,0)$ and this leads to a contradiction, which completes the proof.
\end{proof}

We are now in the position to introduce

\subsection{Some key parameters related to $4^{-1}|x|^2\,\cD(u;\cdot)$ }
\label{pfN}

Under Assumption (A) (cf. \eqref{n78n}), recalling \eqref{Hpp} and \eqref{ps2D}, we set in the sequel,
\begin{gather}
\m=\m(g):=\frac{|x|^2}{4}\bH(1)=\frac{|x|^2}{4}\cD(u;\ss),\label{mdef}\\[1mm]
\chri=\chri(g):= \frac{|x|^2}{4}\partial^2_{s_1} \cD(u;\ss) = \frac{|x|^2}{4}\bH''(1) =\left(\frac{\tz_1^2}{-\Upsilon'(|\tz|)|\tz|}+\frac{\tz_2^2}{-\Upsilon''(|\tz|)\sin^2|\tz|}\right)|x|^2,\label{chr1def}\\[1mm]	\chrii=\chrii(g):=\frac{|x|^2}{4} \, \partial^2_{s_2} \cD(u;\ss) = \frac{\psi(|\tz|) \, |\tz|}{2\psi'(|\tz|)} \, \K_3(\tz_1,\tz_2) \, |x|^2 = \frac{u_2 \, |x|^2 \, \ww \, |\theta|}{2\bA(1)}, \label{chr2def}
\end{gather}
where we have used in the last ``$=$'' of \eqref{chr2def} the fact that $\bA(1) = \theta_2^2 \, \psi'(|\theta|)$.
Note that $\m$ is exactly the minimum of $\frac{|x|^2}{4}\cD(u;\cdot)$, and in fact  $\chri,\, \chrii > 0$ are two eigenvalues of the Hessian matrix of $\frac{|x|^2}{4}\cD(u;\cdot)$ at its unique minimum  point $(-1,0)$.

These parameters are important for the asymptotics of the heat kernel in Sections \ref{bigchr} and \ref{schr} below. We summarize  some useful properties of them in the following lemma, which is a direct application of  Theorems \ref{mapL} and \ref{tmm}.

\begin{lem}\label{asyin} Supposing that Assumption (A) (cf. \eqref{n78n}) holds, then we have:
{\em\begin{compactenum}[(i)]
		\medskip\item $\mz \sim u_1 \, \tz_1^{-1} \sim \tz_2^2 \, |\pi-|\tz| \, |^{-2} \sim u \cdot \tz
  \sim \ep^{-2} \, \ww^2 \sim |\tz|^{-1} \, \bA(1)$.
		\medskip\item $\chrii \sim \ep \, u_2 \, |x|^3 \, \m^{-\frac12}$, $\chri \sim \ep^2 \, d(g)^2$, and $\chri \sim \chrii + \ep \, \m$.
		
		\medskip\item If $|u|\gtrsim1$, then
		\begin{align}\label{con1}
			|\tz| \sim 1,\quad \ep \sim 1, \quad |u|\sim \frac{\tz_2^2}{(\pi-|\tz|)^2}, \quad  \mz \sim |u|, \quad \ww \sim |u|^{\frac12}.
		\end{align}
		
		\medskip\item If $|u|\ll1$ and $|\tz|\ge1$, then
		\begin{align}\label{con6}
			\epsilon \sim u_1 + u_2 \, u_1^{-\frac12}, \quad \mz \sim u_1, \quad \ww \sim \ep \, u_1^{\frac12}, \quad \theta_1 \sim 1.
		\end{align}
		In particular, $u_2/\sqrt{u_1}$ is bounded.
		
		\medskip\item If $|\tz|\le3$, then
		\begin{align}\label{con7}
			u_1\sim\tz_1\tz_2^2,\quad u_2\sim\tz_2,\quad\mz\sim\tz_2^2,\quad\chrii\sim|x|^2.
		\end{align}
		
		\medskip\item It holds that
		\begin{equation}\label{crieq2}
			|u_2 - 2 \ww| = -\frac{\Upsilon'(|\tz|)}{|\tz|}\, |\tz_2| \, \ww^2 \sim |\tz_2| \, \ep^{-2} \,\ww^2 \sim |\pi-|\tz| \, | \,\mz^{\frac32} \lsim \mz^\frac32.
		\end{equation}
\end{compactenum}}
\end{lem}
\begin{proof}	
Recall that $\mz = \cD(u;\ss)$, $\ep = \var - |\theta|$,  $\ww = \theta_2 \, \psi(|\theta|)$ and $d(g)^2 \sim |x|^2 \, (1 + |u|)$ (cf. \eqref{ehd}).

For item (i), let us begin with the proof of $\cD(u;\ss) \sim u_1/\theta_1$. By the penultimate equality in \eqref{dEn2}, it remains to prove that the continuous function $\varphi_2(r) \, r = r^2 \, (r^2 - \sin^2{r})/h(r) \sim 1$ for all $0 < r \le \var$. In fact, it follows from Corollary \ref{nCn1} that $\varphi_2(r) \, r > 0$. Next, a simple calculation, via Taylor's expansion at $0$ for  $r^2 \, (r^2 - \sin^2{r})$ and $h(r)$, shows that $\lim\limits_{r \to 0^+} \varphi_2(r) \, r > 0$, which implies the desired result. Similarly, using the third and the second ``$=$'' in \eqref{dEn2} successively, we get that  \[
\cD(u;\ss) \sim u \cdot \tz = |u \cdot \tz| \sim  \tz_2^2 \, (\pi-|\tz|)^{-2}.
\]

On the other hand, applying \eqref{mapL_p} with $r = |\theta|$, we have that
\[
\ww^2 = \theta_2^2 \, \psi^2(|\theta|) \sim \frac{\theta_2^2}{(\pi - |\theta|)^2} \, \ep^2, \quad \mbox{and} \quad
\bA(1) = \theta_2^2 \, \psi'(|\theta|) \sim \frac{\theta_2^2}{(\pi - |\theta|)^2} \, |\theta|.
\]
Hence $\mz \sim \ep^{-2} \, \ww^2 \sim |\tz|^{-1} \bA(1)$.

For item (ii), to show the first claim, it suffices to use
the last  ``='' in \eqref{chr2def}, since it holds that $\ww \sim \ep \, |x|^{-1} \m^\frac12$ and $\bA(1) \sim |\tz| \, |x|^{-2} \m$ by item (i) and the fact that $\bH(1) = 4\m/|x|^2$. Next we prove items (iii)-(v) before proving  the remaining claims.

For item (iii),  it follows from Assumption (A) that
\begin{eqnarray} \label{nSEn1}
	1 \lsim |u| \sim u_1+u_2 = \frac{\psi'(|\theta|)}{|\theta|} \theta_2^2 \, \theta_1 + \theta_2 \left( \frac{\psi'(|\theta|)}{|\theta|} \theta_2^2 + 2 \psi(|\theta|) \right) \lsim \frac{\tz_2^2}{(\pi-|\tz|)^2} + \frac{|\tz_2|}{|\pi - |\tz| \, |},
\end{eqnarray}
where we have used \eqref{mapL_p} in the last inequality. Hence
$|\tz|\sim1$ and $|\tz_2|/|\pi-|\tz| \, | \gsim 1$.

In the case $|\theta| < \pi$ (so $\ep \sim 1$), by \eqref{u1sim} and \eqref{u2sim1}, the estimate \eqref{nSEn1} can be improved as
\begin{align*}
	|u| \sim \frac{\tz_2^2}{(\pi-|\tz|)^2} + \frac{\tz_2}{\pi - |\tz|} \sim \frac{\tz_2^2}{(\pi-|\tz|)^2}.
\end{align*}

In the opposite case $\pi < |\theta| < \var$, it deduces from \eqref{u2sim2} that $\ep/(|\theta| - \pi) \gsim \tz_2^2/(|\theta| - \pi)^2 \gsim1$, and thus $\ep \sim 1$. Using \eqref{u2sim2} again we obtain $u_2 \lsim |\tz_2|/(|\theta| - \pi)\sim \sqrt{u_1}$ by \eqref{u1sim} and the fact that $\tz_1\sim1$. This also leads to $|u|\sim \tz_2^2/(|\theta| - \pi)^2$.

In conclusion, we establish the first three estimates in \eqref{con1}. The others are clear by (i).

We are in a position to prove item (iv). Indeed by item (i) we have  $\mz\lesssim|u|$ and $|\tz_2| \sim |\pi-|\tz| \, | \, \mz^{\frac12}$, whence  $|\tz_2| \ll 1$ and $\tz_1 \sim 1$.  By (i) again, it turns out
that
\begin{align} \label{nZn1}
\mz \sim u_1 \sim \tz_2^2 \, (\pi-|\tz|)^{-2} \ll 1, \qquad  \ww \sim \ep \, \sqrt{u_1}.
\end{align}

To show the first estimate,  consider the following three cases: (1) $\ep \ll 1$, (2) $\ep \sim 1$ with $|\theta| < \pi$, and (3) $\ep \sim 1$ with $|\theta| > \pi$. In the first case, it follows from \eqref{u2sim2} that $\ep \sim u_1 + u_2 \, |\tz_2|^{-1}$. And in such case, we remark that $|\tz_2|\sim u_1^\frac12$, thereby showing that $\ep \sim u_1 + u_2 \, u_1^{-\frac12}$. In the second case, by \eqref{u2sim1} and the first equation in \eqref{nZn1}, it holds that
$$ u_2 \sim \frac{|\tz_2|}{|\pi-|\tz| \, |} \left(1+ |\tz_2| \frac{|\tz_2|}{|\pi-|\tz| \, |} \right) \sim \frac{|\tz_2|}{|\pi-|\tz| \, |} \sim u_1^\frac12,
$$
and therefore, $u_1+u_2 u_1^{-\frac12}\sim1\sim\ep$. In the third case, the argument is similar except using \eqref{u2sim2} instead.

For item (v), since $|\tz|\le3$ (so $\ep\sim1$), then by \eqref{u1sim} and \eqref{u2sim1} we have $u_1\sim\tz_1\tz_2^2,\,u_2\sim\tz_2$. Thus, by item (i) we obtain $\mz\sim\tz_2^2\sim u_2^2$.  It follows from the first claim in item (ii) that $\chrii\sim u_2|x|^3\m^{-\frac12}\sim|x|^2$.

Now we return to the rest claims in item (ii). Recall \eqref{chr1def}. By \eqref{Ups_est} and the trivial fact that $\sin^{-2}{r} \sim r^{-2} \, (\pi - r)^{-2}$ for $0 < r \neq \pi < \var$, we obtain
\begin{equation}\label{chris}
	\chri\sim|x|^2\ep^2|\tz|^{-2} \left(\, \tz_1^2+\ep\, |\pi-|\tz| \, |^{-2}\, \tz_2^2 \,\right).
\end{equation}
There are two possible cases:

{\bf (ii-1)} $|u|\gtrsim1$. We have $d(g)^2 \sim |x|^2 \, |u|$. Then combining item (i) with item (iii) yields  that $\ep \sim 1$, $\chri \sim |x|^2 \,(\tz_1^2+|u|) \sim |x|^2 \, |u| \sim \ep^2 \, d(g)^2 \sim \ep \, \m$ and $\chrii \sim |x|^2 \, u_2 \, |u|^{-\frac12} \lesssim |x|^2 \, |u| \sim \chri$, whence $\chrii + \ep \, \m \sim \chri$.

{\bf (ii-2)} $|u|\ll1$. We have $d(g)^2 \sim |x|^2$. {\bf (ii-2a) } If $|\tz|\ge1$, recalling that $\tz_1\sim1$ (cf. \eqref{con6}), then by \eqref{chris} and the first equation in \eqref{nZn1}, we get $\chri \sim |x|^2 \, \ep^2 \, (1 + \ep \, u_1) \sim \ep^2 \, d(g)^2$. By using \eqref{con6} and again the first claim in item (ii), we yield that $\ep \, \m + \chrii \sim \ep \, (|x|^2 \, u_1 + |x|^2 \, u_2 \, u_1^{-\frac12}) \sim \ep^2 \, |x|^2 \sim \chri$. {\bf (ii-2b)} If $|\tz|\le3$, by \eqref{chris} and item (v) we obtain $\chri \sim |x|^2 \, |\tz|^{-2} \, (\tz_1^2 + \tz_2^2) = |x|^2 \sim \ep^2 \, d(g)^2 \sim \chrii$, and therefore $\ep \, \m \lesssim |x|^2 \, |u| \ll \chri$. Hence item (ii) is fully verified.

Finally, item (vi) follows easily from \eqref{crieq}, \eqref{Ups_est} and item (i).
\end{proof}

\begin{remark} \label{rem67}
	For Lemma \ref{asyin} (iv), if we replace the condition ``$|\tz|\ge1$" by ``$|\tz|\ge\az_0$ with $\az_0\in(0,1]$", then \eqref{con6} still holds with the implicit constants depending additionally on $\az_0$.
	
	Similarly for Lemma \ref{asyin} (v), if we replace the condition ``$|\tz|\le3$" by ``$|\tz| \le \beta_0$ with $\beta_0 \in [3,\pi)$", then \eqref{con7} still holds with the implicit constants depending additionally on $\beta_0$.
	
\end{remark}

			\section{Uniform asymptotics for  the subtlest case: $|\tz|\ge1$, $\chri\gg1$ and $\m \to +\infty$} \label{bigchr}
			\setcounter{equation}{0}
			
			Recall that  $d(g)^2 = |x|^2 + 4 \m$, $\m = \frac{|x|^2}{4} \bH(1)$ is the minimum of $\frac{|x|^2}{4}\cD(u;\cdot)$, and   $\chri,\, \chrii > 0$ are two eigenvalues of the Hessian matrix of $\frac{|x|^2}{4}\cD(u;\cdot)$ at its unique minimum  point $(-1,0)$. Our target in this section is the following:
			
			\begin{theo}\label{bigL1}
				Let $|\tz|\ge1$ and $\chri, \m \ge \nzz$ with $\nzz\gg1$. Then
				\medskip
				$$
				p(g) = 16 \pi^2 \sqrt{\pi\vtz_1} \,  \q(|\tz|) \, \frac{e^{-\chrii} \, I_0(\chrii)}{\sqrt{\ep \, \chri}} \, e^{ -\frac{\vartheta_1|x|^2 \, \ww^2}{2\ep}} \, I_0\left(\frac{\vartheta_1|x|^2 \, \ww^2}{2\ep}\right) \,  e^{-\frac{d(g)^2}{4}} \, (1+o_{\nzz}(1)).
				$$
			\end{theo}\medskip
			
			 Its proof is based on Propositions \ref{nPn1} and \ref{lP}. As mentioned in  Subsection \ref{ideaS}, Laplace's method is not enough for our purpose, since $\chrii$ could be bounded. To prove it in a uniform way, we will use Proposition \ref{Dsim} below, via the modified polar coordinates provided by  \eqref{pol_cor}.  In short, we introduce a suitable coordinate system, give a nice expression for the integrand (especially, \eqref{simpD} for the phase), and determine the main contributing region $\cw_1$ of the integral.

					{\bf Structure of $\boldsymbol\cw_1$.} This set depends on two parameters $\dz$ and $\eta$ (cf. \eqref{delg} below), which in fact involves three different cases. To avoid redundant computation, we list some preparatory estimates for the aforementioned parameters in the following Table \ref{tab1}. These results, which can be checked without difficulties by Lemma \ref{asyin} (i)-(iv) and the fact that $d(g)^2 \sim |x|^2 \, (1 + |u|)$, will be used iteratively throughout the proof of Theorem \ref{bigL1}.
					
					\medskip
					{\bf More explanations for Table \ref{tab1}:}
					\begin{itemize}
						\item The condition for all estimates in Table \ref{tab1} to hold is the same as that in Theorem \ref{bigL1}, i.e., $|\tz|\ge1$, $\chri\gg1$ and $\m\gg1$, while the last two columns need an additional hypothesis that $\chrii\gg1$.
						
						\item These estimates should be used in the sense of ``$\sim$". For instance, ``1" in the $\ep$-Column means $\ep\sim1$.
					\end{itemize}
					
					\begin{table}[ht]
						\centering
						\caption{{\bf Preparatory estimates for Theorem \ref{bigL1}}}\label{tab1}
						\vspace{.2cm}
						\renewcommand{\arraystretch}{2}
						\begin{tabular}{|c|c|c|c|c|c|c|c|}
							\hline
							{\bf Case} & {\bf Condition} & $\ep$ & $\mz$ & $\dz$ & $\chri \, \dz^2$ & $\eta$ & $\chrii \, \eta^2$ \\
							\hline
							{\bf (1) } & $|u|\gsim1$ & 1 & $|u|$  & $\chri^{-\frac38}$ &  $\m^{\frac14}$  & $\chrii^{-\frac38}$ & $\chrii^{\frac14}$  \\
							\hline
							{\bf (2) }  & $|u|\ll1$, $u_2 \gsim u_1$ & $u_2 \, u_1^{-\frac12}$ & $u_1$ & $|x|^{-\frac34}\, u_2^{-1} \, u_1^\frac58$ & $\m^\frac14$ & $|x|^{-\frac34} \, u_2^{-1} \, u_1^\frac58$ & $\m^\frac14$  \\
							\hline
							{\bf (3) } & $|u|\ll1$, $u_2 \ll u_1$ & $u_1 + \frac{u_2}{\sqrt{u_1}}$ & $u_1$ & $\chri^{-\frac38}$ & $\chri^{\frac14}$ & $\chrii^{-\frac38}$ &  $\chrii^{\frac14}$ \\
							\hline
						\end{tabular}
					\end{table}

					\begin{prop}\label{Dsim}
						Let $|\tz|\ge1$ and  $\chri, \m \ge \nzz$ with $\nzz \gg1$. Then it holds that
						\begin{equation}\label{simpD}
							\cDp(\w,\gz) = \bH(1) + \frac12 \bH^{\prime\prime}(1) (\w-1)^2 +  \frac{4\chrii}{|x|^2}(1 -  \cos{\gz}) +  o_{\nzz}(|x|^{-2}), \quad\forall (\w,\gz)\in\cw_1,
						\end{equation}
						where $\cw_1:=\{(\w,\gz);|\w-1|\le\dz,\,|\gz|\le\eta\}$ with
						\begin{equation}\label{delg}
							\dz=\dz(g):=\chri^{-\frac38}\left(\frac{\mz}{|u|}\right)^\frac14\ll1,\quad
							\eta=\eta(g):=\begin{cases}
								\chrii^{-\frac38}\left(\frac{\mz}{|u|}\right)^\frac14\ll1,&\chrii\gg1,\\
								\pi,&\chrii\lesssim1.
							\end{cases}
						\end{equation}
						Moreover, we have uniformly on $\cw_1$ (cf. \eqref{def_cP}, \eqref{qdef} and the convention \eqref{n78n} for pertinent definitions)
						\begin{equation}\label{ma1}
							\begin{aligned}
								\w \, \cP_\pl(\w,\gz) &=\sqrt{2\vtz_1}16\pi^2  \, \q(\atz)  \, \frac{\ep^{-\frac12}} {|x|^2 \, \ww^2}  \, \exp\left(-\chrii-\frac{\vartheta_1|x|^2 \, \ww^2}{2\ep}\right)
								\,I_0 \left(\frac{\vartheta_1|x|^2 \, \ww^2}{2\ep} \right) e^{-\m} \\[1mm]
								&\qquad \cdot  \exp\left( - \frac12 \chri(\w-1)^2 +  \chrii \cos{\gz}\right)\,(1 + o_{\nzz}(1)).
							\end{aligned}
						\end{equation}
						
					\end{prop}
					
					The proof of Proposition \ref{Dsim} is postponed to Subsection \ref{ss73} below. At this point we see how to apply it to obtain Theorem \ref{bigL1}.  For the sake of brevity, from now on we will omit the dependence of undetermined constants in the proof.
					
					Let us decompose $(0,+\infty) \times (-\pi,\pi)
					= \bigcup\limits_{i=1}^3\cw_i$ with $\cw_1$ defined as in Proposition \ref{Dsim} and
					\begin{equation}
						\cw_2:=\{(\w,\gz);\w< C|u| \, \du^{-1} \, \ww^{-1}\}\setminus\cw_1, \quad
						\cw_3:=\{(\w,\gz);\w\geq C|u| \, \du^{-1} \, \ww^{-1}\},
					\end{equation}
					where the universal  constant $C$ will be chosen later. Then by \eqref{ehk4} and \eqref{pol_cor} we can write
					\begin{align}
						p(g) &= \frac{1}{4\pi} \, e^{-\frac{|x|^2}{4}}|x|^2 \, \ww^2 \int_0^{+\infty} \int_{-\pi}^\pi \w \, \cP_\pl(\w, \gz)  \, d\w d\gz \nonumber\\[1mm]
						&= \frac{1}{4\pi} \, e^{-\frac{|x|^2}{4}} \sum_{i=1}^3 \int_{\cw_i}  |x|^2 \, \ww^2 \,  \w \, \cP_\pl(\w, \gz)  \, d\w d\gz =: \frac{1}{4\pi} \, e^{-\frac{|x|^2}{4}} \sum_{i=1}^3\pp_i.
					\end{align}
					We will see that the main contribution comes from $\pp_1$.
					
					{\bf Estimate of $\pp_1$}.   We use \eqref{ma1} to treat $\pp_1$. Under our assumptions, observe that $\chri \dz^2 \gtrsim \min\{\m, \chri\}^{1/4} \ge \nzz^{1/4} \gg 1$. Then the standard Laplace's method implies that:
                  \begin{align*}
						\int_{1 - \delta }^{1 + \delta} e^{- \frac12\chri (\w - 1)^2} d\w
						= \sqrt{\frac{2\pi}{\chri}}\int_{-\sqrt{\frac{\chri\dz^2}{2\pi}}}^{\sqrt{\frac{\chri\dz^2}{2\pi}}} e^{-\pi r^2}dr = \sqrt{\frac{2\pi}{\chri}} \, \left( 1 + O(e^{-c \, \chri \, \dz^2}) \right).						
					\end{align*}
					On the other hand, if $\chrii \lesssim 1$, then $\eta=\pi$ and therefore,
					\begin{equation*}
						\int_{|\gz|<\eta}e^{\chrii\cos\gz}d\gz=2\pi I_0(\chrii).
					\end{equation*}
					If $\chrii\gg1$, notice that  $\chrii \,  \eta^2 \gtrsim \min\{\chrii, \m\}^{1/4} \gg1$ (cf. also Table \ref{tab1}). Then Lemma \ref{lasinI0} gives:
				\begin{equation*}
					\int_{|\gz|<\eta}e^{\chrii\cos\gz}d\gz=2\pi I_{0}(\eta;\chrii)=2\pi I_0(\chrii) \, (1+o(1)).
				\end{equation*}
				
				Combining all the estimates above yields that
				\begin{align}
					\pp_1 &= 64 \, \pi^3 \, \sqrt{\pi \vtz_1}  \, \q(|\tz|) \,  \frac{e^{-\chrii}  I_0 (\chrii)}{\sqrt{\ep \, \chri} }
					e^{-\frac{\vartheta_1|x|^2 \, \ww^2}{2\ep}} \, I_0 \left(\frac{\vartheta_1|x|^2 \, \ww^2}{2\ep}\right)  e^{-\m}(1+o(1))\nonumber\\
					&\sim \ep^2 \, \chri^{-\frac12} \, (1 + \ep \, \m)^{-\frac12}(1+\chrii)^{-\frac12} \, e^{-\m}, \label{pbnd1}
				\end{align}
				where we have used in ``$\sim$''
				\eqref{qvar2}, \eqref{aI0} and the fact that $|x|^2 \, \ww^2/\ep \sim \ep \, \m$ (cf. Lemma \ref{asyin} (i)  with the definition of $\m$ \eqref{mdef}).
				
				As a consequence, to finish the proof of
				Theorem \ref{bigL1},  it is enough to
				show that
				\begin{equation} \label{NRE1N}
					\pp_2+\pp_3  =  o(\pp_1).
				\end{equation}
				This is done in the following two cases, $|u|\gtrsim1$ and $|u|\ll1$, which will be treated in Subsections \ref{ss71} and \ref{usmall}, respectively.
				
				The following estimates will be used repeatedly hereafter:
				\begin{equation}\label{sim_P}
					\cDp \sim \ww^2 \, \w^2 + \bT_\pl, \qquad
					\cP_\pl \lsim \m^{-1} \, e^{-\frac{|x|^2}{4}\cDp}, \quad\forall\, (\w,\gz).
				\end{equation}
				Indeed, by recalling the definition of $\cD$ (cf. \eqref{decD}),
				$\bT$ (cf. \eqref{defT}) and $\cP$ (cf. \eqref{def_cP}) (with the coordinates \eqref{pol_cor}), the first estimate comes from  \eqref{D_scal}, and the second one from Proposition \ref{lP} (ii) and the fact that the minimum of $|x|^2 \, \cD/4$ equals to $|x|^2 \, \mz/4 = \m \gg 1$ (cf. \eqref{mdef} again and Theorem \ref{tmm}).
				
				\subsection{Proof of {\eqref{NRE1N}} in the case where $|u| \gtrsim 1$} \label{ss71}
				
				In this case, notice that $d(g_u)^2 \sim 1 + |u| \sim |u|$ (cf. \eqref{ehd}), so $d(g)^2 \sim |x|^2 \, |u|$. First from Table \ref{tab1} with  Lemma \ref{asyin} (i)-(iii), we remark that
				\begin{align} \label{nZn2}
					\ep \sim 1, \  d(g)^2 \sim \chri \sim \m \sim |x|^2 \, \ww^2 \sim |x|^2 \, |u| \gg 1,  \ \chrii \lesssim \m, \ \chri\dz^2 \sim \m^{\frac14} \gg 1.
				\end{align}
				
				Then \eqref{pbnd1} can be simplified as
				\begin{equation}\label{p1_1}
					\pp_1 \sim (1+\chrii)^{-\frac12} \, \m^{-1} \, e^{-\m}  \gtrsim \m^{-\frac{3}{2}} \, e^{-\m}.
				\end{equation}
				
				{\bf Estimate of $\pp_3$}. On $\cw_3$, by \eqref{sim_P} we obtain $|x|^2 \, \cDp \gtrsim |x|^2 \, \ww^2 \, \w^2 \gsim C^2 \, \m$, whence there is a positive constant $c$ such that $\frac{|x|^2}{4}\cDp-\m \geq c \, \m + c \, |x|^2 \, \ww^2 \, \w^2$ by selecting $C$ large enough. Combining this with the second claim in \eqref{sim_P}, we get that
				\begin{equation*}
					\pp_3 \lesssim |x|^2 \, \ww^2 \, \m^{-1} \, e^{-\m} \, e^{-c \, \m} \int_0^\infty e^{-c \, |x|^2 \, \ww^2 \, \w^2} \, \w \, d\w \sim \m^{-1} \, e^{-\m} \, e^{-c \, \m} =  \pp_1 \, o(e^{-c' \, \m}).
				\end{equation*}
				
				{\bf Estimate of $\pp_2$}. We have to deal with two kinds of subcases: $\chrii\lsim1$ and $\chrii\gg1$.
				For $\chrii\lesssim1$, recalling $\eta=\pi$, then by applying Lemma \ref{pcD} and \eqref{simpD} (with $\gz=0$)  successively, we have
				\begin{equation}\label{H_incre}
					\frac{|x|^2}{4} \cDp \ge \frac{|x|^2}{4}\min\{\bH(1-\dz),\bH(1+\dz)\}=\m+\frac12\chri\dz^2+o(1)\ge\m+ c \, \m^{\frac14}
				\end{equation}
				on $\cw_2$ for some positive constant $c$. Hence by  the second estimation in \eqref{sim_P} again, the fact that $d(g_u)^2 \sim |u| \sim \ww^2$  and  $|x|^2\,\ww^2 \m^{-1}\sim1$ (cf. the second equation in \eqref{nZn2}), we obtain
				\begin{equation*}
					\pp_2 \lesssim |x|^2 \, \ww^2 \, \m^{-1} \, e^{-\m} \, e^{-c\m^{\frac14}}\int_{\cw_2} \w \, d\w d\gz\lsim e^{-\m}e^{-c\m^{\frac14}} = \pp_1 \, o(\exp\{-c' \, \m^{\frac{1}{4}}\}).
				\end{equation*}
				
				In the opposite case $\chrii\gg1$, we decompose $\cw_2=\cw_2'\bigcup\cw_2''$ with
				\begin{equation}\label{setw2}
					\begin{gathered}
						\cw_2':=\{(\w,\gz);|\w-1|\le\dz, \, |\gz|>\eta\},\\[1mm]
						\cw_2'':=\{(\w,\gz);|\w-1|>\dz, \,  \w< C \, |u| \, \du^{-1} \, \ww^{-1}\},
					\end{gathered}
				\end{equation}
				on which the integrals of $|x|^2\, \ww^2 \, \w \, \cP_\pl$ are denoted by $\pp_2'$ and $\pp_2''$ respectively. Notation for such decomposition will be used repeatedly in the rest of this section without particular statements.

				The same argument as in the case $\chrii\lsim1$ yields $\pp_2'' =  \pp_1 \, o(e^{-c' \, \m^\frac14})$.
				
				To estimate  $\pp_2'$, first notice that $\chrii \, \eta^2 \sim \chrii^{\frac14} \gg1$.  Similar to the proof of \eqref{H_incre}, there exists some positive constant $c$ such that
				\begin{equation}\label{H_incre2}
				\frac{|x|^2}{4} \cDp(\w,\gz) \ge \frac{|x|^2}{4} \cDp(\w,\eta) \ge \m + \frac{\chri}{2} (\w - 1)^2 + c \,  \chrii^{\frac14}, \qquad \forall (\w,\gz) \in \cw_2'.
				\end{equation}
				Combining this with  $|x|^2 \, \ww^2 \sim \m$ and Proposition \ref{lP} (ii), we conclude
				\begin{equation} \label{L2est}
					\w \, \cP_\pl(\w, \gz) \lesssim \m^{-3/2} \, \exp\left\{-\m - \frac{\chri}{2} (\w - 1)^2 - c \,  \chrii^{\frac14}\right\},
				\end{equation}
				which implies immediately that $\pp_2' = \pp_1 \, o(\exp\{-c' \, \chrii^{\frac14}\})$ because of the estimate $\chri \sim \m$ and \eqref{p1_1}.
				
				This completes the proof of \eqref{NRE1N} under our assumptions.
				
				\subsection{Proof of {\eqref{NRE1N}} in the case where $|u| \ll 1$} \label{usmall}
				
				In such case, $d(g_u) \sim 1$, $d(g) \sim |x| \gg 1$,  $\bH(1) \sim u_1$ and  $\m \sim |x|^2 u_1 \gg 1$ (cf. \eqref{con6}). Moreover, Lemma \ref{asyin} (ii) says that $\chrii \lesssim \chri$ and $\chri \sim \ep^2 \, |x|^2 \gg 1$. Thus it follows from \eqref{pbnd1} that
				\begin{equation}\label{p_low}
					\pp_1 \gtrsim \ep^2 \, \chri^{-1} \, \m^{-\frac12} \, e^{-\m} \sim |x|^{-2} \, \m^{-\frac12} \, e^{-\m}.
				\end{equation}
				
				{\bf Estimate of $\pp_3$}.  It's very similar to the case $|u| \gtrsim 1$. Since  $u_2\leq|u|\sim|u|\,\du^{-1}$,  then on $\cw_3$, we can choose $C$ large enough so that $u_2<C|u|\du^{-1}/2 \leq \ww\w/2$. Hence it deduces from \eqref{A_pol} that $\bT_\pl^2\ge|u|^2+4\ww\w(\ww\w-u_2)\ge|u|^2+2\ww^2\w^2$, and therefore, $\m\lesssim |x|^2|u|\lesssim|x|^2\ww\w/C\lsim|x|^2\bT_\pl/C$. As a result, with $C$ large enough, by the first estimate in \eqref{sim_P} we have
				\begin{equation*}
					\frac{|x|^2}{4} \cDp(\w,\gz) > 2 \m + c \, |x|^2 \, \ww \, \w, \qquad \forall (\w,\gz) \in \cw_3.
				\end{equation*}
				Consequently, combining this with the second estimate in \eqref{sim_P} and  \eqref{p_low}, we yield
				\begin{equation} \label{nren}
					\pp_3 \lesssim |x|^2 \, \ww^2 \, \m^{-1} e^{-2 \, \m} \, \int_{\cw_3} e^{- c \, |x|^2 \, \ww \, \w} \, \w \, d\w d\gz \lsim |x|^{-2} \m^{-1} e^{-2 \, \m} = \pp_1 \, o(e^{-\m}).
				\end{equation}
				
				{\bf Estimate of $\pp_2$}. We divide this remaining estimate into two cases.
				
				{\em Case I. $u_2 \gsim u_1.$} By Table \ref{tab1} with Lemma \ref{asyin} (i)-(ii), we have $\ep \sim u_2 \, u_1^{-\frac12}$, $\chri \sim |x|^2 \, u_2^2 \, u_1^{-1} \sim \chrii
				\gg 1$ and $\chri\dz^2\sim\chrii\eta^2\sim\m^\frac14\gg1$. Now we write $\cw_2=\cw_2^*\bigcup\cw_2^{**}$ in the rectangular coordinates with
				\begin{gather*}
					\cw_2^*=\{s\in\cw_2; |s_1+1|< C_*u_1\ww^{-1}, |s_2|< C_*u_1\ww^{-1}\},\\[1mm]
					\cw_2^{**}=\{s\in\cw_2; |s_1+1|
					\ge C_*u_1\ww^{-1} \,\,{\rm or}\,\, |s_2|\ge C_*u_1\ww^{-1}\},
				\end{gather*}
				where the constant $C_*\ge1$ will be determined later.
				
				As in the proof of \eqref{H_incre} and \eqref{H_incre2}, there exists a constant $c > 0$ such that for all $(\w,\gz) \in \cw_2^*$,
					$$
					\frac{|x|^2}{4}\cDp(\w,\gz) \ge \frac{|x|^2}{4} \inf_{(\w, \gz) \in \partial \cw_1} \cDp(\w,\gz) \ge \m + c \, \m^{\frac{1}{4}}.
					$$
					As a consequence,
					\begin{equation*}
						\pp_2^* \lsim |x|^2 \, \ww^2 \, \m^{-1} \, e^{-\m} \, e^{-c \, \m^\frac14} \int_{\cw_2^*} \, ds \lsim |x|^{-2}\, \m \, e^{-\m} \, e^{-c \, \m^\frac14} = \pp_1 \, o(e^{-c' \, \m^\frac14})
					\end{equation*}
					where we have used $\m \sim |x|^2 \, u_1$ in the second ``$\lsim$'' and \eqref{p_low} in ``$=$''.
					
					To deal with $\pp_2^{**}$, we use an argument similar to that in the proof of \eqref{nren}. In brief, on $\cw_2^{**}$, recalling that $|u_2 - 2 \, \ww| \ll u_1$ (from \eqref{crieq2} and $\bH(1) \sim u_1 \ll 1$), we obtain:
					\begin{align*}
						\bT_\pl^2 &= u_1^2+(u_2+2 \, \ww s_1)^2+4\ww^2 \, s_2^2=u_1^2+(u_2-2 \, \ww+2 \, \ww \, (s_1+1))^2+4\ww^2 \, s_2^2 \\
						\mbox{} &\ge \left( 2 \, \ww \, |s_1 + 1| - | u_2-2 \, \ww | \right)^2+4\ww^2 \, s_2^2 \ge (\ww \, |s_1 + 1|)^2 + 4\ww^2 \, s_2^2 \ge C_*^2 u_1^2,
					\end{align*}
					for $C_*$ large enough. Then by the first estimate in \eqref{sim_P}, there is a constant $c > 0$ such that
					\[
					\frac{|x|^2}{4} \, \cDp \ge 2 \, \m + c \, |x|^2 \, ( \ww \, |s_1 + 1| + 2 \, \ww |s_2|),
					\]
					which further implies that
					\begin{equation*}
						\pp_2^{**} \lsim |x|^2 \, \ww^2 \, \m^{-1} \, e^{-2 \, \m} \int_{\cw_2^{**}} e^{-c \, |x|^2 \, ( \ww \, |s_1 + 1|
							+ 2 \, \ww |s_2|)} \, ds \lsim |x|^{-2} \, \m^{-1} \, e^{- 2 \, \m} = \pp_1 \, o(e^{-\m}).
					\end{equation*}
					Thus we have shown that $\pp_2=\pp_2^*+\pp_2^{**}=o(\pp_1)$.
					
					{\em Case II. $u_2 \ll u_1.$} From Table \ref{tab1} we can read that $\ep\ll1$ and $\chri\dz^2\sim\chri^\frac14\gg1$. There are two subcases as well: $\chrii\lsim1$ and $\chrii\gg1$.
					
					{\em Subcase II-1. $\chrii\lsim1.$}
					In such subcase, by \eqref{pbnd1} and the fact that $\ep \, \m + \chrii \sim \chri \sim \ep^2 \, |x|^2$ (cf. Lemma \ref{asyin} (ii)), we have
					\begin{equation}\label{p_low2}
						\pp_1\gsim \ep^2\chri^{-\frac32}e^{-\m}\sim|x|^{-2}\chri^{-\frac12}e^{-\m}.
					\end{equation}
					
					Recall that $\eta=\pi$ (cf. \eqref{delg}). Fix selected constant $C$. Without loss of generality, we may assume that $C u_1 \le C |u| \ll 1$. And  we decompose $\cw_2 = \cw_{2}^\sharp \bigcup \cw_{2}^{\sharp\sharp}$  with
					\begin{gather*}
						\cw_{2}^\sharp = \{(\w,\gz);\w\le C_\sharp, \,  |\w-1|>\dz\},\quad
						\cw_{2}^{\sharp\sharp} =\{(\w,\gz);C_\sharp<\w< C\,|u|\,\du^{-1}\ww^{-1}\},
					\end{gather*}
					where the constant $C_\sharp \ge2$ will be chosen later. The argument used in the estimation  \eqref{H_incre} shows that we have uniformly on $\cw_2^\sharp$, $\frac{|x|^2}{4} \cDp(\w,\gz) \ge \m + \frac12\chri\dz^2+o(1)$. Then, using \eqref{p_low2}, a simple calculation yields $\pp_{2}^\sharp = \pp_1 \, O_{C_\sharp}(e^{-c \, \chri^\frac14})$.
					
					The estimate for  $\pp_{2}^{\sharp\sharp}$ is somewhat more complicate.  Recall that $\bA(1) \sim \bH(1) \sim u_1$ (cf.  Lemma \ref{asyin} (i) and \eqref{con6}). Since now it holds that $u_1\sim|u|\gg u_2$ and $\ww \sim \ep u_1^\frac12 \sim u_1^\frac32 + u_2 \ll u_1$ by \eqref{con6}, then from the first equality in \eqref{A_pol} and \eqref{nBNn}, a direct calculation shows that, for any $\w<C\,|u|\,\du^{-1}\ww^{-1}$ and $\gz\in[-\pi,\pi]$,
					\begin{equation}\label{AU}
						u_1 + \w \ww \sim  \bT_\pl \sim_C \bA(1) \sim u_1\ll1,\quad \u(1)\sim \frac{1}{\ep^2} \gg1, \quad \L_\pl = \frac{\bT_\pl}{\ww^2 \, \w^2}  \gtrsim \frac{1}{\w \ww} \gtrsim \frac{1}{C u_1} \gg1.
					\end{equation}
					Consequently, we can invoke \eqref{eD4} and \eqref{aPsi} to obtain
					\begin{gather*}
						\frac{|x|^2}{4}\cDp(\w,\gz)=\frac{|x|^2}{4} \ww^2 \, \w^2 \Phi(\L_\pl) \ge \frac{|x|^2}{4}\left(\vartheta_1 \bT_\pl - 2 \sqrt{\vartheta_1 \bT_\pl}\,\ww \, \w + \frac{3}{2} \ww^2\w^2\right),\\
						\m=\frac{|x|^2}{4}\ww^2 \, \Phi(\u(1)) \le \frac{|x|^2}{4}\left(\vartheta_1 \bA(1) - 2 \sqrt{\vartheta_1 \bA(1)}\, \ww + \frac{5}{2} \ww^2\right),
					\end{gather*}
					which follows that
					\begin{equation}\label{d1}
						\frac{|x|^2}{4}\cDp - \m \ge \frac{|x|^2}{4}\left[(\sqrt{\vartheta_1 \bT_\pl} - \ww\w)^2 - (\sqrt{\vartheta_1 \bA(1)} - \ww)^2 + \frac{1}{2}\ww^2(\w^2-3)\right].
					\end{equation}
					Note that $\w^2\ge4$ on $\cw_2^{\sharp\sharp}$ and $u_2\lsim\ww$. Then by \eqref{AU} we can choose $C_\sharp$ properly such that, for any $(\w,\gz)\in\cw_{2}^{\sharp\sharp}$,
					$$
					\begin{aligned}
						\sqrt{ \bT_\pl} - \sqrt{\bA(1)} &= \frac{4u_2\ww (1-\w\cos\gz) + 4\ww^2 (\w^2 - 1)}{(\sqrt{ \bT_\pl} + \sqrt{\bA(1)})(\bT_\pl +\bA(1))}\,
      \gsim \,\frac{\ww\w \, (\ww \w -  2 \, u_2) }{(u_1 + \ww \w)^\frac32}   \\
      &\gsim \min\left\{ u_1^{-\frac32} \ww \w, (\ww \w)^{-\frac12} \right\} \ww \w \gsim \min\left\{C_\sharp, (C u_1)^{-\frac12} \right\} \ww\w,
					\end{aligned}
					$$
					where we have used \eqref{A_pol} in the equality and $\ww \gtrsim u_1^{\frac32}$ in the last inequality. Hence
					\begin{equation}\label{d2}
						\sqrt{\vartheta_1 \bT_\pl} - \sqrt{\vartheta_1 \bA(1)} -\ww\w + \ww \, \gsim \, u_1^{-\frac32}\ww^2\w^2/C^2,
					\end{equation}
					as long as $C_\sharp$ is large enough. Besides, observing on $\cw_2^{\sharp\sharp}$ it holds that $\ww<\ww\w \, \lsim \, C \, u_1$, thus by \eqref{AU} we get
					\begin{equation}\label{d3}
						\sqrt{\vartheta_1 \bT_\pl} + \sqrt{\vartheta_1 \bA(1)} -\ww\w -\ww  \, \gsim \,
						u_1^\frac12.
					\end{equation}
					Consequently, it follows from \eqref{d1}-\eqref{d3} that
					\begin{equation}\label{schrD}
						\frac{|x|^2}{4}\cDp(\w,\gz) -\m \,
      \gsim \, u_1^{-1}|x|^2 \, \ww^2 \, \w^2/C^2.
					\end{equation}
					On the other hand, as in the proof of \eqref{H_incre},
					we also have $\frac{|x|^2}{4}\cDp-\m\gsim\chri\dz^2\sim\chri^\frac14\gg1$. Hence there is a positive constant $c$ such that
					\begin{equation}\label{schrD2}
						\frac{|x|^2}{4} \cDp(\w,\gz) - \m \ge  c \, u_1^{-1} \, |x|^2 \, \ww^2 \,
						\w^2 + c \, \chri^\frac14, \quad\forall \, (\w,\gz)\in\cw_2^{\sharp\sharp}.
					\end{equation}
					By the second estimate in \eqref{sim_P} again, we yield that
					\begin{equation*}
						\pp_2^{\sharp\sharp} \lsim |x|^2 \, \ww^2 \, \m^{-1} \, e^{-c \, \chri^\frac14} \, e^{-\m} \int_0^\infty e^{-c \, u_1^{-1} \, |x|^2 \,  \ww^2 \, \w^2} \, \w d\w \lsim |x|^{-2} \, e^{-c \, \chri^\frac14} \, e^{-\m}=o(\pp_1)
					\end{equation*}
					where we have used that $\m \sim |x|^2 \, u_1$ in the second ``$\lsim$'', and \eqref{p_low2} in ``$=$''. In conclusion, $\pp_2=\pp_{2}^\sharp+\pp_{2}^{\sharp\sharp}=o(\pp_1)$ and the proof for $\chrii\lsim1$ is finished.
					
					{\em Subcase II-2. $\chrii\gg1.$} We decompose $\cw_2$ as in \eqref{setw2} and write $\pp_2=\pp_2'+\pp_2''$. To show $\pp_2''=o(\pp_1)$ we can argue as in {\em Subcase II-1}, in fact the two key estimates \eqref{p_low2} and \eqref{schrD2} (on $\cw_2''$ instead of $\cw_2^{\sharp\sharp}$) are still valid.
				
				To show $\pp_2'=o(\pp_1)$ we can argue as in the case where $\chrii\gg1$ in Subsection \ref{ss71}, whereas the key estimates at this point we yield are \eqref{L2est} with the  non-exponential factor replaced by $\m^{-1} \,(1+\ep\,\m)^{-\frac12}$ and \eqref{pbnd1}.
				
				In summary, the estimate \eqref{NRE1N} is established. To complete the proof of Theorem \ref{bigL1}, we are left to provide the
				
				\subsection{Proof of Proposition \ref{Dsim}} \label{ss73}
				
				\begin{proof}
					The proof of \eqref{simpD} is elementary, mainly by means of Taylor's formula and Chain Rule.  However, it is very cumbersome and we have to make use of Lemma \ref{asyin} and Table \ref{tab1} by distinguishing three cases therein repeatedly, as explained earlier.
					
					Let us begin with the ``radial'' part. Under our assumptions, it holds that for any $|\w-1|\le\dz$,
					\begin{equation}\label{firstE}
						\begin{gathered}
							\bA \sim \mz, \quad |\bA^\prime| \lesssim \ep^2 \du^2  \, \mz|u|^{-1}, \quad |\bA^{\prime\prime}| \lesssim  \ep^2 \,  \du^2, \quad
							|\bA^{\prime\prime\prime}| \lesssim \ep^4  \, \du^4|u|^{-1}, \\[2mm]
							\u \sim \epsilon^{-2}, \quad |\u^\prime| \lesssim \ep^{-2} \, \mz^{-1} \, |u|, \quad |\u^{\prime\prime}| \lesssim \ep^{-2} \, \mz^{-2} \, |u|^2, \quad
							|\u^{\prime\prime\prime}| \lesssim \ep^{-2} \, \mz^{-3} \, |u|^3.
						\end{gathered}
					\end{equation}
					Here we only prove the first two estimates, which together with \eqref{dws}-\eqref{dws3}, the fact that $d(g_u)^2 \sim 1 + |u|$ (cf. \eqref{ehd}), and Table \ref{tab1} with Lemma \ref{asyin} will allow us to conclude the remaining claims.
					
					To this end, recall the definition of $\bA(\cdot)$ (cf. \eqref{defu}), \eqref{crieq2} and the fact that $\ww\sim \ep\,\mz^\frac12$. Then
					$$\begin{aligned}|\bA(\w)^2-\bA(1)^2|&=4\ww \, |\w-1|\,|u_2-\ww(\w+1)|\le4\ww \, |\w-1|(|u_2-2\ww|+\ww \, |\w-1|)\\[1mm]
						&\lsim\dz\ep\,\mz^\frac12 \, \left[\mz^\frac32+\ep\,\mz^\frac12 \, \dz\right]\ll\mz^2\sim\bA(1)^2,
					\end{aligned}
					$$
					where ``$\sim$" trivially holds since at this point $|\tz|\sim1$ and $\bA(1)\sim|\tz|\mz$
					(cf. Lemma \ref{asyin} (i)). We explain how to obtain ``$\ll$'', or equivalently, $\delta^2 \, \ep^2 \ll \mz$.  Indeed, its proof is based on Table \ref{tab1} and Lemma \ref{asyin}. For instance, we consider the subtlest situation: $|u| \ll 1$ and $u_2 \gtrsim u_1$. In such case, Table \ref{tab1} says that $\ep \sim u_2/\sqrt{u_1}$, $\mz \sim u_1$ and $\delta \sim |x|^{-3/4} \, u_2^{-1} \, u_1^{5/8}$. Hence $\mz/(\delta^2 \, \ep^2) \sim (|x|^2 \, u_1)^{3/4} \sim \m^{3/4} \gg 1$, which implies the desired estimate.  As a consequence we obtain $\bA\sim\bA(1)\sim\mz$.
					
					Similarly, $|2\ww\w-u_2|\le|u_2-2\ww|+2\ww|\w-1|\lsim \mz^\frac32+\ep\,\mz^\frac12\dz$, thereby  from \eqref{dws}
\begin{align}\label{estAp}
				|\bA'|\lsim\ep\,\mz^\frac12 \, \frac{\mz^\frac32+\ep\,\mz^\frac12 \, \dz}{\mz}\sim\ep\,\mz+\dz \, \ep^2\lsim \ep^2  \, \du^2  \, \mz \, |u|^{-1},
\end{align}
					where the last ``$\lsim$" can also be checked using Table \ref{tab1}  and Lemma \ref{asyin} again.
				
				Note that we have shown in \eqref{firstE} that $\u\sim \ep^{-2} \gsim 1$.  It follows from the first estimate in \eqref{Ups_est} and the definition of $\cZ(\cdot)$ (cf. \eqref{cZd}) that
				\[
				\frac{\cZ(\u)}{(\var - \cZ(\u))^2} \sim  - \Upsilon'(\cZ(\u)) = \u \sim \ep^{-2} \gsim 1,
				\]
				which implies that $\cZ(\u) \sim 1$ and $\var - \cZ(\u) \sim \ep$.
				
				Thus using \eqref{dws6}, \eqref{firstE}, \eqref{Z_est} and Table \ref{tab1}, we can show that:
				$$|\bH^{\prime\prime\prime}| \lesssim \ep^{2} \, \du^2 \, \mz^{-1} \, |u|. $$
				Hence by the mean value theorem, we get
				\begin{gather}\label{Hpp5}
					\bH(\w) = \bH(1) + \frac{1}{2} \bH^{\prime\prime}(1) (\w - 1)^2 + O( \delta^3 \, \ep^2 \,  \du^2 \, \mz^{-1} \, |u|),\quad\forall \,  \w\in[1-\dz,1+\dz].
				\end{gather}
				
				Now we handle the ``angular'' part. Recall that  $\m = |x|^2 \, \mz/4$, $\ww \sim \ep \, \mz^\frac12$, and $\chrii \sim \ep u_2 |x|^3 \m^{-\frac12}$ (cf. Lemma \ref{asyin} (i)-(ii)), it follows from the first estimate in \eqref{firstE} that:
				\begin{equation}\label{Ds_con}
					\frac{2 u_2 \, \ww \, \w (1 - \cos{\gz})}{\bA(\w)^2} \sim \ep \, \mz^{-\frac32} u_2 \, \gz^2
					\lsim \chrii \, \eta^2 \, \m^{-1} \ll1 , \quad \forall (\w,\gz)\in \cw_1.
				\end{equation}
				Here to show ``$\ll$'', we have distinguished the following cases: (I) $\chrii \lsim 1$; (II) $\chrii \gg 1$ with: (II-1) $|u| \gtrsim 1$; (II-2a) $|u| \ll 1$ and $u_2 \gtrsim u_1$; and (II-2b) $|u| \ll 1$ and $u_2 \ll u_1$. We only show the subtlest case (II-2b), and the others are trivial. In such case, from Lemma \ref{asyin} and Table \ref{tab1}, we yield
				\[
				\frac{\chrii \, \eta^2}{\m} \lsim \frac{\chrii^{\frac{1}{4}}}{\m} \lsim  \frac{\chri^{\frac{1}{4}}}{\m}  \sim \left(  \frac{|x|^2 \, \ep^2}{(|x|^2 \, u_1)^4} \right)^{\frac{1}{4}} \lsim \left(  \frac{|x|^2 \, u_1}{(|x|^2 \, u_1)^4} \right)^{\frac{1}{4}} \lsim \m^{-3/4} \ll 1.
				\]
				
				Therefore,  using \eqref{defu} and the Taylor's expansion of $\sqrt{1 + \rho}$ at $0$, we can write $\L_\pl(\w,\gz)$ as
				\begin{align}\nonumber
					\ww^{-2}\w^{-2} \bT_\pl(\w,\gz) &=\ww^{-2} \w^{-2} \sqrt{|u|^2 + 4\ww^{2}\w^2 - 4 u_2\ww \,  \w \cos{\gz} } \\
					\nonumber
					&=  \,\ww^{-2}\w^{-2} \bA(\w) \left( 1 + \frac{4 u_2\ww \,  \w(1 -  \cos{\gz})}{\bA(\w)^2}\right)^{\frac{1}{2}} \\
					\nonumber
					&= \, \u(\w) \left[ 1 + \frac{2u_2\ww \,  \w(1 - \cos{\gz})}{\bA(\w)^2} + O\left(\ep^2 \, \mz^{-3}u_2^2 \gz^4\right)\right] \\
					\label{prep2}
					&= \, \u(\w) + \frac{2u_2 (1 -  \cos{\gz})}{\ww \, \w \bA(\w)} + O\left(\mz^{-3}u_2^2\gz^4\right),
				\end{align}
where we have used $\u(\w) \sim \ep^{-2}$ (cf. \eqref{firstE}) in the last equality.
				In particular, from \eqref{Ds_con} again,  we find that
				\begin{equation}\label{prep6}
					|\L_\pl-\u(\w)| \lsim \ep^{-1}\mz^{-\frac32}u_2\gz^2 \ll\ep^{-2}, \qquad (\w,\gz) \in \cw_1.
				\end{equation}

Now applying the Taylor's expansion to the function $\cZ$ at $\u(1)$ (also recalling that $\cZ(\u(1)) = |\tz|$ by \eqref{nBNn}), we obtain that
\begin{align*}
&|(\var-\cZ(\L_\pl)) - \ep| = |(\var-\cZ(\L_\pl)) - (\var - |\tz|)| = |\cZ(\L_\pl) - \cZ(\u(1))| \\
\le & \sup_{\iota \in [0,1]} |\cZ'(\iota  \L_\pl + (1 - \iota) \u(1))| \, |\L_\pl-\u(1)|.
\end{align*}
It follows from \eqref{prep6} and \eqref{firstE} that for all $\iota \in [0,1]$ we have $\iota \u(1) + (1 - \iota) \L_\pl \sim \ep^{-2}$ and thus by \eqref{asP}
\[
\sup_{\iota \in [0,1]} |\cZ'(\iota \u(1) + (1 - \iota) \L_\pl)| = O(\ep^3),
\]
which implies
\begin{align}\label{prep13}
|(\var-\cZ(\L_\pl)) - \ep| \ll\ep,  \qquad (\w,\gz) \in \cw_1
\end{align}
and also
\begin{align}\label{prep133}
|\cZ(\u(\w)) - |\tz|| = |(\var-\cZ(\u(\w)) - \ep| = O(\ep^3 \, |\u(\w) - \u(1)| ) = O(\dz \, \ep \, \mz^{-1} \, |u|)
\end{align}
by another Taylor's expansion and \eqref{firstE}.

Similarly, using the Taylor's expansion to the function $\Phi$ at $\u(\w)$ gives
\begin{align} \nonumber
&\Phi(\L_\pl) = \Phi(\u(\w)) + \cZ(\u(\w)) (\L_\pl - \u(\w)) + O(\ep^3 \, |\L_\pl - \u(\w)|^2) \\
\label{prep3}
= & \,  \Phi(\u(\w)) + \cZ(\u(\w)) \frac{2u_2 \,  (1 -  \cos{\gz})}{\ww \, \w\bA(\w)} + O\left(\mz^{-3} \, u_2^2 \,  \gz^4\right),
\end{align}
where in the first ``$=$'' we have used \eqref{deZP} and \eqref{asP}, and in the second ``$=$'' \eqref{prep13} and \eqref{prep6}. Writing
\[
\frac{\cZ(\u(\w)) \, \w}{\bA(\w)}-\frac{|\tz|}{\bA(1)} = \frac{\cZ(\u(\w)) \, (\w - 1)}{\bA(\w)} + \frac{\cZ(\u(\w)) - |\tz|}{\bA(\w)} + |\tz| \left( \frac{1}{\bA(\w)} - \frac{1}{\bA(1)}  \right),
\]
by \eqref{firstE}, the first term is bounded by $\dz \, \mz^{-1}$. Using \eqref{prep133} in addition, the second term is bounded by
\[
\dz \,\ep \, \mz^{-2} \, |u|  \lsim \dz \, \ep^2  \, \du^2 \, \mz^{-1} \, |u|^{-1}.
\]
As before, the ``$\lsim$" here can be checked using Table \ref{tab1}  and Lemma \ref{asyin}. For the last term, we just use Taylor's expansion, together with \eqref{firstE} again to obtain the same upper bound. In conclusion, we obtain that
\begin{gather}
					\left|\frac{\cZ(\u(\w)) \, \w}{\bA(\w)}-\frac{|\tz|}{\bA(1)}\right|
					\lsim \dz\left[\mz^{-1}+\ep^2  \, \du^2 \, \mz^{-1} \, |u|^{-1}\right]. \label{prep14}
				\end{gather}			
				
				Furthermore, substituting \eqref{prep3}, \eqref{Hpp5} and \eqref{prep14} into $\cDp =\ww^2 \w^2 \, \Phi(\L_\pl)$, we can yield:
				$$\begin{aligned}
					\cDp(\w,\gz) &= \bH(\w) + \cZ(\u(\w)) \frac{2u_2 \, \ww \, \w(1 -  \cos{\gz})}{\bA(\w)} + O\left(\ep^2 \,  \mz^{-2} \, u_2^2 \, \gz^4 \, \right) \\
					&= \bH(1) + \frac12 \bH^{\prime\prime}(1) (\w-1)^2 +  \frac{4\chrii}{|x|^2}(1 -  \cos{\gz})+o(|x|^{-2}),
				\end{aligned}$$
				where in the last  equality, we have used  $|\gz|\le\eta$ and repeated the trick of combining Table \ref{tab1} with Lemma \ref{asyin}. This establishes \eqref{simpD}.
				
				It remains to show \eqref{ma1}.  By Lemma \ref{q_lem}, \eqref{prep13} and the mean value theorem we can check \begin{align*}
					\q(\cZ(\L_\pl))=\q(\atz) \, (1+o(1)), \qquad \frac{(\vtz_1-\cZ(\L_\pl))^{-\frac12}}{|x|^2 \, \ww^2 \, \w^2}=\frac{\ep^{-\frac12}}{|x|^2 \, \ww^2}(1+o(1)),
				\end{align*}
				and temporarily take it for granted that
				\begin{equation}\label{Ds1}
					e^{-\frac{\vtz_1|x|^2 \, \ww^2\w^2}{2(\vtz_1-\cZ(\L_\pl))}}I_0\left(\frac{\vtz_1|x|^2 \, \ww^2\w^2}{2(\vtz_1-\cZ(\L_\pl))}\right)=e^{-\frac{\vtz_1|x|^2 \, \ww^2}{2\ep}}I_0\left(\frac{\vtz_1|x|^2 \, \ww^2}{2\ep}\right)(1+o(1)),
				\end{equation}
				then by \eqref{def_cP}, Proposition \ref{lP} (iii) (with the fact that  $\frac{|x|^2}{4}\cDp(\w,\gz)\ge\m\gg1
				$) and \eqref{simpD}, we get \eqref{ma1}.
				
				Finally, to prove \eqref{Ds1},
				the above procedure allows us to obtain:
				\begin{equation}\label{Ds2}
						\begin{aligned}
							\frac{\vtz_1|x|^2 \, \ww^2 \, \w^2}{2(\vtz_1-\cZ(\L_\pl))} = \frac{\vtz_1|x|^2 \, \ww^2}{2\ep}(1+o(1)) \sim \ep \, \m.
						\end{aligned}
					\end{equation}
					Thus, if $\ep \, \m \gg1$ then \eqref{Ds1} follows from \eqref{asinI0}; if $\ep \, \m \lsim 1$ then from \eqref{Ds2} we see that
					\begin{equation*}
						\frac{\vtz_1|x|^2 \, \ww^2\w^2}{2(\vtz_1-\cZ(\L_\pl))}=\frac{\vtz_1|x|^2 \, \ww^2}{2\ep}+o(1),
					\end{equation*}
					which also leads to \eqref{Ds1},  since $e^{-r} I_0(r)$ is smooth and positive on $\R$.
				\end{proof}

				\begin{remark} \label{Vara2}
					From the above proof and Remark \ref{rem67}, we observe that Theorem \ref{bigL1} still holds if the condition ``$|\tz|\ge1$" is replaced by ``$|\tz|\ge\az_0$  with $\az_0\in(0,1]$", the cost of which is that the remainder ``$o_{\zeta_0}(1)$" will depend on one more parameter $\az_0$ than before. As before, the choice of $\zeta_0$ depends on $\az_0$.
				\end{remark}
				
				\section{Uniform asymptotics for $\chri\lesssim1$ and $\m\to\infty$}\label{schr}
				\setcounter{equation}{0}
			
				We continue to use the notation as before, and for the sake of brevity we still
omit the dependence of undetermined constants in the proof of our theorems.
    The aim of this section is to establish the following:
				
				\begin{theo}\label{sL1}
					Let $\zeta_0 > 0$.  Then there is a constant $C(\zeta_0) \gg 1$ such that, for all $g$ satisfying  $\chri \le \zeta_0$ and $\m \ge C(\zeta_0)$,
					$$p(g) = 4\pi \, \frac{\vartheta_1^2 }{ - \sin{\vartheta_1}  }\, e^{-\frac{d(g)^2}{4}}
     \int_0^{\w_0}\int_{-\pi}^\pi  \widetilde{\cP}_\pl(\w,\gz)\, d\w \,d\gz\,(1 +  o_{\zeta_0}(1)),
					$$
					where $\w_0:=|x|^{-\frac12}\, u_1^\frac34\,\ww^{-1}$ and
					\begin{align}
						\widetilde{\cP}_\pl(\w,\gz) &:= \exp \left(-\frac{|x|^2}{4} \left[ \left(\sqrt{\vartheta_1 \bT_\pl(\w,\gz)} - \ww^2\w\right)^2 - \left(\sqrt{\vartheta_1 \bA(1)} - \ww\right)^2 \right] \right) \nonumber\\
						&\qquad \cdot \exp\left(-\frac{ \vartheta_1 |x|^2\ww^2\w^2}{2(\vartheta_1 - \cZ(\L_\pl))} \right)
						I_0 \left(  \frac{ \vartheta_1 |x|^2\ww^2\w^2}{2(\vartheta_1 - \cZ(\L_\pl))} \right) \,
						\frac{ (\vartheta_1 - \cZ(\L_\pl))^2}{\w}. \label{tild_P}
					\end{align}
				\end{theo}

				Recall that even in the special $5$-dimensional non-isotropic Heisenberg group, it indeed happens that the leading term of the asymptotic expansion for the heat kernel cannot be represented in an explicit expression but an integral form. See \cite[Theorem 4]{LZ19}. Hence we believe that in our group $N_{3,2}$, which is $6$-dimensional, it is unpractical to find all the explicit expressions  as in the setting of isotropic Heisenberg groups (cf. \cite{Li07}). Nevertheless, in such case precise estimates for the heat kernel can still be provided in a very concise form:
				
				\begin{theo}\label{sL1*}
					Let $\zeta_0 > 0$.  Then there is a constant $C(\zeta_0) \gg 1$ such that, for all $g$ satisfying  $\chri \le \zeta_0$ and $\m \ge C(\zeta_0)$,
					\begin{equation} \label{psL1}
						p(g) \sim_{\zeta_0}|x|^{-2}\, e^{-\frac{d(g)^2}{4}}.
                   \end{equation}
				\end{theo}
				
				Let us begin with the
				\begin{proof}[Proof of Theorem \ref{sL1*}]
					Under our assumptions, it deduces from Lemma \ref{asyin}  that $1 \ll \m \lsim |x|^2 \, |u| \lesssim d(g)^2$ and $\ep^2 d(g)^2 \sim \chri \lsim1$, which yield  $d(g)^2 \gg 1$ and $\ep \ll 1$. This occurs only in the case of Lemma \ref{asyin} (iv), whence $|u|\ll1$ and \eqref{con6} holds. In particular, $\m = \bH(1)\, |x|^2/4 \sim |x|^2 \, u_1$. Now from \eqref{con6}, $d(g)^2 \sim |x|^2$ and $\chri\sim\ep^2|x|^2\lesssim1$ force that
					\begin{align}\label{con10}
						u_1 \lesssim |x|^{-1} \ll 1, \qquad u_2 \lesssim
						|x|^{-1} \, u_1^\frac12 \sim \frac{u_1}{\sqrt{\m}} \ll u_1, \qquad \ww \sim \epsilon \, u_1^\frac12 \lesssim |x|^{-1} \, u_1^\frac12.
					\end{align}
					
					We split the integral \eqref{ehk4},  by means of the modified polar coordinates \eqref{pol_cor}, into the following three parts:
						\begin{align*}
							p(g) &= \frac{1}{4\pi} \, e^{-\frac{|x|^2}{4}}|x|^2 \, \ww^2 \int_0^{+\infty} \int_{-\pi}^\pi \w \, \cP_\pl(\w, \gz)  \, d\w d\gz \nonumber\\[1mm]
							&= \frac{1}{4\pi} e^{-\frac{|x|^2}{4}} \, |x|^2 \, \ww^2 \left( \int_{\Xi_1} + \int_{\Xi_2} + \int_{\Xi_3} \right)  =: \frac{1}{4\pi} e^{-\frac{|x|^2}{4}} \sum_{i=1}^3 L_i  =: \frac{1}{4\pi} e^{-\frac{|x|^2}{4}}L,
						\end{align*}
					where
					\begin{gather}
						\Xi_1 := \{(\w,\gz); \w \le C^\prime \, |x|^{-1} \, u_1^\frac12 \, \ww^{-1} \}, \\[1mm]
						\Xi_2 := \{(\w,\gz); C^\prime \,  |x|^{-1} \, u_1^\frac12 \, \ww^{-1} < \w \le C \, u_1 \, \ww^{-1}  \}, \\[1mm]
						\Xi_3 := \{(\w,\gz); \w > C u_1 \, \ww^{-1}  \}.
					\end{gather}
					Here $C$ and $C^\prime$ are two constants to be chosen later.
					
					The choice of $C$ and the estimate of $L_3$ are similar to those for $\pp_3$ in Subsection \ref{usmall}, which yield that
					\begin{equation}\label{L3}
						L_3 \lesssim |x|^{-2} \, \m^{-1}  e^{- 2 \, \m}.
					\end{equation}
					
					The estimate of $L_2$ follows in a similar argument  but with some simplification for $\pp_2^{\sharp\sharp}$ as in Subsection \ref{usmall}. Indeed, using \eqref{con10} one can check \eqref{schrD} is still valid on $\Xi_2$ provided that $C'$ is large enough. Consequently, by  the second claim in \eqref{sim_P} we have
					\begin{equation}\label{L2}
						L_2 \lesssim |x|^2 \, \ww^2 \, \m^{-1} \, e^{-\m}\int_{\Xi_2} e^{-c \, u_1^{-1} \, |x|^2 \, \ww^2 \, \w^2} \w \, d\w d\gz \lsim |x|^{-2} \, e^{-\m}.
					\end{equation}
					
					Aiming now at $L_1$, as in the proof of \eqref{AU}, we can show that $\u(1)\sim\ep^{-2}\gg1$, $\bT_\pl \sim \bA(1) \sim u_1$ and  $\L_\pl  = \bT_\pl/(\ww^2 \, \w^2)  \gtrsim |x|^2 \gg1$ on $\Xi_1$.
					Then an application of \eqref{aPsi} with the fact that $|x| u_1\lsim1$ gives
					\begin{equation*}
						\frac{|x|^2}{4}\cDp = \frac{|x|^2}{4} \ww^2 \w^2 \Phi(\L_\pl) =  \frac{|x|^2}{4}\vartheta_1 \bT_\pl + O(1), \qquad \m = \frac{|x|^2}{4} \ww^2 \Phi(\u(1)) =  \frac{|x|^2}{4}\vartheta_1 \bA(1) + O(1).
					\end{equation*}
					Moreover, it follows from \eqref{A_pol} and \eqref{con10} that
					\begin{align*}
						&|\bT_\pl - \bA(1)| = \frac{|\bT_\pl^2 - \bA(1)^2|}{\bT_\pl +\bA(1)} \le \frac{4u_2 \, \ww \, (\w + 1) + 4 \, \ww^2 \, ( \w^2 +1)}{\bT_\pl +\bA(1)} \lesssim |x|^{-2}.
					\end{align*}
					Hence $\frac{|x|^2}{4}\cDp = \m + O(1) (\sim |x|^2 u_1) \gg1$ and therefore, $|x| \, \ww \, \w \cdot \sqrt{|x|^2 \,  \cDp}  \lesssim u_1 |x| \lesssim 1$ on $\Xi_1$. As a result, from \eqref{ehk4} and Proposition \ref{lP} (ii), we get
					\begin{equation}\label{L1}
						L_1 \sim |x|^2 \, \ww^2 \, \m^{-1} \,(|x|^{-1} \, u_1^\frac12 \, \ww^{-1})^2 \, e^{-\m} \sim |x|^{-2} \, e^{-\m}.
					\end{equation}
					
					To summarize, by \eqref{L3}-\eqref{L1}  and the assumption that $\m\gg1$, we obtain
					\begin{equation}
						L\sim |x|^{-2} e^{-\m}, \qquad p(g) \sim |x|^{-2} e^{-\m - \frac{|x|^2}{4}}.
					\end{equation}
					
This together with the fact that $d(g)^2=|x|^2+4\m$ concludes the proof of Theorem \ref{sL1*}.
				\end{proof}
				
				Now, we return to the
				
				\begin{proof}[Proof of Theorem \ref{sL1}]
					We split the integral \eqref{ehk4} again and write this time
		\begin{align*}
						p(g) = \frac{1}{4\pi} \, e^{-\frac{|x|^2}{4}} |x|^2 \, \ww^2 \left( \int_{\widetilde{\Xi}_1} + \int_{\widetilde{\Xi}_2} + \int_{\widetilde{\Xi}_3} \right)  =:  \frac{1}{4\pi} e^{-\frac{|x|^2}{4}} \, (\widetilde{L}_1 + \widetilde{L}_2 + \widetilde{L}_3)  =:  \frac{1}{4\pi} e^{-\frac{|x|^2}{4}}\widetilde{L},
					\end{align*}
					where $\widetilde{\Xi}_3$ is equal to $\Xi_3$ and
					\begin{align*}
						\widetilde{\Xi}_1 := \{(\w,\gz); \w < |x|^{-\frac12} \, u_1^\frac34 \, \ww^{-1}\}, \quad
						\widetilde{\Xi}_2 := \{(\w,\gz); |x|^{-\frac12} \, u_1^\frac34 \, \ww^{-1} \le \w \le C u_1 \, \ww^{-1} \}.
					\end{align*}
					
					Recall that we have already shown that $\widetilde{L}_3 = L_3 = L_1 \, O(e^{-\m})$. To treat $\widetilde{L}_2$, remark that
					\[
					\frac{|x|^{-\frac{1}{2}} \, u_1^{\frac{3}{4}} \, \ww^{-1}}{|x|^{-1} \, u_1^{\frac{1}{2}} \, \ww^{-1}} = \left( |x|^2 \, u_1 \right)^{\frac{1}{4}} \sim \m^{\frac{1}{4}} \gg 1.
					\]
					Then one follows the argument just given for $L_2$. The main difference is that we can write now $\frac{|x|^2}{4} \cDp - \m \ge c \, u_1^{-1} \, |x|^2 \, \ww^2 \, \w^2 + c \, \m^\frac12$  with the additional term $c \, \m^\frac12$,  due to the fact that $u_1^{-1} \, |x|^2 \, \ww^2 \, \w^2 \gtrsim \m^{\frac{1}{2}}$ on $\widetilde{\Xi}_2$. Thus,
					\begin{align*}
						\widetilde{L}_2 \lsim |x|^2 \, \ww^2 \, \m^{-1} \, e^{-\m} \, e^{-c \, \m^\frac12} \int_{\widetilde{\Xi}_2}  e^{- c \, u_1^{-1} \, |x|^2 \, \ww^2 \, \w^2}  \w \, d\w d\gz \lsim |x|^{-2} \, e^{-\m} \, e^{- c \, \m^\frac12} = L_1 \, O(e^{- c \, \m^\frac12}).
					\end{align*}
					Notice that $\widetilde{\Xi}_1\supseteq \Xi_1$, then the above estimates imply that $\widetilde{L}=\widetilde{L}_1(1+o(1))$.
					
					On the other hand, restricting ourselves to the region $\widetilde{\Xi}_1$, some simplifications can be made to the expression of the heat kernel. In fact, using \eqref{con10} it can be checked without difficulties that $\bT_\pl \sim\bA(1)\sim u_1$, $\u(1)\sim\ep^{-2}\gg1$, and $\L_\pl \sim \ww^{-2}\w^{-2} u_1 \gtrsim |x|u_1^{-\frac12}\gg1$. Then by the first estimate in \eqref{asP} we have $\vtz_1-\cZ(\L_\pl)\sim\L_\pl^{-\frac12}\lsim|x|^{-\frac12 }u_1^\frac14\ll1$, which, together with \eqref{q_var}, implies that
					$$\q(\cZ(\L_\pl))=\frac{\sqrt{2\vtz_1}\vtz_1}{-2\sin\vtz_1}(\vtz_1-\cZ(\L_\pl))^\frac52(1+o(1)).
					$$
					And moreover, we can apply \eqref{aPsi} with the fact that $|x| u_1^\frac32\lsim \sqrt{u_1}\ll1$ and obtain that, for all $(\w,\gz)\in\widetilde{\Xi}_1$,
					\begin{align*}
						\frac{|x|^2}{4}\cDp(\w,\gz) = \frac{|x|^2}{4}(\sqrt{\vartheta_1 \bT_\pl(\w,\gz)} - \ww\w)^2 + o(1), \quad
						\m = \frac{|x|^2}{4}(\sqrt{\vartheta_1 \bA(1)} - \ww)^2 + o(1).
					\end{align*}
					
					Thus, combining these estimates with Proposition \ref{lP} (iii) yields that
					\begin{align*}
				|x|^2 \, \ww^2\, \w\,\cP_\pl(\w,\gz) = 16\pi^2 \, \frac{\vartheta_1^2 }{ - \sin{\vartheta_1}  }\, e^{-\m}  \, \widetilde{\cP}_\pl(\w,\gz) \, (1 + o(1)),\quad\forall \,(\w,\gz) \in \widetilde{\Xi}_1,
					\end{align*}
				where $\widetilde{\cP}_\pl(\w,\gz)$ is defined by \eqref{tild_P}.
					
					Now set
					\begin{align*}
						\widetilde{L}' := 16\pi^2 \, \frac{\vartheta_1^2 }{ - \sin{\vartheta_1}  }\, e^{-\m}  \int_{\widetilde{\Xi}_1}  \widetilde{\cP}_\pl \, d\w d\gz.
					\end{align*}
					Then our estimates above imply $\widetilde{L}' = \widetilde{L}_1 (1 + o(1))=\widetilde{L} \, (1+o(1))$, which completes the proof of Theorem \ref{sL1}.
				\end{proof}

				\section{Uniform asymptotics for the remaining cases: points near the abnormal set} \label{s32}
				\setcounter{equation}{0}
				
				Recall that $\m = |x|^2 \, \bH(1)/4  = (d(g)^2 - |x|^2)/4$.  In this section, our goal is to deal with the following two remaining cases:
				
				\medskip
				(C1). $d(g)^2  \to +\infty$, $|\tz| \le 3$, and $\theta_2 |x| \lesssim 1$.
				
				\medskip
				(C2). $d(g)^2 \to +\infty$, $|\tz|\ge1$ and $\m \lesssim 1$.
				
				\medskip
				However, we shall give asymptotics under more general conditions as follows:
				
				\medskip
				(C3). $d(g)^2 \to +\infty$ and $\m \lesssim 1$.
				
				\medskip
				(C4). $d(g)^2 \to +\infty$, $u_1 \lesssim |x|^{-2}$ and $u_2 \lesssim |x|^{-1}$.
				
				\medskip
				
				More precisely, the relations among them are the following:
				\begin{equation}\label{ab_rel}
					\mbox{(C1) or (C2) $\Rightarrow$ (C3) $\Leftrightarrow$ (C4).}
				\end{equation}
				Indeed, (C2) $\Rightarrow$ (C3) trivially holds and  (C1) $\Rightarrow$ (C3) can be explained by the fact that $\m \sim \tz_2^2|x|^2$ (cf. Lemma \ref{asyin} (i)). To show (C3) $\Rightarrow$ (C4), noting now that $d(g)^2 = |x|^2 + 4\,\m \sim |x|^2\gg1$. Then Lemma \ref{asyin} (i)-(ii) imply that $u_1\sim\tz_1 \, \m \, |x|^{-2}\lesssim|x|^{-2}$ and $u_2\sim\ep \, \chrii \, \chri^{-1} \, \m^{\frac12} \, |x|^{-1}\lesssim|x|^{-1}$, which yield the desired results. Finally, suppose (C4) holds. First we claim that $|u| \ll 1$. Suppose the contrary. Then $1\lsim|u|\lsim|x|^{-2}+|x|^{-1}$, whence $|x|\lsim1$ and $|u|\lsim|x|^{-2}$, which, however, implies that $d(g)^2 \sim |x|^2 |u| \lesssim 1$, a contradiction! Now $\m\lesssim1$ follows at once from \eqref{con6} as $|\tz|\ge1$ and from \eqref{con7} as $|\tz|\le3$, respectively. Hence (C4) $\Rightarrow$ (C3).
				
				To conclude, we reduce the remaining cases to (C4), or equivalently, (C3). Roughly speaking, these points are near the ``(shortest) abnormal set'', i.e., the points $g$ such that the Hessian matrices of the reference functions $\He_\theta \, \phi(g;\theta)$ are degenerate at the critical points $\theta$. Hence in such case, the method of stationary phase no longer works even for small-time heat kernel asymptotics. The interested readers may consult \cite{Li20}. More precisely, we shall prove the following theorem.
			
			\begin{theo}\label{mlsim1}
				Let $\nzz>0$. Then there is a constant $C(\nzz)\gg1$ such that, for all $g$ satisfying $\m\le\nzz$ and $d(g)^2\ge C(\nzz)$,
				\begin{align} \label{hk_ab}
					p(g) =  e^{-\frac{|x|^2}{4}} \, \frac{4\pi}{|x|^2} \,\rF\left(\frac12|x|u_2,\,\frac14|x|^2u_1\right) (1 + o_{\nzz}(1)),
				\end{align}
				where
				\begin{align}\label{defrF}
					\rF(v_1,v_2):= \int_{\R} \frac{r^3}{r \cosh{r} - \sinh{r}} \, e^{ - \frac14\frac{r^2}{r \coth{r} - 1} v_1^2+i v_2 r} \, dr, \qquad v_1, v_2 \in\rr.
				\end{align}
			\end{theo}
			
			Before showing Theorem \ref{mlsim1}, we state the positivity of the smooth function $\rF$ in the following lemma:
			
			\begin{lem}\label{l1}
				It holds that $\rF(v_1,v_2) > 0$, $\forall \,  v_1, v_2\in\rr$.
			\end{lem}
			
			Its proof is essentially the same as in Appendix \ref{AX1} below, but simpler in the 1D case.
			We now proceed with the proof of Theorem \ref{mlsim1}.

			\begin{proof}
				Recall we have shown that $|x| \sim d(g) \gg 1$ in the proof of \eqref{ab_rel}, and the heat kernel is given by (cf. \eqref{ehk2'})
				\begin{align}\label{ehk3}
					p(g) &=  e^{-\frac{|x|^2}{4}} \, \int_{\R^3} \frac{|\lambda|}{\sinh{|\lambda|}}
					\, \exp\left\{-\frac{|x|^2}{4}  \left[\frac{|\lambda| \coth{|\lambda|} - 1}{|\lambda|^2} \, (\lambda_2^2 + \lambda_3^2)  - i u \cdot \lambda\right] \right\} \, d\lambda \\[1mm]
					\nonumber
					& =  e^{-\frac{|x|^2}{4}} \left( \int_{\Lambda_1} + \int_{\Lambda_2} + \int_{\Lambda_3}\right) =: e^{-\frac{|x|^2}{4}} \, (\cI_1 + \cI_2 + \cI_3) =:  e^{-\frac{|x|^2}{4}} \, \cI,
				\end{align}
				where
				\begin{gather*}
					\Lambda_1 := \, \{(\lambda_1,\lambda^\prime) \in \R \times \R^2; \, |\lambda^\prime| \le |x|^{-\frac{3}{4}}, \, |\lambda_1| \le |x|^{\frac{1}{8}}\}, \\
					\Lambda_2 :=  \,\{(\lambda_1,\lambda^\prime) \in \R \times \R^2; \, |\lambda^\prime|  > |x|^{-\frac{3}{4}}, \, |\lambda_1| \le |x|^{\frac{1}{8}}\}, \\
					\Lambda_3 :=  \, \{(\lambda_1,\lambda^\prime) \in \R \times \R^2; \,  |\lambda_1| > |x|^{\frac{1}{8}}\}.
				\end{gather*}
				Notice that the functions
				\begin{gather*}
					\left(\ln\left( \frac{r}{\sinh{r}} \right)\right)^\prime = \frac{1}{r} - \coth{r}  = \frac{\sinh{r} - r \, \cosh{r}}{r \, \sinh{r}}, \\[2mm]
					\left( \frac{r \coth{r} - 1}{r^2} \right)^\prime = \frac{1}{r^2} \left( \coth{r} - \frac{r}{(\sinh{r})^2} \right) -2 \left( \frac{r \coth{r} - 1}{r^3} \right)
				\end{gather*}
				are all bounded on $\R$, and $|\,|\lambda| - |\lambda_1| \,| \le |\lambda^\prime|.$ Then on $\Lambda_1$,  using the finite-increment theorem, we can write
				\begin{align}\nonumber
					\frac{|\lambda|}{\sinh{|\lambda|}}
					e^{-\frac{|x|^2}{4} \frac{|\lambda| \coth{|\lambda|} - 1}{|\lambda|^2} \, |\lambda^\prime|^2}
					\nonumber
					&= \, \frac{\lambda_1}{\sinh{\lambda_1}}
					e^{-\frac{|x|^2}{4} \frac{\lambda_1 \coth{\lambda_1} - 1}{\lambda_1^2} \, |\lambda^\prime|^2} e^{O(|\lambda^\prime|)  + O(|\lambda^\prime|^3 \, |x|^2)} \\\label{Lambda1}
					&= \, \frac{\lambda_1}{\sinh{\lambda_1}}
					e^{-\frac{|x|^2}{4} \frac{\lambda_1 \coth{\lambda_1} - 1}{\lambda_1^2} \, |\lambda^\prime|^2} (1 + O(|x|^{-\frac{1}{4}})) \\[1mm]
					&=:  \mathbf{A}_1+ \mathbf{A}_2 \mbox{\,\,(with the obvious meaning)},\nonumber
				\end{align}
				where we have used the fact that $d(g)^2 \sim |x|^2 \to +\infty$.
				
				Put $\cI_{1,k} = \int_{\Lz_1} \mathbf{A}_k \exp\{i|x|^2 u\cdot\lz/4\}\, d\lz$, $k=1,2$. We then write
					$$\begin{aligned}
						\cI_{1,1} &= \int_{\Lz_1} \frac{\lambda_1}{\sinh{\lambda_1}}
						e^{-\frac{|x|^2}{4} \frac{\lambda_1 \coth{\lambda_1} - 1}{\lambda_1^2} \, |\lambda^\prime|^2} e^{\frac{|x|^2}{4}i u \cdot \lambda} \, d\lambda \\[1mm]
						&=\int_{\R^3}\cdots-\int_{\Lambda_2}\cdots- \int_{\Lambda_3}\cdots=:\cI_{1,1,1} - \cI_{1,1,2} - \cI_{1,1,3}.
					\end{aligned}$$
					Next, we shall illustrate that the term $\cI_{1,1,1}$ is principal and the other terms $\cI_{1,2}, \, \cI_{1,1,2},\, \cI_{1,1,3}, \, \cI_{2},\, \cI_{3}$ are negligible.
					
					{\bf Estimate of $\cI_{1,1,1}$.}  Applying \eqref{Hor90} with $q=2$, $A=\frac{|x|^2}{2}\frac{\lz_1\coth\lz_1-1}{\lz_1^2} \, \I_2$ and $Y=\frac{|x|^2}{4}(u_2,0)$ respectively w.r.t. the variable $\lz'$, we see that
					\begin{align*}
						\cI_{1,1,1} = \frac{4\pi}{|x|^2} \, \rF\left(\frac12|x|u_2,\,\frac14|x|^2u_1\right)
					\end{align*}
					from the definition of $\rF$ in \eqref{defrF}.
					
					By Lemma \ref{l1}, $\rF(v_1,v_2)\sim1$ whenever $v_1$ and $v_2$ are both bounded. Then $\cI_{1,1,1} \sim_{\zeta_0}|x|^{-2}$ under our assumptions. It is therefore sufficient to show
						\begin{equation}\label{remain}
							|\cI_{1,2}| + |\cI_{1,1,2}| +|\cI_{1,1,3}| +|\cI_{2}|+ |\cI_{3}|=o(|x|^{-2}).
						\end{equation}
						
						{\bf Bound of $\cI_{1,2}$.} Applying \eqref{Lambda1} and the exponential decay of $r/\sinh r$, we obtain
							$$
							|\cI_{1,2}| \lsim |x|^{-\frac14} \,\int_{\Lz_1}\frac{\lambda_1}{\sinh{\lambda_1}}
							e^{-\frac{|x|^2}{4} \frac{\lambda_1 \coth{\lambda_1} - 1}{\lambda_1^2} \, |\lambda^\prime|^2} d\lz' \, d\lz_1 \lsim |x|^{-\frac94}.
							$$

						{\bf Evaluation of $\cI_{2}$ and $\cI_{1,1,2}$.} On $\Lambda_2$, first observe that $|\lambda^\prime|^2 \, |x|^2 \ge |x|^{- \frac{3}{2}} \, |x|^2 = |x|^{\frac{1}{2}}$. Moreover, the simple estimate
						\begin{equation}\label{e1}
							\frac{r \coth{r} - 1}{r^2} \sim \frac{1}{1 + |r|}, \quad \forall \, r \in\mathbb{R}
						\end{equation}
						implies that
						\begin{align*}
							\frac{\lambda_1 \coth{\lambda_1} - 1}{\lambda_1^2} \sim \frac{1}{1 + |\lambda_1|} \ge \frac{1}{1 + |x|^{\frac{1}{8}}}
							\sim \frac{1}{|x|^{\frac{1}{8}}}.
						\end{align*}
						Then
						\[
						\frac{|x|^2}{4} |\lambda'|^2 \,\frac{\lambda_1 \coth{\lambda_1} - 1}{\lambda_1^2}  \gtrsim |x|^{\frac{3}{8}} + |x|^{\frac{15}{8}} |\lambda'|^2.
						\]
						From this and the trivial estimate $r/\sinh{r} \lesssim 1$, we obtain
						$$ |\cI_{1,1,2}| \le e^{-c \, |x|^{\frac{3}{8}}}  \int_{\Lambda_2} \exp\{- |x|^{\frac{15}{8}} |\lambda'|^2\} \, d\lambda_1 d\lambda'  \lesssim e^{-c \, |x|^{\frac{3}{8}}}. $$
						
						To estimate $\cI_2$,  using \eqref{e1} again for $r=|\lambda|$, we see
						\begin{align}\label{I2est}
							\frac{|\lambda| \coth{|\lambda|} - 1}{|\lambda|^2} \, |\lambda^\prime|^2 \, |x|^2 \gtrsim |x|^{\frac{3}{8}}, \qquad \forall \, \lambda \in \Lambda_2,
						\end{align}
						which implies that
						$$
						|\cI_2| \lesssim e^{-c \, |x|^{\frac{3}{8}}} \int_{\Lambda_2} \frac{|\lambda|}{\sinh|\lambda|} \, d\lambda
						\le  e^{-c \, |x|^{\frac{3}{8}}} \int_{\mathbb{R}^3} \frac{|\lambda|}{\sinh|\lambda|} \, d\lambda\lesssim e^{-c \, |x|^{\frac{3}{8}}}.
						$$
						
						{\bf Estimations of $\cI_{3}$ and $\cI_{1,1,3}$.} It follows from the monotonicity and the exponential decay of $\frac{s}{\sinh{s}}$ on $[0,+\infty)$ that
						\begin{align*}
							\frac{\lambda_1}{\sinh{\lambda_1}} = \left( \sqrt{\frac{\lambda_1}{\sinh{\lambda_1}}} \right)^2 \le \sqrt{\frac{|x|^{\frac{1}{8}}}{\sinh{|x|^{\frac{1}{8}}}}}  \sqrt{\frac{\lambda_1}{\sinh{\lambda_1}}} \lesssim \sqrt{\frac{\lambda_1}{\sinh{\lambda_1}}} \, e^{ - \frac{|x|^\frac18}{4}}, \quad \forall (\lambda_1,\lambda^\prime) \in \Lambda_3.
						\end{align*}
						Inserting this estimate into $\cI_{1,1,3}$, we have
						$$
						|\cI_{1,1,3}| \lesssim  e^{ - \frac{|x|^{\frac{1}{8}}}{4}} \int_{\Lambda_3} \sqrt{\frac{\lambda_1}{\sinh{\lambda_1}}}
						\, \exp\left\{-\frac{1}{4} \frac{\lambda_1 \coth{\lambda_1} - 1}{\lambda_1^2} \, |\lambda^\prime|^2 \, |x|^2\right\}
						\, d\lambda_1 \, d\lambda'  \\
						\lesssim  e^{ - \frac{|x|^{\frac{1}{8}}}{4}}|x|^{-2}.
						$$
						Similarly,
						\[
						|\cI_3| \lesssim e^{ - \frac{|x|^{\frac{1}{8}}}{4}} \int_{\Lambda_3} \sqrt{\frac{|\lambda|}{\sinh{|\lambda|}}} d\lambda  \lesssim e^{ - \frac{|x|^{\frac{1}{8}}}{4}}.
						\]
						
						From these estimates we conclude the proof of \eqref{remain} and hence the theorem.
					\end{proof}
					
					\begin{remark} \label{rethm101}
						The argument above also works for $u_1 = u_2 = 0$.
					\end{remark}
					
					Notice that the above proof and the fact that $d(g)^2= |x|^2 + 4\,\m$ also give the following
					
					\begin{cor} \label{nABnc1}
						Let $\nzz>0$. Then there is a constant $C(\nzz)\gg1$ such that
						\begin{equation}\label{pbnd3}
							p(g) \, \sim_{\zeta_0} \, |x|^{-2} \, e^{-\frac{d(g)^2}{4}}, \quad  \mbox{for $d(g)^2\ge C(\nzz)$ with  $\m\le\nzz$.}
						\end{equation}
					\end{cor}
					
					\section{Summaries of main results} \label{sec11}
					Recall that in previous sections we have established the uniform asymptotic expansions of the heat kernel at infinity in four possible cases, namely Theorems \ref{asyab}, \ref{bigL1}, \ref{sL1} and \ref{mlsim1} with Remarks \ref{Vara1} and \ref{Vara2}. Now we summarize these results and deduce the precise estimates as well as the small-time asymptotics of the heat kernel.
					
					\subsection{Uniform asymptotics}
Recall \eqref{mdef}-\eqref{chr2def}, \eqref{qdef}, \eqref{tild_P} and \eqref{defrF} for pertinent definitions. We have:
					\begin{theo}\label{Mtheo}
						Under the Assumption (A) (cf. \eqref{n78n}), it holds that:
						{\em\begin{compactenum}[(i)]
								\item Let  $\az_0\in [3,\pi) $. Then there exists a constant $C(\az_0) \gg 1$ such that for all $g$ satisfying $|\tz| \le \az_0$ with $\tz_2|x| \ge C(\az_0)$,
								$$p(g) =  (8\pi)^{\frac{3}{2}} \, e^{-\frac{d(g)^2}{4}} \, \frac{|\theta|}{\sin{|\theta|}} \, \frac{1}{\sqrt{\det(- \He_\theta \phi(g;\theta))}} \,  (1 + o_{\az_0}(1)).$$
								
								\item Let $\beta_0 \in (0,1]$. Then there exists a constant $C(\beta_0) \gg 1$ such that for all $g$ satisfying $|\tz| \ge \beta_0$, $\chri, \m \ge C( \beta_0)$,
								$$
								p(g) = 16 \pi^2 \sqrt{\pi\vtz_1} \,  \q(|\tz|) \, \frac{e^{-\chrii} \, I_0(\chrii)}{\sqrt{\ep \, \chri}} \, e^{ -\frac{\vartheta_1|x|^2 \, \ww^2}{2\ep}} \, I_0\left(\frac{\vartheta_1|x|^2 \, \ww^2}{2\ep}\right) \,  e^{-\frac{d(g)^2}{4}} \, (1+o_{\beta_0}(1)).
								$$
								
								\item Let $\zeta_0 > 0$.  Then there is a constant $C(\zeta_0) \gg 1$ such that, for all $g$ satisfying  $\chri \le \zeta_0$ and $\m \ge C(\zeta_0)$,
								$$
								p(g) = 4\pi \, \frac{\vartheta_1^2 }{ - \sin{\vartheta_1}  }\, e^{-\frac{d(g)^2}{4}}\int_0^{\w_0}\int_{-\pi}^\pi  \widetilde{\cP}_\pl(\w,\gz) \, d\w  \, d\gz\,(1 + o_{\nzz}(1)).
								$$
								\item Let $\nzz>0$. Then there is a constant $C(\nzz)\gg1$ such that, for all $g$ satisfying $\m\le\nzz$ and $d(g)^2\ge C(\nzz)$,
								$$	p(g) =  e^{-\frac{d(g)^2}{4}}\, e^{\m} \, \frac{4\pi}{|x|^2} \, \rF\left(\frac12|x|u_2,\,\frac14|x|^2u_1\right) (1 + o_{\nzz}(1)).
								$$
						\end{compactenum}}
					\end{theo}
					
						\begin{remark}
							Supposing the requirements of Theorem \ref{Mtheo} (i) and (ii) are fulfilled simultaneously, then the leading terms of the asymptotic formulas for the heat kernel given by them indeed coincide. This fact can be verified via an elementary but somewhat tedious calculation, by means of \eqref{asinI0} (noting now that $\chrii\gg 1$ and $\frac{\vartheta_1|x|^2 \, \ww^2}{2\ep} \gg1$ by Lemma \ref{asyin} and Remark \ref{rem67}), with equalities \eqref{chr1def}, \eqref{chr2def},
							\eqref{qdef}, \eqref{32nK},
							\eqref{detph} and $\ww =\tz_2 \, \psi(|\tz|)$ (cf. \eqref{n78n}).
						
						As a by-product, we get the following important equality under Assumption (A):
						\begin{align} \label{Ede1}
							(8\pi)^{\frac{3}{2}} \, \frac{|\theta|}{\theta_2 \sin{|\theta|}} \, \frac{|x|^{-3}}{\sqrt{\K(\tz_1,\tz_2) \,  \K_3(\tz_1,\tz_2)}} = 8 \pi \sqrt{\pi} \,  \q(|\tz|) \,  \left(\frac{|x|^2 \, \ww^2}{2} \chri \, \chrii \right)^{-\frac{1}{2}}.
						\end{align}

					\end{remark}
					
				An application of Theorem \ref{Mtheo} leads to the following consequence:
				
				\subsection{Sharp upper and lower bounds}

				The statement of this result has already appeared in Section \ref{sec24}. We restate it here for the reader's convenience.
				\begin{theo}
					Under Assumption (A) (cf. \eqref{n78n}), we have
					\begin{equation} \label{pbound2}
						\begin{aligned}
							p(x,t) \sim (1+d(g))^{-2}\frac{1+\ep  \, d(g)}{1+\ep \,  d(g) + \ep  \, t_2^{\frac12} \, |x|^{\frac12} \,  (d(g)^2-|x|^2)^{\frac14}} \, e^{-\frac{d(g)^2}{4}}.
						\end{aligned}
					\end{equation}
				\end{theo}

				\begin{proof}
					 Recall that $d(g)^2= |x|^2 + 4\,\m$  and $\m = \bH(1) \, |x|^2/4$. Then from Lemma \ref{asyin} (ii), the right hand side of \eqref{pbound2} is bounded below and above by the quantity
						\begin{align} \label{checkRHS}
							{\rm BND} := \frac{1}{1 + d(g)^2 } \, (1 + \chri)^{\frac{1}{2}} \, (1 + \chri + \epsilon \, \m \, \chrii )^{-\frac{1}{2}} \, e^{-\frac{d(g)^2}{4}}.
						\end{align}
						Thus the estimate \eqref{pbound2} just amounts to $p\sim {\rm BND}$.
						
						For the case that $d(g)^2\lesssim1$, using Lemma \ref{asyin} (ii) again we obtain $d(g)+\ep  \, d(g) + \ep \, \m + \chrii\lesssim1$, whence ${\rm BND}\sim1$, while by the positivity of the heat kernel we have $p\sim1$. Thus, $p\sim{\rm BND}$.
						
						The opposite case $d(g)^2\gg1$ can be deduced directly from Theorem \ref{Mtheo}, according to which we split it into four subcases:
						
						If $\atz\le3$ (so $\ep\sim1$) and $\theta_2|x|\to\infty$, then $|x|\gg1$, and $\m\sim|\tz_2|^2|x|^2\gg1$ by Lemma \ref{asyin} (i), which, together with Theorem \ref{Mtheo} (i) and \eqref{Hetau11} show that
						$$p\sim |x|^{-2} (\,\tz_2|x|\,)^{-1} e^{-\frac{d(g)^2}{4}} \sim |x|^{-2}\m^{-\frac12}\,e^{-\frac{d(g)^2}{4}}.
						$$
						Since $d(g)^2 \sim |x|^2 (1+\tz_2^2) \sim |x|^2$, then $\chri\sim |x|^2$ by Lemma \ref{asyin} (ii). From this and Lemma \ref{asyin} (v) we conclude that ${\rm BND}\sim d(g)^{-1}|x|^{-1}\m^{-\frac12}e^{-\frac{d(g)^2}{4}} \sim p$.

						If $|\tz|\ge1$, $\m\to\infty$ and $\chri\to\infty$, then the desired result follows at once from \eqref{pbnd1} and the fact that $\chri \sim \ep^2 d(g)^2 \sim \epsilon \, \m + \chrii$ (cf. Lemma \ref{asyin} (ii)).
						
						If $\m\to\infty$ and $\chri\lesssim1$, then from the proof of Theorem \ref{sL1*} we see that $d(g)\sim|x|$. Thus, by Lemma \ref{asyin} (ii) again we have $\ep \,  d(g) + \ep \, \m + \chrii \lesssim1$, whence ${\rm BND}  \sim|x|^{-2}e^{-\frac{d(g)^2}{4}}$. Consequently, the estimate \eqref{psL1} yields that $p \sim{\rm BND}$.
						
						If $\m\lesssim1$, then by \eqref{pbnd3} we have $p\sim|x|^{-2}e^{-\frac{d(g)^2}{4}}$. And from the relations \eqref{ab_rel} we see that $d(g)\sim|x|\gg1$, $u_1\lesssim|x|^{-2}$ and $u_2\lesssim|x|^{-1}$. Hence ${\rm BND}\sim|x|^{-2} e^{-\frac{d(g)^2}{4}} \sim p$, since $\ep\,\m\,\chrii \, \lsim\ep \, \chrii\lsim\chri$. This completes the proof of the theorem.
					\end{proof}
					
					Thanks to this theorem, we can derive the precise estimates for the heat kernel for all points $g=(x,t)\in N_{3,2}$ at infinity. Notice that the properties \eqref{prpty_p} allow us to consider only specific points as in the following:
					\begin{cor} \label{simHUD}
						Assuming that $t=(t_1,t_2,0)$ with $t_1,t_2\ge0$ and $x=|x|e_1$, we have:
						\medskip
						{\em\begin{compactenum}[(i)]
								\item If $|x|^2 \lsim |t|$, then
								\begin{align} \label{sim21}
									p(x,t) \sim (1+t_2 \, |x| \, |t|^{-\frac12})^{-\frac12} \, |t|^{-1} \, e^{-\frac{d(g)^2}{4}}, \qquad \mbox{as\,\,} |t| \to +\infty.
								\end{align}
								
								\medskip\item If $|x|^2 \gg |t|$, then
								\begin{equation} \label{sim22}
									p(x,t) \sim (1+ |x|^{-1} t_2 + |x|^{-\frac12}\, t_1^{\frac14}\, t_2^{\frac12}    )^{-1} |x|^{-2}\, e^{-\frac{d(g)^2}{4}}, \qquad \mbox{as\,\,} |x| \to +\infty.
								\end{equation}
						\end{compactenum}}
					\end{cor}
					
					\begin{proof} Via a limit argument, it is enough to prove these estimates under Assumption (A) (cf. \eqref{n78n}) only. Recall that $d(g)^2- |x|^2 = 4\m =|x|^2\,\bH(1)$, and $d(g)^2 \sim |x|^2 + |t|$.
						
						For item (i),  it suffices to use \eqref{pbound2} and the facts that
						$$ \ep\sim1, \quad \m \sim |t| \sim d(g)^2 \gg1, $$
						deduced from Lemma \ref{asyin} (iii).
						
						For item (ii), using Lemma \ref{asyin} (iv)-(v) we have $\m \sim t_1 (1+ u_2^2\, u_1^{-1})$. In fact, when $|\tz|\ge1$, this follows from the facts that $\m\sim t_1$ and $u_2^2\, u_1^{-1}\lsim \ep^2 \lsim1$ (by \eqref{con6}); while when $|\tz|\le3$, by \eqref{con7} it holds that $\m\sim |x|^{2} u_2^2$ and $u_2^2 \, u_1^{-1}\sim \tz_1^{-1}\gsim1$, yielding the desired estimate as well. Consequently, appealing to \eqref{pbound2} again gives
						\begin{equation} \label{sim23}
							p(x,t) \sim \frac{1 + \epsilon |x|}{1 +\epsilon\, |x| + \epsilon\, t_2 + \epsilon\, |x|^\frac12 \,  t_1^\frac14 \,  t_2^\frac12}\, |x|^{-2}\, e^{-\frac{d(g)^2}{4}}.
						\end{equation}
						If $\ep\,|x|\le1$, then $\ep\ll1$. So together with the first estimate in \eqref{con6}, we get that: $t_1 = |x| u_1 |x|/4 \lsim |x| (\epsilon |x|) \lsim |x|$ and $t_2 \lsim \sqrt{t_1}$. As a result, we obtain $p(x,t)\sim |x|^{-2}\exp\{-d(g)^2/4\}\sim {\rm RHS \,\,of\,\, \eqref{sim22}}$. Conversely, if $\ep\,|x|\ge1$, then $1+\ep\,|x|\sim \ep\,|x|$. From this and \eqref{sim23} the estimate \eqref{sim22} follows immediately.
					\end{proof}
					
					\begin{remark}
						One can replace $|t|$ by $d(g)^2$ in \eqref{sim21}, and $|x|$ by $d(g)$ in \eqref{sim22} respectively.
					\end{remark}
					
					Another consequence of Theorem \ref{Mtheo} is the
					\subsection{Small-time asymptotics}
				To obtain the desired asymptotics, the uniform asymptotics of Theorem \ref{Mtheo} (ii) and Remark \ref{rethm101} are sufficient.  Since $g=o$ is trivial, we may assume that $g\neq o$.  Recall that one can reduce the matters to the case  where $x=|x|e_1$ and $t=(t_1,t_2,0)$ with $t_1,t_2\geq0$.
				
				The first corollary is due to Remark \ref{rethm101}.
			
			\begin{cor} Let $x \neq 0$ and $t=0$. Then
				\begin{equation*}
					p_h(x,t)=\cc \, 4\pi \, \rF(0,0) \,   h^{-\frac72} \, |x|^{-2} \, e^{-\frac{d(g)^2}{4h}} \, (1+o_g(1)),\quad{\rm as}\,\, h\to0^+.
				\end{equation*}
			\end{cor}
			
			Applying Theorem \ref{Mtheo} (ii) and \eqref{Ede1}, one obtains immediately the following Corollary \ref{smhk}, 	which shows in particular that \eqref{rteigen} is still valid for  $g_u \in \R^2_{<, +}$.
			\begin{cor}\label{smhk}
				Let $x\neq0$, $t_1,t_2>0$ and $\pi t_2^2\neq|x|^2 t_1$. Then as $h\to0^+$,
				\begin{equation*}
					\begin{aligned}
						p_h(x,t)  = \cc\, (8 \pi)^\frac32 \, h^{-3} \, \frac{|\tz|}{\sin |\tz|} \, \tz_2^{-1} [\K(\tz_1,\tz_2) \,  \K_3(\tz_1,\tz_2)]^{-\frac12} \, |x|^{-3} \, e^{-\frac{d(g)^2}{4h}} \, (1+o_g(1)),
					\end{aligned} 	
				\end{equation*}
				where $\tz=(\tz_1, \tz_2,0):=(\Lambda^{-1}\left(4\,|x|^{-2}\, t_1, 4\,|x|^{-2} \, t_2\right),0)$.
				
			\end{cor}

To deduce the following four corollaries, we can appeal to Theorem \ref{Mtheo} (ii) with an argument of limit. Indeed, by \eqref{chr1def},  Corollary \ref{RLT2} and the fact that $4\m=d(g)^2-|x|^2$, one can check that the approximating sequence (whose choice is similar as in the proof of Corollary \ref{RLT2}) of the given point $(\frac{x}{\sqrt{h}}, \frac{t}{h})$ satisfies the asymptotic condition uniformly provided $h$ is small enough, then a passage to limit in both sides of the asymptotic formula will give us the asserted results.
			
			\begin{cor}
				Let $x\neq0$, $t_1=0$ and $t_2>0$. Then
				\begin{equation*}
					\begin{aligned}
						p_h(x,t) &= \cc\, 8 \pi^\frac32 \, h^{-3} \, \Upsilon(r)^\frac32\, [-\Upsilon'(r)]^{-\frac12}\, t_2^{-\frac12} |x|^{-2}  e^{-\frac{d(g)^2}{4h}}\, (1+o_g(1)),\quad{\rm as}\,\, h\to0^+,
					\end{aligned}
				\end{equation*}
				where $r$ is the unique solution of $\mu(r)=4\,|x|^{-2}\, t_2$.
			\end{cor}
			
			\begin{cor} Let $x\neq0$, $t_1,t_2>0$ and $\pi t_2^2=|x|^2 t_1$. Then
				\begin{equation*}
					p_h(x,t)=\cc\, 2 \,\sqrt{2\pi} \, \pi^3 \, h^{-3}  \, t_2^{-1} \, |x| \, (|x|^4 + \pi^2 t_2^2)^{-\frac12}  \, e^{-\frac{d(g)^2}{4h}} \, (1+o_g(1)),\quad{\rm as}\,\, h\to0^+.
				\end{equation*}
			\end{cor}
			
			\begin{cor}\label{cutL}
				Let $x\neq0$, $t_1>0$ and $t_2=0$. Then as $h\to0^+$ it holds that
				\begin{equation*}
					\begin{aligned}
						p_h(x,t) =\cc\, 16\pi^2 \, h^{-\frac72}\left[\frac{-\Upsilon'(r) \, r^4}{\Upsilon''(r)\psi(r) \, [2r-\sin(2r)]+4r^2 \, \psi'(r)}\right]^\frac12  |x|^{-2 } \, e^{-\frac{d(g)^2}{4h}} \, (1+o_g(1)),
					\end{aligned}
				\end{equation*}
				where $r$ is the unique solution of the equation \eqref{DCUTP} with $\beta=4\,|x|^{-2}\, t_1$.
			\end{cor}
			
			\begin{cor} Let $x=0$ and $t\neq0$. Then
				\begin{equation*}
					p_h(x,t)=\cc\, 4\pi^4 \, h^{-\frac72} \, |t|^{-1}e^{-\frac{d(g)^2}{4h}} \, (1+o_g(1)),\quad{\rm as}\,\, h\to0^+.
				\end{equation*}
			\end{cor}
			
\section{Sharp bounds for derivatives of the heat kernel} \label{sdbpd}
 Recall that we use $g = (x,t)$ to denote an element in $N_{3,2}$ and $\nabla=(\X_1,\,\X_2,\,\X_3)$ to denote the horizontal gradient on this group, with $\X_i$ given by \eqref{LVF}. Let $\widehat{\X}_i$ ($1 \le i \le3$) represent the corresponding right-invariant vector field, that is,
\begin{equation*}
	\begin{gathered}
		\widehat{\X}_1 = \frac{\partial}{\partial x_1} + \frac12 x_3 \frac{\partial}{\partial t_2} -\frac12 x_2 \frac{\partial}{\partial t_3}, \quad
		\widehat{\X}_2 = \frac{\partial}{\partial x_2} - \frac12 x_3 \frac{\partial}{\partial t_1} +\frac12 x_1 \frac{\partial}{\partial t_3}, \\
		\widehat{\X}_3 = \frac{\partial}{\partial x_3} +\frac12 x_2 \frac{\partial}{\partial t_1} -\frac12 x_1 \frac{\partial}{\partial t_2}.
	\end{gathered}
\end{equation*}

    Now set $\widehat{\nabla} = (\widehat{\X}_1, \widehat{\X}_2, \widehat{\X}_3)$ and $\nabla^c = (\nabla, \widehat{\nabla})$. Then we have the following theorem.  We mention that  B. Qian has established $|\nabla \ln{p(g)}| \lsim d(g)$ for any free step-two Carnot group with $k$ generators (cf. \cite[Proposition 5.5]{Q13}), and its  proof is based on the Harnack inequality and Bakry--\'Emery criterion (cf. \cite{BB09}  for more details on this approach).

			\begin{theo} \label{hkDb}
				It holds for any $g \in N_{3, 2}$ that
				\begin{equation*}
					|\nabla^c\, p(g)|\lsim d(g)\,p(g), \qquad |(\nabla^c)^l\, p(g)| \lesssim_l (1 + d(g))^l \, p(g) \quad \mbox{for $l = 2, 3, \ldots$.}
				\end{equation*}
			\end{theo}			
			
			\begin{proof}
	Recalling also that $d(g)^2 \sim |x|^2 + |t|$, we start with the reduction of the problem. When $d(g)$ is bounded, since $p$  is smooth and positive, we see the second estimate is trivial, and the first one follows at once from the the exponential decay of $\mathbf{V}(\lz)$ as well as the facts that $|x|\lsim d(g)$ and $ 0< r^{-2}\,(r\coth r-1)\lsim1 $ for all $r>0$. So only the case where $d(g)$ is large  needs handling with greater care,  then it is enough to show the following
				\begin{equation}\label{target}
					|\partial^\alpha_x \partial^\beta_t p(g)| \lesssim_{\alpha, \beta} d(g)^{|\alpha|} \, p(g), \qquad \forall \, g \in N_{3,2}, \, \alpha, \beta \in \N^3.
				\end{equation}
				
				By an argument of limit and the second equation of \eqref{prpty_p}, we only need to prove \eqref{target} under the following assumption:
					\begin{align} \label{ncg}
					x \ne 0, \qquad   x \cdot t > 0, \qquad  t - \frac{t \cdot x}{|x|} \frac{x}{|x|} \ne 0, \qquad \pi \left|  t - \frac{t \cdot x}{|x|} \frac{x}{|x|}\right|^2 \ne (x \cdot t) |x|.
					\end{align}
					Note that this assumption is equivalent to that we can find a suitable  orthogonal matrix $O_g$ such that  $\ggg = (O_g \, x, O_g \, t)$ satisfies Assumption (A) (cf. \eqref{n78n}).  Notice that there exists only one such $\ggg$, we have also $O_g \, x = |x| \, e_1$, $d(g) = d(\ggg)$, and $\m(\ggg) = (d(g)^2 - |x|^2)/4 = \m(g)$. In principle, the methods in establishing the asymptotics of the heat kernel should be enough to estimate the sharp bounds of its derivatives. As before, we want to reduce the point to Assumption (A) naturally. However, it turns out that the orthogonal matrix $O_g$ may also depend on $g$ and we should be very careful. In fact, we prove \eqref{target} in the following four cases (conditions are stated for the point  $\ggg$):
				
				\medskip
				{\bf Case (1)}:  $|\tz(\ggg)| \le 3$ and $\tz_2(\ggg) |x|\to+\infty$.
				
				\medskip
				{\bf Case (2)}: $\m(\ggg) \lsim 1$  and $d(g) \to+\infty$.
				
				\medskip
				{\bf Case (3)}: $|x|\, d(g)\le 1$ and $d(g) \to+\infty$.
				
				\medskip
				{\bf Case (4)}: $|x|\, d(g)\ge 1$,  $|\tz(\ggg)| \ge1$ and $\m(\ggg) \to+\infty$.
				
				\medskip
					
					For {\bf Cases (1)-(3)}, it transpires that one can first directly taking derivatives in the original expression of the heat kernel (cf. \eqref{ehk2}), and  can reduce the proof of \eqref{target} to the one of \eqref{bdgtar} or \eqref{bdgtar2} below. Then performing the orthogonal transform as before solves the problem. However, we should pay attention to  {\bf Case (3)} since there emerges singularity of higher order
in the amplitude owing to derivation.  Unlike the known cases (Heisenberg groups, non-isotropic Heisenberg groups, and GM-M\'etivier groups), in {\bf Case (3)} we neither have the estimates for higher order singularities in the amplitude on $N_{3,2}$, nor can apply the technique in {\bf Case (4)} to deal with it. So we have to establish the upper bound estimates for higher order singularities. Fortunately, the method to obtain the uniform heat kernel can be modified to overcome this difficulty in an elegant manner, though the integral in this case is converted to a higher dimensional one.
					
					For the remaining {\bf Case (4)}, we first choose an orthogonal matrix $O=O_g$ in \eqref{prpty_p} to enable us to use the new formula \eqref{ehk4}. However, in this way we should be careful since we need also  take derivatives on such orthogonal matrix and thus we require a lower bound of $|x|$ (namely $d(g)^{-1}$, which indeed can be observed from \eqref{target}) to treat it.
 	 		
				\subsection{Proof of {\eqref{target}} in {\bf Cases (1)}-{\bf(3)}} \label{der_rot}
				
				In both {\bf Cases (1)-(2)},  by the fact that $|x| \lesssim d(g)$, it suffices to establish the following estimate
				\begin{align} \label{bdgtar}
					|p^{n_0,\alpha_0} (x,t)| \lesssim_{n_0,\alpha_0} p(x,t), \qquad  \forall\, n_0 \in \N, \, \alpha_0 \in \N^3
				\end{align}
				where  $g = (x, t)$ as in \eqref{ncg}, and
				\begin{align} \label{defpna}
					p^{n_0,\alpha_0} (x,t) := \int_{\R^3} \cV(\lambda) \left( \frac{|\lambda| \coth{ |\lambda| } - 1 }{|\lambda|^2}\right)^{n_0} \lambda^{\alpha_0} \, e^{-\frac{1}{4} \widetilde{\phi}((x,t); \lambda)} \, d\lambda.
				\end{align}
				Notice that from the definition above we have   $|p^{n_0,\alpha_0} (g)| \lesssim_{\alpha_0}  \sum_{|\alpha| = |\alpha_0|} |p^{n_0,\alpha}(\ggg)|$ by means of an orthogonal transform. Combining this with the fact that $p(g) = p(\ggg)$. It remains to prove \eqref{bdgtar} for $g$ as in Assumption (A) (cf. \eqref{n78n}). Note that we are only concerned with the upper bound, the desired estimate can be deduced easily by using an argument
similar to (and simpler than) that of Theorem \ref{asyab} for  {\bf Case (1)}, and that of Theorem \ref{mlsim1} for  {\bf Case (2)}.

				 For {\bf Case (3)}, some modification is now required. Firstly, by the assumption that $|x|\le d(g)^{-1}$, it remains to prove
				\begin{align} \label{bdgtar2}
					|p^{n_0,\alpha_0} (x,t)| \lesssim_{n_0,\alpha_0} d(x,t)^{2n_0} \, p(x,t), \qquad \forall\, n_0 \in \N, \, \alpha_0 \in \N^3.
				\end{align}
				As before we can reduce the proof to the points satisfying \eqref{n78n} or equivalently Assumption (A). Secondly, it follows from $|x|\le d(g)^{-1}$ that $d(g)\sim \sqrt{|t|} \gg1$ and $|u|\gg1$, which occurs only in the case of Lemma \ref{asyin} (iii). As a result, we obtain
					\begin{equation} \label{obes1}
						\ww \sim \sqrt{|t|}/|x| \sim \sqrt{|u|}, \quad |x|^2 \, \ww^2 \sim \m \sim d(g)^2, \quad  |\pi-|\tz| \, |\lsim |u|^{-\frac12} \ll1.
					\end{equation}
					Moreover, from Corollary \ref{simHUD} (i) it follows that
				\begin{equation*}
					p(x,t) \sim |t|^{-1} e^{-\frac{d(g)^2}{4}} \sim d(g)^{-2} \, e^{-\frac{d(g)^2}{4}}.
				\end{equation*}
				
				Thirdly,  we need a more suitable integral expression for $p^{n_0, \alpha_0}$, which will play an important role. To be more precise, applying  \eqref{Hor90} with $q=2 + 2n_0$, $Y = \frac{|x|^2 \ww}{2} (\lz_2, \lz_3, 0 , \ldots, 0) \in \R \times \R \times \R^{2n_0}$ and $A = \frac{1}{2} \frac{|\lz|^2}{|\lz|\coth|\lz|-1} \, |x|^2 \, \ww^2 \, \I_{2 + 2n_0}$,  respectively, we find that
				\[
				p^{n_0,\alpha_0} (x,t) = \frac{(|x|^2 \, \ww^2)^{1 + n_0}}{(4\pi)^{1 + n_0}}  e^{-\frac{|x|^2}{4}} \, \int_{\R^{2 + 2n_0}}  \rP^{n_0,\alpha_0} \left( s|x|\,\ww,\frac14|x|^2(u + 2\ww s_1 e_2 + 2\ww s_2  e_3)\right) \,  ds
				\]
				where
				\[
				\rP^{n_0,\alpha_0} (X,T) := \int_{\R^3} \vv(\lambda) \, \lambda^{\alpha_0} \, e^{-\frac{1}{4} \Ga((X,T);\lambda)} \, d\lambda, \qquad (X,T) \in \R^{2 + 2n_0} \times \R^3.
				\]
				A similar and simpler argument used in Appendix \ref{secA} yields the following upper bound for $\rP^{n_0,\alpha_0}$:
				\begin{equation} \label{PnaBD}
					|\rP^{n_0,\alpha_0}(X,T)| \lesssim_{\alpha_0} \frac{\exp\left\{-\frac{\rD(X,T)^2}{4}\right\}}{(1 + \rD(X,T))^2 (1 + |X| \rD(X,T))^{\frac{1}{2}}},
				\end{equation}
				where $\rD$ is  now is defined by \eqref{eD3} with $X \in \R^{2 + 2n_0}$.
				
				Now adopting similar notation as in Subsection \ref{afH}, we write
				\begin{gather*}
					\cD_{n_0}(s):=\rD\left(s\,\ww,\frac14(u + 2\ww s_1 e_2 + 2\ww s_2  e_3)\right)^2, \quad s=(s_1,s_2,s') \in \R \times \R \times \rr^{2n_0}, \\
					\bH_0(\w):=\ww^2 |\w|^2\,\Phi(\u_0(\w)) \,\,{\rm with}\,\, \u_0(\w):=\frac{\bA(\w_1)}{\ww^2 |\w|^2}, \quad \w=(\w_1,\w_2)\in(0,\infty)\times(0,\infty).
				\end{gather*}
				Then $\cD_{n_0}(s)\ge \bH_0(|(s_1,s_2)|,|s'|)$.  From the above estimates and using polar coordinates \eqref{pol_cor} for $(s_1,s_2)$ and $2n_0$-dimensional ones for $s'$ respectively, together with the fact that $d(g)^2 = |x|^2 + 4\m$, we deduce that \eqref{bdgtar2} will hold provided that
				\begin{equation}\label{bdg3}
					\int_0^\infty \int_0^\infty \frac{\exp\left\{-\frac{|x|^2}{4}\bH_0(\w) \right\} \, \w_1\w_2^{2n_0-1}}{\left(1 + |x|^2 \bH_0(\w) \right) \sqrt{1 + |x|^2 \,\ww |\w| \sqrt{\bH_0(\w)}}} \, d\w_1 d\w_2 \lsim \m^{-2} \, e^{-\m}.
				\end{equation}
				
				To prove this we write LHS of \eqref{bdg3} as $\int_{\mathcal{O}_1} + \int_{\mathcal{O}_2} + \int_{\mathcal{O}_3}$ with
				\begin{equation*}
					\mathcal{O}_1:=\left\{\w; |\w|\le \dz_0 \right\},\quad \mathcal{O}_2:=\left\{ \w; \dz_0<|\w| < C_0 \right\}, \quad
					\mathcal{O}_3:=\left\{\w; |\w|\ge C_0 \right\},
				\end{equation*}
where $\dz_0\ll1$ and $C_0\gg1$ to be determined later. The estimate $\int_{\mathcal{O}_3} \lsim \m^{-2}\,e^{-\m}$ follows from a similar argument as in the estimate of ${\mathcal Q}_3$ in Subsection \ref{ss71}  with $C_0$ large enough. The estimate for $\int_{\mathcal{O}_1}$ is also similar to that by choosing $\dz_0$ small enough and using the following assertion:  $0.8\cdot 4^{-1}|x|^2 \, \bH_0(\w) \ge \m$ for $\w \in \mathcal{O}_1$. In fact, from \eqref{defu} and \eqref{obes1} we can check that
				\begin{equation}\label{simA}
					\bA(\w_1)=|u| (1+  O_{C_0} (|u|^{-\frac12})) \sim \bA (|\w|) \sim \ww^2 \sim |u|, \quad \forall\, |\w| < C_0.
				\end{equation}
				Therefore, $\u_0(\w)\sim|\w|^{-2}\gsim \dz_0^{-2}$ on $\mathcal{O}_1$. By this and \eqref{aPsi}, we get that:
				\begin{align*}
					\bH_0(\w) &=\var \ww^2 \, |\w|^2 \, \u_0(\w) \,  \Big( 1 + O(\u_0(\w)^{-\frac{1}{2}}) \Big) \\
					&= \var |u|\, \left[1+ O_{C_0} (|u|^{-\frac12}) + O(\dz_0)\right] = \var |u|\, (1+o(1)).
				\end{align*}
				 However, recalling $\varphi_1(\pi) = \pi$ (cf. \eqref{defs}), the third equality in \eqref{dEn2} implies that $\bH(1) \le \varphi_1(|\tz|) |u| = \pi |u|(1+o(1))$ since $|\tz|\to\pi$ by the third estimate in \eqref{obes1}, which readily yields the assertion.
				
				To estimate $\int_{\mathcal{O}_2}$,  first recall that $|u|\gg1$ and $|x|\, d(g) \le1$. Fixing selected $\dz_0$ and $C_0$, then one can conclude similarly as in the proof of \eqref{prep3}  (a rougher argument is enough) that
				\begin{equation} \label{HHc}
					\left| \,  |x|^2 \bH_0(\w)- |x|^2 \bH(|\w|) \, \right| \lsim 1, \quad \forall\, \w\in \mathcal{O}_2.
				\end{equation}
				In particular, $|x|^2 \bH_0(\w) \sim |x|^2 \bH(|\w|) \ge 4 \m \gg1$. Applying these estimates  with the  $2$-dimensional polar coordinates $\w=\rho \,(\cos\gz, \sin\gz)$ to $\int_{\mathcal{O}_2}$, we infer it is bounded by
				\begin{equation*}
					\m^{-\frac32}\int_{\dz_0}^{ C_0} \rho^{2n_0}  \exp\left\{-\frac{|x|^2}{4}\bH(\rho) \right\} \, d \rho.
				\end{equation*}
    Recalling that $1$ is the only minimal point of $\bH$ with $\chri = |x|^2 \bH''(1)/4 \sim d(g)^2 \sim \m$, so it is exactly majorized by $\m^{-2} e^{-\m}$ via the standard Laplace's method. This finishes the proof of {\bf Case (3)}.
				
				\subsection{Proof of {\eqref{target}} in {\bf Case (4)}} \label{rot_der}
				To prove \eqref{target} in this case,
				the following estimate is crucial:
				\begin{equation}\label{hkb2}
					\int_{\rr^2} \left| s^\iota\, \rP_\az(s) \right| \,  d s \lsim_{\iota,\az} d(g)^{|\iota|}\, p(g), \quad \forall \, \iota\in\nn^2,\, \az\in\nn^3,
				\end{equation}
				where  $|\tz(\ggg)|\ge1$, $\m(\ggg) \to\infty$, and
				\begin{equation*}
					\rP_\az(s):= \frac{1}{\pi}\, e^{-\frac{|x|^2}{4}}\, \rP_\az\left(2 s, t + |x| s_1  \, e_2 + |x| s_2   \, e_3\right), \quad s\in \rr^2,
				\end{equation*}
				with
				\begin{equation*}
					\rP_\az(X,T)=  \rP^{0, \az}(X, T) = \int_{\R^3}  \, \vv(\lambda) \, \lz^\az \, e^{-\frac{1}{4} \Ga((X,T);\lambda)} \, d\lambda, \quad (X,T) \in\rr^2 \times\rr^3.
				\end{equation*}
Remark that the upper bound \eqref{PnaBD} is valid for $|\rP_\az(X,T)|$. Then following a similar line of reasoning as in Sections \ref{bigchr} and \ref{schr}, we can obtain \eqref{hkb2} under the condition that $|\tz(\ggg)|\ge1$ and $\m(\ggg) \to\infty$.
				
				The rest proof of \eqref{target} is then simple based on \eqref{hkb2} and we shall go on in the following three steps. Firstly, without loss of generality we may assume that $|x|^2\ge x_2^2+x_3^2\ge |x|^2/3 > 0$, then taking the orthogonal transform
				\begin{equation} \label{O0x}
					O_0(x):=\begin{pmatrix}
						|x|^{-1}x_1 & |x|^{-1}x_2 & |x|^{-1}x_3 \\
						0& \frac{x_3}{\sqrt{x_2^2+x_3^2}} &\frac{-x_2}{\sqrt{x_2^2+x_3^2}} \\
						-\frac{\sqrt{x_2^2+x_3^2}}{|x|} & \frac{x_1 x_2}{|x|\sqrt{x_2^2+x_3^2}} & \frac{x_1 x_3}{|x|\sqrt{x_2^2+x_3^2}}
					\end{pmatrix}
				\end{equation}
				in \eqref{prpty_p} and applying \eqref{Hor90} to \eqref{ehk2} we arrive at that
				\begin{equation}\label{rot1}
					p(g)= \frac{1}{\pi}\, e^{- \frac{|x|^2}{4}} \, \int_{\R^2}\rP\left(2 s, \tilde{t} + |x| s_1  \, e_2 + |x| s_2  \,  e_3\right)  ds,
				\end{equation}
				where
				\begin{equation*}
					\tilde{t}:=O_0(x) \, t=\left(\frac{x\cdot t}{|x|},\, \frac{t_2 x_3 - t_3 x_2 }{\sqrt{x_2^2+x_3^2}}, \, \frac{-t_1 (x_2^2+x_3^2) + t_2 \, x_1 x_2 + t_3 \, x_1 x_3}{|x|\sqrt{x_2^2+x_3^2}} \right).
				\end{equation*}
Note that to reduce such points to ones obeying \eqref{ass}, one only needs to use some orthogonal transform (in $\rP$ or $\rP_\az$) again. This will be used implicitly in what follows. Moreover, the parameter $\tilde{t}$, viewed as a function of $(x,t)$, is homogeneous on $x$ of degree $0$ and on $t$ of degree $1$ respectively.
				
Secondly, suppose for the moment that $|x|^2\ge|t|$. Take derivatives in \eqref{rot1}. By Leibniz's rule, the chain rule, \eqref{hkb2} and the homogeneous property of $\tilde{t}$, we see that to show \eqref{target} it is enough to observe that, for instance, for any $1\le i,j\le3$,
					\begin{gather*}
						|\partial_{x_i} (\tilde{t})| \lsim |x|^{-1} |t| \lsim d(g), \quad |\partial_{t_i} (\tilde{t})|\lsim 1;\\[1mm]
						|\partial_{x_i}\partial_{x_j} (\tilde{t})| \lsim |x|^{-2}|t| \lsim d(g)^2, \quad |\partial_{x_i}\partial_{t_j} (\tilde{t})| \lsim |x|^{-1} \lsim d(g),\quad |\partial_{t_i}\partial_{t_j} (\tilde{t})|=0.
					\end{gather*}
					Other possible situations can also be verified easily.
				
Thirdly, let us improve the above result to the desired case $|x|\, d(g)\ge1$, for which we only need to cope with the remaining case $|x|^2\le|t|$. In particular one has $|t|\sim d(g)^2 \gg1$. We may assume $\tilde{t}_2^2 + \tilde{t}_3^2\ge |t|^2/3$ as well, since only homogeneous property plays a role when estimating the derivatives of $\tilde{t}$. Notice also that a change of variables $\lambda \mapsto O \lambda$ in the definition of $\rP$ (cf. \eqref{defP2}) gives the following
\[
\rP(X,T) = \rP(X,O \, T), \quad \forall \, O \in \mathrm{O}_3.
\]
Then applying the previous formula to \eqref{rot1} with $O = O_0(\tilde{t})$ (cf. \eqref{O0x}), we get:
				\begin{equation}\label{rot2}
					p(g)= \frac{1}{\pi}\, e^{- \frac{|x|^2}{4}}\, \int_{\R^2}\rP\left(2 s, \, |t| + |x|\, \frac{\tilde{t}_2 \,  s_1 + \tilde{t}_3  \, s_2}{|t|}, \, |x| \, \frac{\tilde{t}_3 \,  s_1 - \tilde{t}_2  \, s_2}{\sqrt{\tilde{t}_2^2 + \tilde{t}_3^2}}, \, |x| \, \frac{\tilde{t}_1\tilde{t}_2 \,  s_1 + \tilde{t}_1\tilde{t}_3  \, s_2}{|t|\sqrt{\tilde{t}_2^2 + \tilde{t}_3^2}} \right)  ds.
				\end{equation}
				Observe that all coefficients of $s$ in the $T$-variable of \eqref{rot2} is homogeneous on $x$ of degree $1$ and on $t$ of degree $0$. Then using a similar argument as in the case $|x|^2\ge|t|$, one can easily check that the condition $|x| \, d(g)\ge1$ and the estimate \eqref{hkb2} are sufficient for \eqref{target}. This proves {\bf Case (4)}, and hence Theorem \ref{hkDb}.
			\end{proof}

			\begin{appendices}
				
				\section{Proof of Proposition \ref{lP}} \label{secA}
				\setcounter{equation}{0}
				
					Recall the functions $\rP(X,T)$, $\rD(X,T)^2$ and $\tau^*(X,T)$ introduced in Subsections \ref{ideaS} and \ref{sec51}-\ref{sec52}, and set $\ep_*:=\var - |\tau^*|$. Then \eqref{Phi_est} implies that
					\begin{align}\label{eD1}
						\rD(X,T)^2 \sim \frac{|X|^2}{(\vartheta_1 - |\tau^*|)^2} = \ep_*^{-2}\, |X|^2, \quad \mbox{if} \ |X| \ne 0.
					\end{align}
					
				Let us now begin with the positivity of $\rP$. Indeed, the argument below also allows us to establish \cite[Proposition~2.2]{LZ212}. The fact is implicitly indicated therein.

				\subsection{Proof of Proposition \ref{lP} (i)} \label{AX1}
				
				\begin{proof}
					It follows from \eqref{defP2}-\eqref{MFF2} that
					\begin{align}\label{limPk}
						\rP(X,T) =3 \, e^{-\frac34|X|^2}  \,
						\lim_{l_0 \to +\infty} \int_{\R^3} e^{iT \cdot \lambda} \, \prod_{l = 1}^{l_0} \rP_{l}(X;\lambda)  \, d\lambda,
					\end{align}
					where
					\begin{align*}
						\rP_{l}(X;\lambda):= \left( 1 + \frac{|\lambda|^2}{\vartheta_l^2}\right)^{-1} \, \exp\left\{ \frac12 \left( 1 + \frac{|\lambda|^2}{\vartheta_l^2}\right)^{- 1} |X|^2 - \frac12 \, |X|^2 \right\}.
					\end{align*}
					We apply \eqref{Hor90} with $q = 2$, $A = \left(1 + \frac{|\lambda|^2}{\vartheta_l^2}\right) \I_2$ and $Y = - i X$, respectively, to get that
					\begin{align*}
						\rP_{l}(X;\lambda) = \frac{1}{2\pi} \, \int_{\R^2} e^{-\frac{1}{2}|s_l - X|^2} e^{- \frac{|\lambda|^2}{2 \vartheta_l^2} |s_l|^2}  ds_l.
					\end{align*}
					As a result, for every  $j_0, l_0 \in \N^*$ with $l_0 - j_0 \ge 1$,  $X \in \R^2$, and $T \in \R^3$, we obtain
					\begin{align*}
						Q(j_0, l_0; X, T)  :=& \int_{\R^3} e^{iT \cdot \lambda} \, \prod_{l = j_0}^{l_0} \rP_{l}(X;\lambda)  \, d\lambda  \\
						=& \frac{1}{(2 \pi)^{l_0 - j_0 + 1}} \, \int_{(\R^2)^{l_0 - j_0 + 1}} \prod_{l = j_0}^{l_0} e^{-\frac{1}{2} |s_l -  X|^2} ds \int_{\R^3} \exp\left\{ - \frac{|\lambda|^2}{2} \sum_{l = j_0}^{l_0} \frac{|s_l|^2}{ \vartheta_l^2}  + i T \cdot \lambda  \right\}  d\lambda  \\
						=& \frac{1}{(2\pi)^{l_0 - j_0 - \frac{1}{2}}} \, \int_{(\R^2)^{l_0 - j_0 + 1}} \frac{\prod\limits_{l = j_0}^{l_0} e^{-\frac{1}{2}|s_l - X|^2} \, }{ \left( \sum\limits_{l = j_0}^{l_0} \frac{|s_l|^2}{ \vartheta_l^2}  \right)^{\frac32} }  \, \exp\left\{ - \frac{|T|^2}{2 \sum\limits_{l = j_0}^{l_0} \frac{|s_l|^2}{ \vartheta_l^2}}\right\} \, ds > 0,
					\end{align*}
					where $s = (s_{j_0}, \ldots, s_{l_0}) \in (\R^2)^{l_0 - j_0 + 1}$ and we have used \eqref{Hor90} again in the last ``$=$''. Hence we get that $Q(1, 2; X, T) > 0$, and by induction $Q(3, +\infty; X, T) \ge 0$.
					
					It remains to show that $Q(1, +\infty; X, T) > 0$.
					In fact, the basic properties of the Fourier transform and convolution give that
					\begin{align}\label{Plan}
						Q(1, +\infty; X, T) = (2 \pi)^{-3} \, \int_{\R^3} Q(1, 2; X, \xi) \, Q(3, +\infty; X, T - \xi) \, d\xi.
					\end{align}
					And it suffices to recall the continuous function $Q(1, 2; X, \cdot) > 0$ and notice that the nonnegative continuous function  $Q(3, +\infty; X, \cdot)$ satisfies $Q(3, +\infty; X, 0) > 0$.
				\end{proof}
				
				\subsection{Proof of Proposition \ref{lP} (ii) and (iii)}
				
				It is enough to prove (iii) in detail, since (ii) will follow easily from it and the fact that $\rP$ is positive and continuous we have just obtained, with estimates \eqref{aI0}, \eqref{qvar2} and \eqref{eD1}.
					
					To show (iii), we consider two cases as in the proposition below:
				
\begin{prop}\label{lP2}
					The following  uniform asymptotic estimates  hold:
					{\em\begin{compactenum}[(I)]
							\item If $X \ne 0$, $\rD(X,T) \to +\infty$, and $\vartheta_1 - |\tau^*| \gtrsim 1$, then we have
							\begin{align} \label{nax1}
								\rP(X,T) =  \frac{(8\pi)^{\frac{3}{2}} \, e^{-\frac{\rD(X,T)^2}{4}}}{\det(- \He_{\tau^*} \, \Gamma((X,T); \tau^*))^{\frac{1}{2}}}  \,
								\vv(i\tau^*)  \, (1 + o(1)).
							\end{align}
							\item If $X \ne 0$, $\rD(X,T) \to +\infty$ and $|\tau^*| \to  \vartheta_1^-$, then we have
							\begin{align} \label{nax2}
								\rP(X,T) &=  (2\pi)^2 \, \frac{4 \vartheta_1^2 }{ - \sin{\vartheta_1}  }  \,e^{-\frac{\rD(X,T)^2}{4}} \, e^{- \frac{\vartheta_1 |X|^2}{2(\vartheta_1 - |\tau^*|)}}
								I_0 \left( \frac{ \vartheta_1 |X|^2}{2(\vartheta_1 - |\tau^*|)} \right)\frac{ (\vartheta_1 - |\tau^*|)^2}{ |X|^2  }  \, (1 + o(1)).
							\end{align}
					\end{compactenum}}
				\end{prop}
				
				Now we show how (iii) of Proposition \ref{lP} follows from Proposition \ref{lP2}. First we consider the case where $\rD(X,T) \to +\infty$ with $X \ne 0$ and $\vartheta_1 - |\tau^*| \gtrsim 1$. In such case, a direct computation shows that
				\begin{equation}\label{HeGam}
					-\He_{\tau^*} \,  \Gamma((X,T);\tau^*) =|X|^2 \left[- \frac{\Upsilon'(|\tau^*|)}{|\tau^*|} \I_3 + \left(-\Upsilon''(|\tau^*|) + \frac{\Upsilon'(|\tau^*|)}{|\tau^*|} \right) \frac{\tau^*}{|\tau^*|} \left(\frac{\tau^*}{|\tau^*|}\right)^\T \right],
				\end{equation}
				and thus,  by Schur's Lemma,
				\[
				\det(- \He_{\tau^*} \, \Gamma((X,T); \tau^*)) = -\Upsilon^{\prime\prime}(|\tau^*|) \left( - \frac{\Upsilon^\prime(|\tau^*|)}{|\tau^*|}\right)^2 \, |X|^6.
				\]
				Combining this,  \eqref{nax1}  with \eqref{asinI0} and the following observation:
				\begin{align}\label{relvU}
					\vv(i \tau^*) = \Upsilon(|\tau^*|) \frac{|\tau^*|}{\sin{|\tau^*|}},
				\end{align}
				we obtain the asserted asymptotic in Proposition \ref{lP} without difficulties.
				
				For the case where $\rD(X,T) \to +\infty$ with $X \ne 0$  and $0 < \var - |\tau^*| \ll 1$,  one can appeal to \eqref{nax2} and \eqref{q_var}.
				
				\subsection{Proof of {\eqref{nax1}} }\label{SA3}
				
				 It suffices to use the argument introduced in \cite{Li20}, namely, the conjunction of the method of stationary phase and the operator convexity.
				To this end, using \eqref{MFF3} and \eqref{deD}, the argument in
				the proof of Lemma \ref{rephi} implies its counterpart:
				\begin{lem} \label{reGa}
					There exists a constant $c > 0$ such that for any $(X, T) \in \R^2 \times \R^3$, we have
					\[
					\Re [\Gamma((X,T); \tau^* - i \lambda) - \Gamma((X,T); \tau^* )] \ge c \,
					\frac{|\lambda|^2}{1 + |\lambda|^2} |X|^2, \quad \forall \, \lambda \in \R^3.
					\]
				\end{lem}
			
			For $0 < \var - |r| \le \zeta_0 \le 1$,  it follows from \eqref{Ups_ser} that
			\begin{equation*}
				-\frac{\Upsilon'(r)}{r} \sim_{\zeta_0} 1, \qquad
				0\le -\Upsilon''(r) + \frac{\Upsilon'(r)}{r} =
				16 \sum_{k=1}^{+\infty} \frac{\vtz_k^2 \, r^2}{(\vtz_k^2-r^2)^3} \lsim_{\zeta_0} 1.
			\end{equation*}
			 By this and \eqref{HeGam}, we get that
			\begin{align}\label{Hes_sim}
				-\He_{\tau^*} \,  \Gamma((X,T);\tau^*) \sim_{\zeta_0} |X|^2 \,  \I_3.
			\end{align}
			Then we split $\R^3$ into disjoint sets $\{ \lambda; \, |\lambda| \le |X|^{-\frac{3}{4}} \}$ and $\{ \lambda; \, |\lambda| > |X|^{-\frac{3}{4}} \}$. Using a similar argument as in the proof of Theorem \ref{asyab}, one obtains immediately the desired result.
			
			\subsection{Proof of {\eqref{nax2} }}
			Our assumptions are now $X \ne 0$, $\rD(X,T) \to +\infty$ and $|\tau^*| \to \vartheta_1$. In principle we use the same argument as in \cite[\S5 -- \S6]{LZ22}. The main difference is that in this case $q = 2 < m = 3$. Fortunately, the functions $\vv$ and $\Ga$ in the definition of $\rP$ (cf. \eqref{defP2})  are explicit enough for us to get a better leading term than the ones in \cite{LZ22}. In fact, the function $\rP$ is more or less a	 heat kernel at time $1$ on $H$-type groups, which has a more concrete leading term. See \cite{Li10} for more details.
			
			Let us start with the
			
			\subsubsection{Preliminaries}
		Recalling \eqref{eD1}, now we have $\rD(X,T)^2 \gg |X|^2$. As in Subsection \ref{ideaS},
		we need a new formula of $\rP$ to proceed our estimates. In fact, set in the sequel for $r \in \R$ and $\lambda \in \R^3$
		\begin{gather} \label{defcV2}
			\vv_2(\lambda) := 3 \, \prod_{k = 2}^{+\infty} \left( 1 + \frac{\lambda \cdot \lambda}{\vartheta_k^2} \right)^{-1}, \quad   \widetilde{\Upsilon}_2(r) := 3 + 2 \sum_{k = 2}^{+\infty} \frac{r^2}{\vartheta_k^2 + r^2}, \\
			\label{MFF22}
			\Ga_2((X,T);\lambda) :=  \widetilde{\Upsilon}_2(|\lambda|) \, |X|^2  - 4i\, T \cdot \lambda.
		\end{gather}
		Then from the definition we find that
		\begin{align*}
			\vv(\lambda) \, e^{-\frac{\Ga((X,T);\lambda)}{4}} &=\left( 1 +  \frac{|\lambda|^2}{\vartheta_1^2}\right)^{-1} e^{\frac{1}{2} \frac{\vartheta_1^2}{\vartheta_1^2 + |\lambda|^2} |X|^2} \,  e^{-\frac{|X|^2}{2}} \, \vv_2(\lambda) \, e^{-\frac{\Ga_2((X,T);\lambda)}{4}} \\
			&=  \,  e^{-\frac{|X|^2}{2}} \, \vv_2(\lambda) \, e^{-\frac{\Ga_2((X,T);\lambda)}{4}} \frac{1}{2\pi} \int_{\R^2} e^{-\frac{1}{2} \left( 1 +  \frac{|\lambda|^2}{\vartheta_1^2}\right) |\eta|^2 + \eta \cdot X} d\eta,
		\end{align*}
		where we have used in the last equality \eqref{Hor90} with $q = 2$, $A = \left( 1 +  \frac{|\lambda|^2}{\vartheta_1^2}\right) \I_2$, and $Y = - i X$, respectively. Hence
		\begin{align}\label{expP2}
			\rP(X,T) = \frac{1}{2\pi} \int_{\R^2} e^{-\frac{1}{2} |\eta - X|^2} d\eta \int_{\R^3} \, \vv_2(\lambda) e^{-\frac{1}{4} \Ga_2((X,T);\lambda) - \frac{|\eta|^2}{2\vartheta_1^2 } |\lambda|^2} \, d\lambda.
		\end{align}
		
		As in \eqref{MFF3}, we define
			\begin{align}\label{MFFi22}
				\Gamma_2((X,T);\tau) :=  \Ga_2((X,T);i \tau).
			\end{align}
			
		Here we state a counterpart of Lemma \ref{reGa}, which follows from the same argument.
		
		\begin{lem} \label{reGa2}
			There exists a constant $c > 0$ such that
			\[
			\Re [\Gamma_2((X,T);   \tau^* - i\lambda) - \Gamma_2((X,T); \tau^* )] \ge c \,
			\frac{|\lambda|^2}{1 + |\lambda|^2} |X|^2, \quad \forall \, \lambda \in \R^3.
			\]
		\end{lem}
		
		It follows from the definition that
		\begin{gather*}
			\rD(X,T)^2 = \Gamma((X,T); \tau^*) = \Gamma_2((X,T); \tau^*) -  \frac{2 |\tau^*|^2}{\vartheta_1^2 - |\tau^*|^2} |X|^2,  \\
			0 = \nabla_{\tau^*}\,\Gamma((X,T); \tau^*) =  \nabla_{\tau^*} \, \Gamma_2((X,T); \tau^*) - \frac{4\tau^* \vartheta_1^2}{(\vartheta_1^2 - |\tau^*|^2)^2} |X|^2.
		\end{gather*}
		
		Then deforming the contour from $\R^3$ to $\R^3 + i \tau^*$ in the inner integral of \eqref{expP2}, and applying the change of variables $\eta = \left( 1 - \frac{|\tau^*|^2}{\vartheta_1^2} \right)^{-\frac{1}{2}} \xi$ for the outer integral, we yield that
		\begin{align}\label{ndefP}
			\rP(X,T) = \frac{1}{2\pi} \left( 1 - \frac{|\tau^*|^2}{\vartheta_1^2} \right)^{-1} e^{-\frac{\rD(X,T)^2}{4}} \, \int_{\R^2} e^{-\frac{1}{2} |\xi - \bY|^2} \Ft(X,T;\xi)  \, d\xi
		\end{align}
		where
		\begin{align}\label{defPP2}
			\Ft(X,T;\xi)&:= \int_{\R^3}  \, \vv_2(\lambda + i\tau^*)  e^{-\frac{1}{4} \bS(X,T;\lambda) + i \lambda \cdot \bW(X,T;\xi) - \frac{1}{2} \frac{|\xi|^2 |\lambda|^2}{\vartheta_1^2 - |\tau^*|^2}} \, d\lambda,
		\end{align}
		with
		\begin{gather}
			\bS(X,T;\lambda) := \Gamma_2((X,T);\tau^* - i\lambda) - \Gamma_2((X,T); \tau^*) +i \nabla_{\tau^*} \, \Gamma_2((X,T); \tau^*) \cdot \lambda, \nonumber \\
			\bW(X,T;\xi) := \frac{|\bY|^2 - |\xi|^2}{\vartheta_1^2 - |\tau^*|^2} \tau^*, \quad \mbox{with } \quad \bY := \left( 1 - \frac{|\tau^*|^2}{\vartheta_1^2} \right)^{-\frac{1}{2}} X. \label{nAna2}
		\end{gather}

		Finally, we split up the integral in \eqref{ndefP} as
		\[
		\int_{\R^2} := \sum_{i = 1}^5 \int_{\blacklozenge_i} := \sum_{i = 1}^5 \cK_i,
		\]
		where
		\begin{gather*}
			\blacklozenge_1 := B_{\R^2}(0, \frac{1}{2} \, |\bY|), \quad  \blacklozenge_2 := B_{\R^2}\left(0, |\bY| -  \sqrt{\epsilon_*} \,  \rD(X,T)^{\frac{1}{4}}\right) \setminus  B_{\R^2}(0, \frac{1}{2} \, |\bY|), \\[1mm]
			\blacklozenge_3 := B_{\R^2}\left(0, |\bY| +  \sqrt{\epsilon_*} \,  \rD(X,T)^{\frac{1}{4}}\right)  \setminus B_{\R^2}\left(0, |\bY| -  \sqrt{\epsilon_*} \,  \rD(X,T)^{\frac{1}{4}}\right) , \\[1mm]
			\blacklozenge_4 := B_{\R^2}(0, 2 \, |\bY|) \setminus B_{\R^2}\left(0, |\bY| +  \sqrt{\epsilon_*} \,  \rD(X,T)^{\frac{1}{4}}\right) , \quad \blacklozenge_5 := B_{\R^2}(0, 2 \, |\bY|)^c.
		\end{gather*}
		 Notice that from the definitions of $\bY$ and $\epsilon_*$, \eqref{eD1} implies that
		\begin{align} \label{nAna1}
			|\bY| \sim  \frac{|X|}{\sqrt{\epsilon_*}} \sim \sqrt{\ep_*} \, \rD(X,T) \gg \sqrt{\epsilon_*} \, \rD(X,T)^{\frac{1}{4}}.
		\end{align}
		
		\subsubsection{Properties of $\Ft$}
		
		We will see that the main contribution of the integral in \eqref{ndefP} comes from the part of the integral taken only over the domain $\blacklozenge_3$.  To this end, we would like to study the asymptotic behavior of $\Ft$ on $\blacklozenge_3$ and the upper bound of $\Ft$ for other regions.
		
		For this purpose, let us first introduce
		\begin{gather}
			\cnA(X,T;\xi) :=  - \frac{1}{4} \He_{\tau^*} \, \Gamma_2((X,T); \tau^*) + \frac{|\xi|^2 }{\vartheta_1^2 - |\tau^*|^2} \I_3, \label{defcA} \\
			\bS_*(X,T;\lambda) := \bS(X,T;\lambda) + \frac{1}{2} \lambda^\T \, \He_{\tau^*} \, \Gamma_2((X,T); \tau^*) \, \lambda.\nonumber
		\end{gather}
		Then $\bS_*
		(X,T;\cdot)$ vanishes to order 2 at the origin, and the phase function of the integral in \eqref{defPP2} becomes
		\[
		- \frac{1}{4} \bS_*(X,T;\lambda) - \frac{1}{2} \lambda^\T \, \cnA(X,T;\xi) \, \lambda + i \lambda \cdot \bW(X,T;\xi).
		\]
		
		Now we can establish:
		
		\begin{prop} \label{estK3}
			 We have uniformly for all $\xi \in \blacklozenge_3$ that:
			\[
			\Ft(X,T;\xi) = (2\pi)^{\frac{3}{2}} \, \vv_2(i\tau^*) \,
			\frac{\exp\left\{- \frac{1}{2} \bW(X,T;\xi)^\T \, \cnA(X,T;\xi)^{-1} \, \bW(X,T;\xi)\right\}}{\sqrt{\det{\cnA(X,T;\xi)}}}\,(1 + o(1)).
			\]
			 \end{prop}
		
		\begin{proof}
			Recall $\epsilon_* := \vartheta_1 - |\tau^*|$. By \eqref{nAna1} and the first equality of \eqref{nAna2},  we have for all $\xi \in \blacklozenge_3$ that:
			\begin{gather}\label{estA}
				\frac{|\xi|^2 }{\vartheta_1^2 - |\tau^*|^2} \sim \frac{|\bY|^2}{\epsilon_*} \sim \rD(X,T)^2, \\
				\label{estW}
				|\bW(X,T;\xi)| \lesssim \frac{(|\xi| + |\bY|) \, | \, |\xi| - |\bY| \, |}{\epsilon_*} \lesssim \rD(X,T)^{\frac{5}{4}}.
			\end{gather}
			Then from the definition of $\Gamma_2((X,T);\cdot)$ (cf. \eqref{MFFi22}), a simple computation (compared with \eqref{Hes_sim}) shows that
			\begin{align*}
				-\He_{\tau^*} \, \Gamma_2((X,T);\tau^*) \sim |X|^2 \, \I_3 \ll \rD(X,T)^2 \, \I_3 \sim \frac{|\xi|^2 }{\vartheta_1^2 - |\tau^*|^2} \, \I_3, \quad \forall \, \xi \in \blacklozenge_3,
			\end{align*}
			which implies that, for all  $\xi \in \blacklozenge_3$,
			\begin{align} \label{AXT}
				\cnA(X,T;\xi) = \frac{|\xi|^2 }{\vartheta_1^2 - |\tau^*|^2}  \, \I_3 \,  (1 + o(1)) = \frac{|\bY|^2 }{\vartheta_1^2 - |\tau^*|^2}  \, \I_3  \, (1 + o(1)) \sim \rD(X,T)^2 \, \I_3.
			\end{align}
			
			Now setting
			\[
			\rho_*(X,T;\xi) :=  \cnA(X,T;\xi)^{-1} \,  \bW(X,T;\xi),
			\]
			we have $|\rho_*(X,T;\xi)| = O(\rD(X,T)^{-\frac{3}{4}})$ by \eqref{AXT} and \eqref{estW}. Then lifting the contour again by $i\rho_*(X,T;\xi)$ in \eqref{defPP2}, and splitting the integral into the ones over $\{\lambda; \, |\lambda| \le \rD(X,T)^{-\frac{3}{4}}\}$ and $\{\lambda; \, |\lambda| > \rD(X,T)^{-\frac{3}{4}}\}$. It follows from the estimate of $\rho_*(X,T;\xi)$ above that
			\[
			\bS_*(X,T;\lambda + i\rho_*(X,T;\xi)) =  O(\rD(X,T)^{-\frac14})
			\]
			in the first region and thus using a similar  argument in the proof of Theorem \ref{asyab}, we obtain the desired proposition.
		\end{proof}
		
		Next we give an upper bound of $\Ft$ for all $\xi \in \R^2$.
		
		\begin{prop} \label{uppP2}
			There exists a constant $c > 0$ such that
			\[
			|\Ft(X,T;\xi)| \lesssim \exp\left\{ - c \, \frac{|\bW(X,T;\xi)|}{\sqrt{|X|^2 + |\xi|^2/\epsilon_* + 1}}\right\}.
			\]
		\end{prop}
		
		\begin{proof}
			We deform the contour in \eqref{defPP2} from $\R^3$ to $\R^3 + i \rho^*(X,T;\xi)$, where we choose
			\[
			\rho^*(X,T;\xi) := \frac{c}{\sqrt{|X|^2 + |\xi|^2/\epsilon_* + 1}} \frac{\bW(X,T;\xi)}{|\bW(X,T;\xi)|}
			\]
			for some small $c > 0$. From our choice and the analyticity of $\widetilde{\Upsilon}_2$ on $\{ z \in \C; \  -\vartheta_2 < \Im(z) < \vartheta_2 \}$,   we have
			\[
			\widetilde{\bS}(X,T;\xi) :=	-\frac{1}{4}\bS_*(X,T;i\rho^*(X,T;\xi)) + \frac{1}{2 }  \rho^*(X,T;\xi)^\T \cnA(X,T;\xi) \rho^*(X,T;\xi) = O(1)
			\]
			since $\cnA(X,T;\xi) \sim \left( |X|^2 + |\xi|^2 / \epsilon_* \right) \I_3$. Finally it follows from Lemma \ref{reGa2} and the inequality $|\vv_2(i \tau + \xi)| \le \vv_2(i\tau) \vv_2(\xi)$  for any $|\tau| < \vartheta_2$ and any $\xi \in \R^3$ (compared with \eqref{nAnz1}) that
			\begin{align*}
				|\Ft(X,T;\xi)| &\lesssim \exp\left\{ \widetilde{\bS}(X,T;\xi) -\bW(X,T;\xi) \cdot \rho^*(X,T;\xi)\right\}\\
				&\lesssim \, \exp \left\{ -\bW(X,T;\xi) \cdot \rho^*(X,T;\xi)\right\}  = \exp\left\{ - c \, \frac{|\bW(X,T;\xi)|}{\sqrt{|X|^2 + |\xi|^2/\epsilon_* + 1}}\right\}.
			\end{align*}
			This proves the proposition as required.
		\end{proof}
		
		\subsubsection{Estimation of $\cK_3$}
		To treat $\cK_3$, we need the following lemma, which is an analogue of \cite[Proposition 4.2]{Li12}.
		
		\begin{lem}
			Let $Y \in \R^2 \setminus \{0\}$, $r > 0$, $\delta \in (0, \ |Y|)$ and set
			\begin{align*}
				\widetilde{\cI}_\delta(Y, r) := \int_{\{w \in \R^2; \, |Y|- \delta \le |w| < |Y|+ \delta\}}  e^{-\frac{1}{2} |w - Y|^2} e^{-\frac{(|w|^2 - |Y|^2)^2}{r}} \, dw.
			\end{align*}
			Then we have uniformly
			\begin{align*}
				\widetilde{\cI}_\delta(Y, r) =
				\pi^{\frac{3}{2}} \, \sqrt{r} \, e^{-|Y|^2}  I_0(|Y|^2) \, (1 + o(1)),
			\end{align*}
			provided
			\begin{align} \label{asyass}
				\frac{\delta}{|Y|} \le \frac{1}{2}, \quad \frac{\delta^2 |Y|^2}{r} \to + \infty , \quad \mbox{and} \quad \frac{|Y|^2}{r}  \to +\infty.
			\end{align}
		\end{lem}
		
		\begin{proof}
			First notice that the change of variables $w \mapsto O w$ gives $\widetilde{\cI}_\delta(Y, r) = \widetilde{\cI}_\delta(O Y, r)$ for any $O\in{\rm O}_2$. Then without loss of generality we can assume $Y = |Y| (1,0)$.
			Using polar coordinates $w=  \rho \, (\cos{\gz}, \sin{\gz})$ and the integral representation for $I_{0}$ (cf. \eqref{defI0}), we get
			\begin{align*}
				\widetilde{\cI}_\delta(Y, r) &= e^{- \frac{1}{2} |Y|^2} \int_{\{w \in \R^2; \, |Y|- \delta \le |w| < |Y|+ \delta\}}  e^{-\frac{1}{2} |w|^2 + |Y| w_1 } e^{-\frac{(|w|^2 - |Y|^2)^2}{r}} \, dw \\
				&= e^{- \frac{1}{2}|Y|^2}
				\int_{|Y| - \delta}^{|Y| + \delta} \rho \, e^{-\frac{1}{2} \rho^2 - \frac{(\rho^2 - |Y|^2)^2}{r}} \, d\rho
				\int_{-\pi}^\pi e^{\rho \, |Y| \cos{\gz}}  \, d\gz \\
				&= 2\pi \, e^{-\frac{1}{2}|Y|^2} \int_{|Y| - \delta}^{|Y| + \delta} \rho
				I_{0}(|Y| \rho) \, e^{-\frac{1}{2} \rho^2 - \frac{(\rho^2 - |Y|^2)^2}{r}} \, d\rho \\
				&= 2\pi \, e^{-\frac{1}{2}|Y|^2} |Y|^2 \int_{1 - \frac{\delta}{|Y|}}^{1 + \frac{\delta}{|Y|}} u I_0(|Y|^2 u) e^{-\frac{|Y|^2}{2} u^2} e^{-\frac{|Y|^4}{r} (u + 1)^2(u - 1)^2} \, du,
			\end{align*}
			where in the last ``$=$'' we have used the change of variables $\rho = |Y| u$ additionally. Then we split the proof into two cases:
			
			\paragraph{Case 1: $|Y| \lesssim 1$.} Noting $|Y|^4/r \ge 4 \delta^2 |Y|^2/r \to +\infty$,  the standard Laplace's method gives
			\[
			\widetilde{\cI}_\delta(Y, r) = \pi^{\frac{3}{2}} \sqrt{r} \, e^{-|Y|^2}  I_0(|Y|^2) \, (1 + o(1)),
			\]
			which ends the proof.
			
			\paragraph{Case 2: $|Y| \to +\infty$.} In such case, from \eqref{asinI0} we get
			\[
			\widetilde{\cI}_\delta(Y, r)  = \sqrt{2\pi} \,  |Y| \int_{1 - \frac{\delta}{|Y|}}^{1 + \frac{\delta}{|Y|}} \sqrt{u} \, e^{-\frac{|Y|^2}{2} (u - 1)^2 \left[ 1 + \frac{2|Y|^2}{r} (u + 1)^2\right]} du \, (1 + o(1)).
			\]
			Then the usage of the Laplace's method yields
			\[
			\widetilde{\cI}_\delta(Y, r) = \frac{\pi}{\sqrt{2}} \, \frac{\sqrt{r}}{|Y|} \, (1 + o(1)) = \pi^{\frac{3}{2}} \sqrt{r}  e^{-|Y|^2}  I_0(|Y|^2) \, (1 + o(1)),
			\]
			where in the last ``$=$'' we have used \eqref{asinI0} again.
		\end{proof}
		
	Recall the estimate \eqref{AXT}. The result of Proposition \ref{estK3} can be further simplified. To be more precise, it follows from \eqref{AXT} and \eqref{nAna2} that:
	\begin{gather*}
		\sqrt{\det \cnA(X,T;\xi)} = \frac{|X|^3}{8  \, \epsilon_*^3} (1 + o(1)), \\[1mm]
		\bW(X,T;\xi)^\T \cnA(X,T;\xi)^{-1} \bW(X,T;\xi) = \frac{(|\bY|^2 - |\xi|^2)^2}{|X|^2} (1 + o(1)).
	\end{gather*}
	Consequently, if we set
	\begin{align} \label{defsr}
		\mathfrak{r}_*(X,T) := \inf_{\xi \in \blacklozenge_3} \left[ \frac{\bW(X,T;\xi)^\T \cnA(X,T;\xi)^{-1} \bW(X,T;\xi) - \frac{(|\bY|^2 - |\xi|^2)^2}{|X|^2}}{\frac{(|\bY|^2 - |\xi|^2)^2}{|X|^2}} \right], \\[2mm]
		\label{defir}
		\mathfrak{r}^*(X,T) := \sup_{\xi \in \blacklozenge_3} \left[ \frac{\bW(X,T;\xi)^\T \cnA(X,T;\xi)^{-1} \bW(X,T;\xi) - \frac{(|\bY|^2 - |\xi|^2)^2}{|X|^2}}{\frac{(|\bY|^2 - |\xi|^2)^2}{|X|^2}} \right],
	\end{align}
	then $\mathfrak{r}_*(X,T) = o(1)$ and $\mathfrak{r}^*(X,T) = o(1)$. Hence it follows from Proposition \ref{estK3} that
	\[
	\Ft(X,T;\xi) \le (2\pi)^{\frac{3}{2}} \frac{8 \, \epsilon_*^{3}}{|X|^3} \, \vv_2(i\tau^*) \, e^{-\frac{(|\xi|^2 - |\bY|^2)^2}{2 |X|^2} \, (1 + \mathfrak{r}_*(X,T))} \, (1 + o(1)), \quad \forall \, \xi \in \blacklozenge_3.
	\]
	which implies that
	\begin{align*}
		\cK_3 \le  (2\pi)^{\frac{3}{2}} \,   \frac{8 \, \epsilon_*^{3}}{|X|^3} \, \vv_2(i\tau^*) \, \rQ_*(X,T) \, (1 + o(1)),
	\end{align*}
	where
	\begin{align*}
		\rQ_*(X,T) := \int_{\blacklozenge_3} e^{-\frac{1}{2} |\xi - \bY|^2} e^{-\frac{(|\xi|^2 - |\bY|^2)^2}{2 |X|^2} (1 + \mathfrak{r}_*(X,T))} d\xi.
	\end{align*}
	
	Now taking $Y = \bY$, $r = 2 \, |X|^2 /(1 + \mathfrak{r}_*(X,T)) = 2 \, |X|^2 \, (1 + o(1))$, and $\delta = \sqrt{\epsilon_*} \, \rD(X,T)^{\frac{1}{4}}$, in the above lemma, by \eqref{nAna1} we can check that the condition \eqref{asyass} is fulfilled, and hence
	\begin{align*}
		\cK_3 \le 32 \, \pi^3 \,   \frac{  \epsilon_*^{3}}{|X|^2} \, \vv_2(i\tau^*) \, e^{-|\bY|^2} I_0(|\bY|^2) \, (1 + o(1)).
	\end{align*}
	On the other hand, using the same argument
	with $\mathfrak{r}_*$ replaced by $\mathfrak{r}^*$, we can obtain the same lower bound for $\cK_3$. So it is actually an equality. In conclusion, under the assumption in Proposition \ref{lP2} (II), it holds that
	\begin{align}
		\cK_3  &= 32 \pi^3 \,   \frac{  \epsilon_*^{3}}{|X|^2} \, \vv_2(i\tau^*) \, e^{-|\bY|^2} I_0(|\bY|^2) \, (1 + o(1))  \nonumber \\[1mm]
		&\sim \epsilon_* \rD(X,T)^{-2} \frac{1}{\sqrt{1 + \epsilon_* \rD(X,T)^2}} \gsim \ep_* \, \rD(X,T)^{-3},
	\end{align}
	where we have used in ``$\sim$'' \eqref{aI0}, \eqref{eD1} and \eqref{nAna1}.
	
	\subsubsection{Bounds for the remaining terms}

We begin with the estimate of $\cK_2$. In fact, on $\blacklozenge_2$, we have that $|\xi| \sim |\bY|$, $|\bW(X,T;\xi)| \gtrsim \rD(X,T)^{\frac{5}{4}}$ and $|X|^2 + \frac{|\xi|^2}{\epsilon_*} + 1 \sim \rD(X,T)^2$ by the first equality in \eqref{nAna2} and \eqref{nAna1}. Then it follows from Proposition \ref{uppP2} that
\[
|\cK_2| \lesssim \int_{\blacklozenge_2} e^{-  c \, \rD(X,T)^{\frac{1}{4}}} d\xi \sim \epsilon_* \, \rD(X,T)^2\, e^{- c \,\rD(X,T)^{\frac{1}{4}}} = o(\cK_3),
\]
where we have used in ``$\sim$'' \eqref{nAna1} again.

Similarly, we can prove $|\cK_4| = o(\cK_3)$ as well.

 Consider now $\cK_1$.  On $\blacklozenge_1$, notice that $| \, |\xi| - |\bY| \, | \sim |\bY|$,  $|\bW(X,T;\xi)| \sim \rD(X,T)^2$, and $|X|^2 + \frac{|\xi|^2}{\epsilon_*} + 1 \lesssim \rD(X,T)^2$. As a result, Proposition \ref{uppP2} yields
\[
|\cK_1| \lesssim   \int_{\blacklozenge_1 } e^{-  c \, \rD(X,T)} d\xi  = o(\cK_3).
\]

We are now left with the estimation of  $\cK_5$.  Remark that we have on $\blacklozenge_5$ that  $|\bW(X,T,\xi)| \sim \frac{|\xi|^2}{\epsilon_*} \gtrsim \rD(X,T)^2$ and $|X|^2 + \frac{|\xi|^2}{\epsilon_*} + 1 \sim  \frac{|\xi|^2}{\epsilon_*}$. Then Proposition \ref{uppP2} gives
\[
|\cK_5| \lesssim \int_{\blacklozenge_5 } e^{-c  \, \frac{|\xi|}{\sqrt{\epsilon_*}}} d\xi \lesssim \epsilon_* \int_{\{\eta; \, |\eta| \ge c \rD(X,T)\}} e^{- c  \, |\eta|} d\eta \lesssim  \epsilon_* e^{- \frac{c}{2} \, \rD(X,T)} =  o(\cK_3),
\]
where we have used in the second ``$\lesssim$'' the change of variables $\xi = \sqrt{\epsilon_*} \, \eta$.

\subsubsection{Summary}
From all the estimates obtained above, we conclude that
\begin{align*}
	\rP(X,T) &=  \frac{1}{2\pi} \left( 1 - \frac{|\tau^*|^2}{\vartheta_1^2} \right)^{-1} e^{-\frac{\rD(X,T)^2}{4}} \, \sum_{i = 1}^5 \cK_i \\
	&= \frac{1}{2\pi} \left( 1 - \frac{|\tau^*|^2}{\vartheta_1^2} \right)^{-1} e^{-\frac{\rD(X,T)^2}{4}} \, \cK_3 \, (1 + o(1)) \\
	&= 16 \pi^2 \,   \frac{  \epsilon_*^{3}}{|X|^2} \, \vv(i\tau^*) \, e^{-\frac{\rD(X,T)^2}{4}} \, e^{-|\bY|^2} I_0(|\bY|^2) \, (1 + o(1)).
\end{align*}
Then to show \eqref{nax2}, it suffices to use
\eqref{relvU}, \eqref{aDp3} and a similar argument as in the proof of \eqref{Ds1}.

\end{appendices}

\section*{Acknowledgement}
\setcounter{equation}{0}
This work is partially supported by
NSF of China (Grants No. 12271102 and No. 11625102).  The first author would like to thank D. Bakry for bringing the free step-two Carnot group with $3$ generators to our attention in 2008.

\phantomsection\addcontentsline{toc}{section}{References}
\bibliographystyle{abbrv}
\bibliography{LMZnhk}

\mbox{}\\
Hong-Quan Li, Sheng-Chen Mao \\
School of Mathematical Sciences/Shanghai Center for Mathematical Sciences  \\
Fudan University \\
220 Handan Road  \\
Shanghai 200433  \\
People's Republic of China \\
E-Mail: hongquan\_li@fudan.edu.cn  \\
scmao22@m.fudan.edu.cn \quad or \quad maosci@163.com \\

\mbox{}\\
Ye Zhang\\
Analysis on Metric Spaces Unit  \\
Okinawa Institute of Science and Technology Graduate University \\
1919-1 Tancha, Onna-son, Kunigami-gun \\
Okinawa, 904-0495, Japan \\
E-Mail: zhangye0217@gmail.com \quad or \quad Ye.Zhang2@oist.jp \mbox{}\\

\end{document}